\documentclass{amsart}

\usepackage{hyperref}
\hypersetup{hypertex=true,
			breaklinks=true, 
            colorlinks=true,
            linkcolor=blue,
            anchorcolor=blue,
            citecolor=blue}
\usepackage[utf8]{inputenc}
\usepackage{cite}
\usepackage{amssymb}
\usepackage{amsthm}
\usepackage{amsfonts}
\usepackage{amsmath}
\usepackage{mathrsfs}
\usepackage[all]{xy}
\usepackage{graphicx}
\usepackage{mathrsfs}
\usepackage{extpfeil}
\usepackage{mathtools}
\usepackage{tikz-cd}
\usepackage{bm}
\usepackage{latexsym, amscd ,psfrag}

\usepackage[mathscr]{euscript}
\usepackage{enumitem}
\usepackage{cleveref}

\newcommand{\ep}{\epsilon}
\newcommand{\si}{\sigma}
\newcommand{\Si}{\Sigma}
\newcommand{\Ga}{\Gamma}

\newcommand{\bC}{\mathbb{C}}

\newcommand{\bK}{\mathbb{K}}
\newcommand{\bL}{\mathbb{L}}
\newcommand{\bP}{\mathbb{P}}
\newcommand{\bQ}{\mathbb{Q}}
\newcommand{\bR}{\mathbb{R}}
\newcommand{\bS}{\mathbb{S}}
\newcommand{\bT}{\mathbb{T}}
\newcommand{\bZ}{\mathbb{Z}}
\newcommand{\btau}{\boldsymbol{\tau}}
\newcommand{\bsi}{{\boldsymbol{\si}}}

\newcommand{\cB}{\mathcal{B}}

\newcommand{\cE}{\mathcal{E}}
\newcommand{\cK}{\mathcal{K}}
\newcommand{\cL}{\mathcal{L}}
\newcommand{\cI}{\mathcal{I}}
\newcommand{\cM}{\mathcal{M}}
\newcommand{\cO}{\mathcal{O}}

\newcommand{\cS}{\mathcal{S}}

\newcommand{\cT}{\mathcal{T}}
\newcommand{\cU}{\mathcal{U}}

\newcommand{\cX}{\mathcal{X}}

\newcommand{\fp}{\mathfrak{p}}

\newcommand{\fQ}{\mathfrak{Q}}
\newcommand{\fX}{\mathfrak{X}}

\newcommand{\End}{\mathrm{End}}
\newcommand{\Spec}{\mathrm{Spec}}

\newcommand{\age}{\mathrm{age}}

\newcommand{{\inv} }{\mathrm{inv}}
\newcommand{\ev}{\mathrm{ev}}
\newcommand{\Aut}{\mathrm{Aut}}

\newcommand{\Res}{\mathrm{Res}}

\newcommand{\val}{ {\mathrm{val}} }
\newcommand{\vir}{{\mathrm{vir}}}
\newcommand{\CR}{  {\mathrm{CR}}  }

\newcommand{\pt}{ {\mathrm{pt}}}
\newcommand{\Pic}{ {\mathrm{Pic}}}
\newcommand{\NE}{ {\mathrm{NE}}}

\newcommand{\one}{\mathbf{1}}

\newcommand{\bu}{\mathbf{u}}

\newcommand{\bp}{\mathbf{p}}

\newcommand{\bGa}{\mathbf{\Ga}}

\newcommand{\su}{\mathsf{u}}
\newcommand{\sv}{\mathsf{v}}
\newcommand{\sw}{\mathsf{w}}

\newcommand{\tfQ}{ \widetilde{\mathfrak{Q}}}

\newcommand{\txi}{ {\widetilde{\xi}} }

\newcommand{\tS}{\widetilde{S}}

\newcommand{\that}{\hat{t}}

\newcommand{\vGa}{\vec{\Ga}}

\newcommand{\Mbar}{\overline{\cM}}

\newcommand{\spa}{ {\ \ \,} }

\newcommand{\ST}{ {S_{\bT}} }
\newcommand{\RT}{ {R_{\bT}} }
\newcommand{\bST}{ {\bar{S}_{\bT}} }
\newcommand{\bSTQ}{ \bST[\![ \tfQ,\tau'' ]\!] }

\newcommand{\nov}{\Lambda_{\mathrm{nov}}}
\newcommand{\novT}{\bar{\Lambda}^{\bT}_{\mathrm{nov}} }

\newcommand{\ThreeSphere}{\mathbb{S}^3}
\newcommand{\Zp}{\mathbb{Z}_p}
\newcommand{\Knot}{\mathcal{K}}
\newcommand{\torus}{K_{r,s}}
\newcommand{\T}{T_{r,s}}
\newcommand{\C}{\mathbb{C}}
\newcommand{\ResolvedConi}{[\mathcal{O}(-1)\oplus\mathcal{O}(-1)\rightarrow\mathbb{P}^1]}
\newcommand{\ReducedConi}{[(\mathcal{O}_{\mathbb{P}^1}(-1)\oplus\mathcal{O}_{\mathbb{P}^1}(-1))/\Zp]}
\newcommand{\X}{\mathcal{X}}
\newcommand{\Lrs}{L_{r,s}}

\newcommand{\AlphaMu}{\langle\frac{\mu r}{p}\rangle}

\newcommand\HcrX[1]{H^{#1}_{\text{CR},\T}(\X)}
\newcommand\GWPol[1]{F^{\X,\Lrs}_{0,1}(\btau_#1; X)}
\newcommand\GWPgn[3]{F^{\X,\Lrs}_{#2,#3}(\btau_{#1}; X_1,\dots,X_{#3})}
\newcommand\GWgl[1]{\langle #1 \rangle_{g,l,d,\vec{\mu}}^{\X,\Lrs,\T}}
\newcommand\GWol[1]{\langle #1 \rangle_{0,l,d,\mu}^{\X,\Lrs,\T}}
\newcommand\DescdentGwol[2]{\langle #1 \rangle_{0,#2,d}^{\X,\T}}
\newcommand\eff{\text{eff}}

\newcommand{\wbeta}{w(\beta)}
\newcommand\Lens[1]{L(p,#1)}
\newcommand\Bigbig{\text{big}}
\newcommand\Jbig{J^\Bigbig_\bT(z)}
\newcommand\Mgldmu{\overline{\cM}_{g,l,d,\vec{\mu}}}

\newcommand\fh{\mathfrak{h}}
\newcommand\BoldEta{{\bm{\eta}}}
\newcommand\SphereTorus{{\tilde{K}_{r,s}}}

\newtheorem{lma}{Lemma}[section]
\newtheorem{coro}[lma]{Corollary}
\newtheorem{defn}[lma]{Definition}
\newtheorem{prop}[lma]{Proposition}

\newtheorem{theorem}[lma]{Theorem}
\newtheorem{remark}[lma]{Remark}

\newtheorem{dummy}{dummy}[section]
\theoremstyle{definition}
\newtheorem{definition}[dummy]{Definition}

\begin{document}
\title{Torus knots in Lens spaces, open Gromov-Witten invariants, and topological recursion}

\author{Jinghao Yu}
\address{Jinghao Yu, Department of Mathematical Sciences, Tsinghua University, Haidian District, Beijing 100084, China}
\email{yjh21@mails.tsinghua.edu.cn}

\author{Zhengyu Zong}
\address{Zhengyu Zong, Department of Mathematical Sciences, Tsinghua University, He er Building, Room 103,
Tsinghua University, Haidian District, Beijing 100084, China}
\email{zyzong@mail.tsinghua.edu.cn}

\begin{abstract}
    Starting from a torus knot $\cK$ in the lens space $L(p,-1)$, we construct a Lagrangian sub-manifold $L_{\cK}$ in $\cX=\big(\cO_{\bP^1}(-1)\oplus \cO_{\bP^1}(-1)\big)/\bZ_p$ under the conifold transition. We prove a mirror theorem which relates the all genus open-closed Gromov-Witten invariants of $(\cX,L_{\cK})$ to the topological recursion on the B-model
    spectral curve. This verifies a conjecture in \cite{Bor-Bri} in the case of lens space.
    \end{abstract}

    \maketitle

    \tableofcontents

    \section{Introduction}\label{introduction}

    \subsection{Historical background and motivation}
    In \cite{Wi89}, Witten noticed that there is a deep relation between the HOMFLY polynomials of knots and the Wilson loop observables of Chern-Simons quantum field theory for $G=U(N)$ gauge theories. In the case of the three-sphere $\bS^3$, this relation can be described as follows.

    Let $A$ be a connection on $\bS^3$ for the gauge group $G=U(N)$ and $R$
    be a representation of $G$. The Chern-Simons
    action functional is defined as
    \[
    S(A)=\frac{k}{4\pi} \int_{\bS^3} \mathrm{Tr}_R\left( A\wedge d A +\frac{2}{3}
      A\wedge A \wedge A\right),
    \]
    where $k$ is the coupling constant. The partition function of this theory is defined by path-integrals in physics
    \[
    Z(\bS^3)={\int \mathcal D Ae^{\sqrt{-1}S(A)}}.
    \]
    Let $K\cong S^1\hookrightarrow \bS^3$ be a framed, oriented knot. In physics, the normalized vacuum expectation value is also defined by path-integrals
    \[
    W_R(K)=\frac{1}{Z(\bS^3)}\int \mathcal D A e^{\sqrt{-1}S(A)}\mathrm{Tr}_RU_K,
    \]
    where $U_K$ is the holonomy around $K$. When $R$ is the
    fundamental representation of $G=U(N)$, $W_R(K)$ is related to the
    HOMFLY polynomial $P_K(q,\lambda)$ of $K$, which is mathematically well-defined, as below
    \[
    W_R(K)=\lambda
    \frac{\lambda^{\frac{1}{2}}-\lambda^{-\frac{1}{2}}}{q^{\frac{1}{2}}-q^{-\frac{1}{2}}}P_K(q,\lambda),
    \]
    where
    \[
    q=e^{\frac{2\pi\sqrt{-1}}{k+N}},\quad \lambda=q^N.
    \]

    Later in \cite{Wi95}, Witten noted that the Chern-Simons theory, in turn, describes the target space physics of A-model topological strings. In \cite{GV, OV}, the gauge theory invariant $W_R(K)$ was conjectured to be related to the open Gromov-Witten theory
    of the resolved conifold $X=\cO_{\bP^1}(-1)\oplus \cO_{\bP^1}(-1)$ with Lagrangian boundary condition $L_K$ under the large-$N$ duality and the conifold transition. Here $L_K$ is a Lagrangian sub-manifold of $\cO_{\bP^1}(-1)\oplus \cO_{\bP^1}(-1)$ constructed from the knot $K$ under the conifold transition \cite{DSV, Ta01, Ko07, LMV}. When $K$ is an unknot, the conjecture can be precisely formulated as the famous
    Mari\~no-Vafa formula \cite{MaVa} concerning cubic Hodge integrals and is later proved in \cite{LLZ, OkPa}. In this case, $L_K$ belongs to a class of Lagrangians called Aganagic-Vafa branes \cite{AgVa,AKV} (Harvey-Lawson type Lagrangian). The open Gromov-Witten invariants for $(X,L_K)$ in this situation are defined in \cite{KaLiu, Liu02}.

    For a general knot $K$, open Gromov-Witten theory is usually difficult to define for $(X,L_K)$. When $K$ is a torus knot, i.e. a knot that can be realized on a real torus in $S^3$, one can still use the localization technique to define such open Gromov-Witten invariants \cite{DSV}. In \cite{DSV}, Diaconescu-Shende-Vafa also prove the conjecture on the correspondence between HOMFLY polynomials and open Gromov-Witten invariants of winding number one.

    One obtains a richer picture when mirror symmetry comes in the story. In \cite{BKMP09, BKMP10}, an all genus open-closed mirror symmetry is proposed for general toric Calabi-Yau 3-folds/3-orbifolds, known as the Remodeling Conjecture. The Lagrangian sub-manifolds on A-model are the Aganagic-Vafa branes and the higher genus B-model is obtained from
    Eynard-Orantin's topological recursion \cite{EO07} on the mirror curve. The Remodeling Conjecture for smooth toric Calabi-Yau 3-folds is proved in \cite{EO15} and for general toric Calabi-Yau 3-orbifolds is proved in \cite{FLZ16}. The resolved conifold $X=\cO_{\bP^1}(-1)\oplus \cO_{\bP^1}(-1)$ is a toric Calabi-Yau 3-fold and one immediately obtains the mirror symmetry for $(X,L_K)$ with $K$ the trivial knot. In \cite{BEM}, Brini-Eynard-Mari\~no propose a modified spectral curve for a torus knot $K_{r,s}$ with coprime $(r,s)$. They conjecture that this curve should be the correct B-model mirror to $(X,L_{K_{r,s}})$. This conjecture is proved in \cite{FZ19}.

    The bridge connecting the above two dualities (large $N$ duality and mirror symmetry) is the matrix model theory. In \cite{BEO}, the colored HOMFLY polynomials of torus knots $K$ are expressed as a matrix model. The loop equation of this matrix model gives us the topological recursion on the mirror curve of $(X,L_{K})$. Therefore the result in \cite{BEO} implies that the Eynard-Orantin invariants of the mirror curve are equivalent to the colored HOMFLY polynomial of the knot $K$. Combining the result in \cite{BEO}, the mirror symmetry result in \cite{FZ19} implies that the following three invariants are equivalent for any torus knot $K$ (see Figure \ref{fig:relation-i})
    \begin{enumerate}
    \item the open-closed Gromov-Witten invariants of $(X,L_K)$;\\
    \item the Eynard-Orantin invariants of the mirror curve;\\
    \item the colored HOMFLY polynomial of the knot $K$.
    \end{enumerate}

    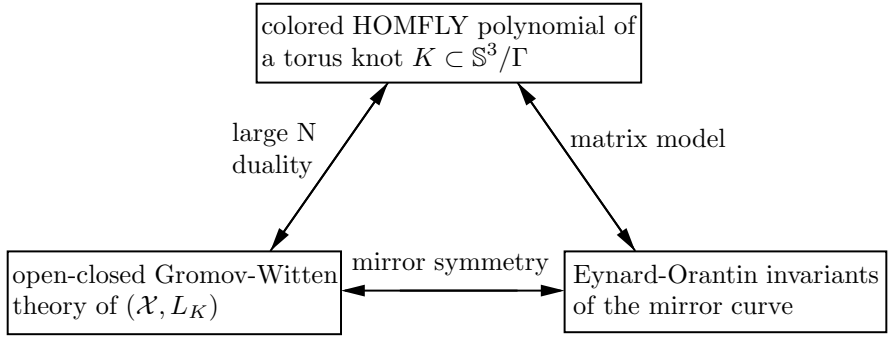
\begin{figure}[h]
       \centering

       \tikzset{every picture/.style={line width=0.75pt}} 

       \begin{tikzpicture}[x=0.75pt,y=0.75pt,yscale=-1,xscale=1]

       \draw    (202.67,200) -- (283.76,200) ;
       \draw [shift={(285.76,200)}, rotate = 180] [fill={rgb, 255:red, 0; green, 0; blue, 0 }  ][line width=0.08]  [draw opacity=0] (12,-3) -- (0,0) -- (12,3) -- cycle    ;
       \draw    (262.02,200) -- (174.99,200) ;
       \draw [shift={(172.99,200)}, rotate = 360] [fill={rgb, 255:red, 0; green, 0; blue, 0 }  ][line width=0.08]  [draw opacity=0] (12,-3) -- (0,0) -- (12,3) -- cycle    ;
       \draw    (321.37,180) -- (263.17,96.64) ;
       \draw [shift={(262.02,95)}, rotate = 55.08] [fill={rgb, 255:red, 0; green, 0; blue, 0 }  ][line width=0.08]  [draw opacity=0] (12,-3) -- (0,0) -- (12,3) -- cycle    ;
       \draw    (262.02,95) -- (320.23,178.36) ;
       \draw [shift={(321.37,180)}, rotate = 235.08] [fill={rgb, 255:red, 0; green, 0; blue, 0 }  ][line width=0.08]  [draw opacity=0] (12,-3) -- (0,0) -- (12,3) -- cycle    ;
       \draw    (196.73,95) -- (138.53,178.36) ;
       \draw [shift={(137.38,180)}, rotate = 304.92] [fill={rgb, 255:red, 0; green, 0; blue, 0 }  ][line width=0.08]  [draw opacity=0] (12,-3) -- (0,0) -- (12,3) -- cycle    ;
       \draw    (137.38,180) -- (195.59,96.64) ;
       \draw [shift={(196.73,95)}, rotate = 124.92] [fill={rgb, 255:red, 0; green, 0; blue, 0 }  ][line width=0.08]  [draw opacity=0] (12,-3) -- (0,0) -- (12,3) -- cycle    ;
       \draw   (130.43,55.21) -- (324.37,55.21) -- (324.37,95) -- (130.43,95) -- cycle ;
       \draw   (4.88,180.21) -- (172.99,180.21) -- (172.99,222) -- (4.88,222) -- cycle ;
       \draw   (285.76,180.21) -- (448.02,180.21) -- (448.02,221) -- (285.76,221) -- cycle ;

       \draw (131.41,60) node [anchor=north west][inner sep=0.75pt]   [align=left] {colored HOMFLY polynomial of \\a torus knot $\displaystyle K\subset \ThreeSphere$};
       \draw (5.86,185.21) node [anchor=north west][inner sep=0.75pt]   [align=left] {open-closed Gromov-Witten\\theory of ($X, L_K$)};
       \draw (288.54,185.21) node [anchor=north west][inner sep=0.75pt]   [align=left] {Eynard-Orantin invariants \\of the mirror curve};
       \draw (116.54,114) node [anchor=north west][inner sep=0.75pt]   [align=left] {large N\\ duality};
       \draw (287.01,117) node [anchor=north west][inner sep=0.75pt]   [align=left] {matrix model};
       \draw (176.95,179) node [anchor=north west][inner sep=0.75pt]   [align=left] {mirror symmetry};

       \end{tikzpicture}\renewcommand{\thefigure}{1.A}
       \caption{Interrelations of mentioned topics}\label{fig:relation-i}
   \end{figure}
    In particular, the equivalence of item (1) and item (3) is the large $N$ duality for \emph{colored} HOMFLY polynomial and open-closed Gromov-Witten invariants of $(X,L_K)$ with \emph{arbitrary} winding numbers. This equivalence generalizes the result in \cite{DSV}.

    In \cite{Bor-Bri}, the dualities in Figure \ref{fig:relation-i} is conjecturally generalized to torus knots $K$ in spherical Seifert 3-manifolds, which are quotients $\bS^3/\Gamma$ of the three-sphere by the free action of a finite isometry group. Under the conifold transition, we obtain a Lagrangian sub-manifold $L_K$ in $\cX=\big(\cO_{\bP^1}(-1)\oplus \cO_{\bP^1}(-1)\big)/\Gamma$. In this case, the following three invariants are conjecturally equivalent (see Figure \ref{fig:relation-ii})
    \begin{enumerate}
    \item the open-closed Gromov-Witten invariants of $(\cX,L_K)$;\\
    \item the Eynard-Orantin invariants of the mirror curve of $(\cX,L_K)$;\\
    \item the colored HOMFLY polynomial of the knot $K$.
    \end{enumerate}

    \begin{figure}[h]
       \centering

       \tikzset{every picture/.style={line width=0.75pt}} 

       \begin{tikzpicture}[x=0.75pt,y=0.75pt,yscale=-1,xscale=1]

              \draw    (202.67,200) -- (283.76,200) ;
              \draw [shift={(285.76,200)}, rotate = 180] [fill={rgb, 255:red, 0; green, 0; blue, 0 }  ][line width=0.08]  [draw opacity=0] (12,-3) -- (0,0) -- (12,3) -- cycle    ;
              \draw    (262.02,200) -- (174.99,200) ;
              \draw [shift={(172.99,200)}, rotate = 360] [fill={rgb, 255:red, 0; green, 0; blue, 0 }  ][line width=0.08]  [draw opacity=0] (12,-3) -- (0,0) -- (12,3) -- cycle    ;
              \draw    (321.37,180) -- (263.17,96.64) ;
              \draw [shift={(262.02,95)}, rotate = 55.08] [fill={rgb, 255:red, 0; green, 0; blue, 0 }  ][line width=0.08]  [draw opacity=0] (12,-3) -- (0,0) -- (12,3) -- cycle    ;
              \draw    (262.02,95) -- (320.23,178.36) ;
              \draw [shift={(321.37,180)}, rotate = 235.08] [fill={rgb, 255:red, 0; green, 0; blue, 0 }  ][line width=0.08]  [draw opacity=0] (12,-3) -- (0,0) -- (12,3) -- cycle    ;
              \draw    (196.73,95) -- (138.53,178.36) ;
              \draw [shift={(137.38,180)}, rotate = 304.92] [fill={rgb, 255:red, 0; green, 0; blue, 0 }  ][line width=0.08]  [draw opacity=0] (12,-3) -- (0,0) -- (12,3) -- cycle    ;
              \draw    (137.38,180) -- (195.59,96.64) ;
              \draw [shift={(196.73,95)}, rotate = 124.92] [fill={rgb, 255:red, 0; green, 0; blue, 0 }  ][line width=0.08]  [draw opacity=0] (12,-3) -- (0,0) -- (12,3) -- cycle    ;
       \draw   (130.43,55.21) -- (324.37,55.21) -- (324.37,95) -- (130.43,95) -- cycle ;
       \draw   (4.88,180.21) -- (172.99,180.21) -- (172.99,222) -- (4.88,222) -- cycle ;
       \draw   (285.76,180.21) -- (448.02,180.21) -- (448.02,221) -- (285.76,221) -- cycle ;

       \draw (131.41,60) node [anchor=north west][inner sep=0.75pt]   [align=left] {colored HOMFLY polynomial of \\a torus knot $\displaystyle K\subset \ThreeSphere/\Gamma$};
       \draw (5.86,185.21) node [anchor=north west][inner sep=0.75pt]   [align=left] {open-closed Gromov-Witten\\theory of ($\X, L_K$)};
       \draw (288.54,185.21) node [anchor=north west][inner sep=0.75pt]   [align=left] {Eynard-Orantin invariants \\of the mirror curve};
       \draw (116.54,114) node [anchor=north west][inner sep=0.75pt]   [align=left] {large N\\ duality};
       \draw (287.01,117) node [anchor=north west][inner sep=0.75pt]   [align=left] {matrix model};
       \draw (176.95,179) node [anchor=north west][inner sep=0.75pt]   [align=left] {mirror symmetry};

       \end{tikzpicture}\renewcommand{\thefigure}{1.B}
       \caption{Interrelations of mentioned topics}\label{fig:relation-ii}
   \end{figure}
    The aim of this paper is to prove the equivalence of item (1) and item (2) in the case when $\bS^3/\Gamma$ is the lens space $L(p,-1)$. In this case, $\Gamma=\bZ_p$ and $\cX=\big(\cO_{\bP^1}(-1)\oplus \cO_{\bP^1}(-1)\big)/\bZ_p$ is a toric Calabi-Yau 3-orbifold of type A singularity. We will prove that the all genus open-closed Gromov-Witten potential of $(\cX,L_K)$ is equal to the Eynard-Orantin invariants of the mirror curve under the mirror map. This is a generalization of the result in \cite{FZ19}.

    \subsection{Statement of the main result}
    We start from an $(r,s)$ torus knot $\cK$ in the lens space $L(p,-1)$. The lens space $L(p,q)$ is defined as
    \[
        \{(x,y)\in\mathbb
        C^2~|~|x|^2+|y|^2=1\}/(x,y)\sim (e^{2\pi i/p} x,e^{2\pi iq/p} y)
    \]
    We assume that $p,r,s$ are pairwisely coprime and there exists an integer $k$ such that $r+s=pk$. See Section \ref{sec:conifold transition} for a discussion on the above assumption. Let $N_\Knot^*$ be the total space of the conormal
    bundle of $\Knot$ in $\Lens{-1}$ and let $T^*\Lens{-1}$ be the total space of the cotangent bundle of $\Lens{-1}$. Then $N_\Knot^*$ is a Lagrangian submanifold
    of $T^*\Lens{-1}$ and the intersection of $N_\Knot^*$ with the zero section of $T^*\Lens{-1}$
    is the knot $\Knot$. The conifold transition describes a procedure which consists of two steps. In the first step, the zero section of $T^*\Lens{-1}$ shrinks to a point, making the limit of $T^*\Lens{-1}$ a singular variety. In the second step, one applies a small resolution to this singular variety to obtain $\cX=\big(\cO_{\bP^1}(-1)\oplus \cO_{\bP^1}(-1)\big)/\bZ_p$ which is toric Calabi-Yau 3-orbifold. Under the conifold transition, the Lagrangian submanifold $N_\Knot^*$ transforms to a Lagrangian submanifold $L_{r,s}$ in $\cX$. This step is not so straightforward since we need to first shift $N_\Knot^*$ off the zero section before applying the conifold transition. See Section \ref{sec:conifold transition} for the construction of $L_{r,s}$ under the conifold transition.

    There is an $S^1$-action on $\cX$ which preserves the Lagrangian submanifold $L_{r,s}$. Therefore, we can define the open Gromov-Witten invariants of $(\cX,L_{r,s})$ by localization (see Section \ref{sec:open-closed-GW}). Alternatively, we can also define the open Gromov-Witten invariants of $(\cX,L_{r,s})$ by using relative Gromov-Witten theory (see Section \ref{sec:relative1} and Section \ref{sec:relative2}). Let $\GWPgn{\empty}{g}{n}$ be the generating function of genus $g$, $n$ boundary open Gromov-Witten invariants of $(\cX,L_{r,s})$ (see equation \eqref{eqn:open-potential}).

    On B-model side, the mirror curve $C_q\subset(\bC^*)^2$ of $(\cX,L_{r,s})$ is defined by
    $$
    U^p + V + 1 + q_1U^pV^{-1} + \sum_{m=1}^{p-1}q_{m+1}U^m = 0.
    $$
    where $q_1,\cdots,q_p$ are parameter of complex structures on B-model. Because $r, s, p$ are pairwisely coprime, we know $r, k$ ($pk = r+s$) are coprime and we can find $\gamma, \delta\in \bZ$ such that
    \begin{equation*}
        \left(
            \begin{array}{cc}
                r & -k
                \\
                \gamma & \delta
            \end{array}
        \right)\in SL(2;\bZ)
    \end{equation*}
    By an $SL(2,\bZ)$ change of coordinates
    \[
        X= e^{-x} = U^rV^{-k}, \quad Y= e^{-y} = U^\gamma V^\delta
    \]
    as in \eqref{eqn:coord2}, we obtain new coordinate functions $X,Y$ on $C_q$. Then we apply Eynard-Orantin's topological recursion to $(C_q,X,Y)$. Eynard-Orantin's topological recursion is a recursive algorithm that produces all genus open invariants of a spectral curve
    (see Section \ref{sec:EO-recursion}). From the recursion, we get a symmetric
    meromorphic $n$-form $\omega_{g,n}$ on the product of compactified mirror curves $(\overline C_q)^n$. Let $\eta = X^{\frac{1}{r}}$. Then $\eta$ is a local coordinate of the compactified mirror curve $\overline{C}_q$ around $\mathfrak{s}_0=(X, V) = (0,-1)$. One can integrate the expansion of
    $\omega_{g,n}$ in $\eta_1,\dots,\eta_n$ around this point and define
    \[
    W_{g,n}(\eta_1,\dots,\eta_n,q)=\int_0^{\eta_1}\dots\int_0^{\eta_n} \omega_{g,n}
    \]
    for $2g-2+n>0$. The definitions of $W_{0,1}$ and $W_{0,2}$ (disk and annulus potentials) have slightly different forms (see Section \ref{sec:EO-recursion}). The following theorem is the main result of this paper (see Theorem \ref{thm:disk-mirror-theorem}, Theorem \ref{thm:main-theorem-annulus}, Theorem \ref{thm:main-theorem-stable}).
    \begin{theorem}
    Under the mirror map $\btau=\btau(q)$ and $X_k=\eta_k^r$, the power series expansion of the open Gromov-Witten amplitude $F_{g,n}(\btau;X_1,\dots,X_n)$ in $X_1,\dots,X_n$ is the part in the power
    series expansion of $(-1)^{g-1} r^n
    W_{g,n}(q,\eta_1,\dots,\eta_n)$ whose degrees of
    each $\eta_k$ are divisible by $r$.
    \end{theorem}

    \subsection{Future work}

    \subsubsection{Large $N$ duality for $(\cX,L_{r,s})$}
    An interesting future work is to prove the large $N$ duality for $(\cX,L_{r,s})$. Since the result in this paper proves the mirror symmetry for $(\cX,L_{r,s})$ which is the bottom arrow in Figure \ref{fig:relation-ii}, one may complete the equivalences in Figure \ref{fig:relation-ii} by relating the colored HOMFLY polynomial of the $(r,s)$-knot $\cK$ to the Eynard-Orantin invariants of the mirror curve for $(\cX,L_{r,s})$. One way to do this is to rewrite the colored HOMFLY polynomial of $\cK$ in terms of matrix integrals as in \cite{BEO} and then use the loop equation to obtain the mirror curve and the topological recursion.

    \subsubsection{The non-abelian case}
    Another future direction is to consider the case when $\Gamma$ is of type $D$ and type $E$. In this case, the corresponding $\cX$ is a non-toric orbifold and it will be more challenging to prove the mirror symmetry for $(\cX,L_K)$. One can also try to prove the other two arrows in Figure \ref{fig:relation-ii}.

    \subsection{Overview of the paper}
    In Section \ref{sec:conifold transition}, we will construct the Lagrangian submanifold $L_{r,s}$ in $\cX$ from an $(r,s)$ torus knot $\cK$ in the lens space $L(p,-1)$ under the conifold transition. In Section \ref{sec:A}, we first review the closed descendent Gromov-Witten invariants and the corresponding graph sum formula. Then we define open Gromov-Witten invariants with respect to $(\cX,L_{r,s})$ in two different ways: by localization in the moludi space of maps from bordered Riemann surfaces (Section \ref{sec:open-closed-GW}) and by using relative Gromov-Witten theory (Section \ref{sec:relative1} and Section \ref{sec:relative2}). We also express the generating functions for open Gromov-Witten invariants in graph sums. In Section \ref{sec:B}, we construct the mirror curve for $(\cX,L_{r,s})$ and define the higher genus B-model by applying Eynard-Orantin's topological recursion to the mirror curve. Then we express the Eynard-Orantin invariants in terms of graph sums. Finally, in Section \ref{sec:ms} we prove the all genus open-closed mirror symmetry between $(\cX,L_{r,s})$ and $C_q$, based on matching graph sum ingredients, which include the localization computation on disk invariants and genus $0$ mirror theorem.

    \subsection*{Acknowledgement}
    The authors wish to thank Chiu-Chu Melissa Liu and Bohan Fang for useful discussions. The work of the second author is partially supported by NSFC grant No. 11701315.

\section{Torus knots in lens space and conifold transition}
This section gives a geometric realization of Gopakumar-Vafa large N duality for the
case of lens space $L(p,-1)$, equipped with an algebraic knot reduced from a
torus knot in $\ThreeSphere$.

\subsection{Torus knots in lens space and Lagrangians in the deformed conifold}
\label{KnotsInLensSpace}
The Lens space $L(p,q) = \ThreeSphere/\Zp$ is
\[
    \{(x,y)\in\mathbb
    C^2~|~|x|^2+|y|^2=1\}/(x,y)\sim (e^{2\pi i/p} x,e^{2\pi iq/p} y)
\]
A polynomial $f(x,y)$ is called $(p,q)$-invariant if for $\zeta = e^{2\pi i/p}$, there exists $l\in\{0,1,\dots, p-1\}$ such that
\[
    f(\zeta \cdot(x,y))= f(\zeta x,\zeta^qy)= \zeta ^lf(x,y).
\]
If $f$ is $(p,q)$-invariant, then the algebraic knot $\Knot = \{f(x,y)=0\}\subset \ThreeSphere$ can be reduced
to $L(p,q)$.
In particular, we consider the torus $(r,s)$-knot on $\SphereTorus \in\ThreeSphere$ ($r,s$ are coprime)
\[
  f(x,y) = x^r-y^s, \quad \SphereTorus = \{f(x,y)=0\}.
\]
Since
\[
    f(\zeta x,\zeta^{-1}y)= \zeta ^rx^r - \zeta ^{-s}y^s
    = \zeta^{-s}(\zeta^{r+s}x^r-y^s),
\]
$\SphereTorus$ can be reduced to $L(p,-1)$
if and only if $r+s = pk$ for some $k\in\mathbb{N}^*$.
For simplicity, we call the reduced torus knot in $\Lens{-1}$ by $\textit{torus knot in lens space}$,
denoted by $\torus$.
\begin{remark}
    In \cite{OS18,Stevan13}, a torus knot $K_{(a,b)}$ in $L(p,q)$ is described as an (a,b)-torus knot lying on the Heegaard torus $\partial V_2$.
As shown in \cite{PP21}, the following map gives the genus 1 Heegaard torus of $\Lens{-1}$.
\begin{equation*}
    \begin{aligned}
    h: \frac{\{|x|^2+|y|^2=1\} \cap \{|y|^2\geq \frac{1}{2}\}}{(x,y)\sim (e^{\frac{2\pi i}{p}} x,e^{-\frac{2\pi i}{p}} y)}
    &\longrightarrow  S^1\times D^2 =:V_2
    \\
    [x,y] &\longmapsto (y^p/|y|^p, \sqrt{2}xy/|y|)
    \end{aligned}
\end{equation*}
Let $f(x,y) = x^r-y^s$ with $r+s=pk$,
\begin{equation*}
    \begin{aligned}
        h|_{\torus}: \{ |y|\geq 1/2\}\cap \{f=0\} &\longrightarrow S^1\times D^2
        \\
        [\frac{1}{\sqrt{2}}e^{\frac{2\pi its}{p}}, \frac{1}{\sqrt{2}}e^{\frac{2\pi itr}{p}}]&\longmapsto (e^{2\pi itr}, e^{2\pi itk}), t\in [0,1]
    \end{aligned}
\end{equation*}
We see that our knot $\torus$ is the same as the knot $K_{(k,r)}$ in \cite{OS18}.
\end{remark}
\par
The deform conifold $Y_\mu, \mu>0$ is defined as
\[
    Y_\mu = \{(x,y,z,w)\in\mathbb{C}^4~|~xz-yw=\mu\}.
\]
It is a symplectic manifold by inheriting the standard symplectic form on $\bC^4$:
\begin{equation*}
    \begin{aligned}
        \omega_{\bC^4} &= \frac{\sqrt{-1}}{2}(dx\wedge d\bar{x}+ dy\wedge d\bar{y}+
        dz\wedge d\bar{z}+ dw\wedge d\bar{w})
        \\
        \omega_{Y_\mu} &= \omega_{\bC^4}|_{Y_\mu}.
    \end{aligned}
\end{equation*}
When $x=\bar{z}, y = -\bar{w}$, the equation $xz-yw=\mu$ becomes $|x|^2+|y|^2=\mu$. It means $Y_\mu$ contains a 3-sphere $\ThreeSphere_\mu$ of radius $\sqrt{\mu}$.
In fact, $\ThreeSphere_\mu$ is the fixed locus of the anti-holomorphic involution for $Y_\mu$:
\begin{equation}
    \label{involution}
    \begin{aligned}
        I: \bC^4&\rightarrow\bC^4 \\
        (x,y,z,w)&\mapsto(\bar{z}, -\bar{w}, \bar{x}, -\bar{y})
    \end{aligned}
\end{equation}
There is a $\Zp$-action on $\C^4$:
\begin{equation}\label{Zp-action}
    \begin{aligned}
        \Zp \times \C^4 &\rightarrow \C^4
        \\
        (\zeta , (x,y,z,w)) &\mapsto (\zeta x,\zeta ^{-1}y, \zeta^{-1}z,\zeta w)
    \end{aligned}
\end{equation}
It is easy to check this $\Zp$-action preserves $Y_\mu$ and $\ThreeSphere_\mu$.
According to \cite[Lemma 1]{Bri09}, the orbit space of (\ref{Zp-action}) restricted
to the deformed conifold $Y_\mu$ is diffeomorphic to $T^*\Lens{-1}$. Moreover, for any $\mu>0$ there is a
symplectomorphism:
\[
    \varphi_\mu: Y_\mu/\Zp \rightarrow T^*\Lens{-1}
\]
and $\varphi_\mu(\ThreeSphere_\mu/\Zp)$ is the zero section of $T^*\Lens{-1}$. Consider
an algebraic knot $\Knot\in\Lens{-1}$. Let $N_\Knot^*$ be the total space of the conormal
bundle of $\Knot$ in $\Lens{-1}$. Then $N_\Knot^*$ is a Lagrangian submanifold
of $T^*\Lens{-1}$ and the intersection of $N_\Knot^*$ with the zero section of $T^*\Lens{-1}$
is the knot $\Knot$. By the symplectomorphism $\varphi_\mu$, we see $\varphi_\mu^{-1}(N_\Knot^*)$
is a Lagrangian submanifold of $Y_\mu$ and the intersection $\varphi_\mu^{-1}(N_\Knot^*)\cap \Lens{-1}$
is a knot in $\Lens{-1}$ which is isomorphic to $\Knot\in\Lens{-1}$.
\par
In the following context, we need to make $\mu\rightarrow 0$ to do the conifold transition. Then
the zero section $\ThreeSphere_\mu/\Zp$ will be shrunk to the singular point of $Y_0$ such that
\[
    \varphi_0: (Y_0/\Zp)\backslash \{0\}\rightarrow T^*\Lens{-1}\backslash \Lens{-1}
\]
is a symplectomorphism. As a by-product, the knot $\Knot\in N_\Knot^*$ is shrunk as well and we lose
the information of the knot in the Lagrangian submanifold. In the case of $\ThreeSphere_\mu$, a solution to this problem is given in
\cite[Section 3.2]{DSV} by constructing a new Lagrangian submanifold in the
deform conifold $Y_\mu$. We briefly review the construction in the following. Let $\phi_\mu:Y_\mu\rightarrow T^*\ThreeSphere$ be the diffeomorphism $\varphi_\mu$ in the case of $p=1$.
We restrict ourselves to torus knots $\Knot=\SphereTorus$ in $\ThreeSphere$.
We consider the 1-dimensional subvariety $Z_\mu\subset Y_\mu$ defined by the complete
intersection of $Y_\mu$ with
\begin{equation}
    f(x,y)=0,\quad f(z,-w)=0  \label{Z_mu condition}
\end{equation}

The subvariety $Z_\mu$ is disconnected in general although the plane curve $f(x,y)=0$ is irreducible.
The equation (\ref{Z_mu condition}) and the defining equation of $Y_\mu$ implies
\[
    (xz)^r-(\mu-xz)^s = 0
\]
Let $\eta = xz$. where $\eta$ is a solution of the equation $u^r-(\mu-u)^s=0$. Each such solution
$\eta$ determines a connected component of $Z_\mu$ of the form
\[
    (x,y,z,w) = (t^s, t^r, \eta t^{-s}, (\eta-\mu)t^{-r})
\]
with $t\in\bC^*$.
\par
Since the coefficients of $f$ are real, $Z_\mu$ is preserved under the anti-homomorphic involution
$I$. Each connected component of the intersection $Z_\mu\cap\ThreeSphere_\mu$ is isomorphic to
the knot $\SphereTorus$ in $\ThreeSphere$. Now let
\[
    P_a = \{(u,v)\in T^*\ThreeSphere~|~|v|=a\}
\]
be the sphere bundle of radius $a$ in $T^*\ThreeSphere$. Suppose $C_\mu$ is an irreducible
component of $Z_\mu$, then for small $a>0$, the intersection $\phi_\mu(C_\mu)\cap P_a$ is nontrivial
and the projection $\pi(\phi_\mu(C_\mu)\cap P_a)$ is equal to $\phi_\mu(C_\mu\cap\ThreeSphere_\mu)
=\SphereTorus\subset\ThreeSphere$. Here $\pi:T^*\ThreeSphere\rightarrow\ThreeSphere$ is the projection. Let
$\gamma_{\mu, a}$ (simply denoted by $\gamma_a$) be the path $\phi_\mu(C_\mu)\cap P_a$:
\begin{equation*}
    \begin{aligned}
        \gamma_a(t): S^1&\rightarrow T^*\ThreeSphere
        \\
        t&\mapsto\gamma_a(t)=(g(t),h(t))
    \end{aligned}
\end{equation*}
Then the path $(g(t),0)\in\ThreeSphere$ is the knot $\SphereTorus$. The conormal bundle $N_\Knot^*$
is given by
\[
    \{(u,v)\in T^*\ThreeSphere : u=g(t), \langle v,g'(t)\rangle = 0\}
\]
where $g'(t)$ is the derivative of $g$ and $\langle,\rangle$ is the natural pairing between
tangent and cotangent vectors. We also define the Lagrangian $M_{\gamma_a}\subset T^*\ThreeSphere$
as
\begin{equation}\label{LiftKnot}
    M_{\gamma_a}=\{(u,v)\in T^*\ThreeSphere : u=g(t), \langle v-h(t), g'(t)\rangle = 0\}
\end{equation}
The Lagrangian $M_{\gamma_a}$ is obtained from $N_\Knot^*$ by translating $N_\Knot^*$
with the cotangent vector $h(t)$ fibrewisely. We denote the Lagrangian $\phi_\mu^{-1}(M_{\gamma_a})$
by $M_{\mu}$. By construction, there is a holomorphic cylinder $C_{\mu, a}$ contained in $C_\mu$,
with one boundary component in $\ThreeSphere_\mu$ and the other boundary component in $M_{\mu}$.
\par
When $\mu=0$, $Z_0$ has two special irreducible components $C^\pm$:
\begin{equation*}
    \begin{aligned}
        C^+ &= \{f(x,y)=0\}\cap\{z=w=0\}
        \\
        C^- &= \{f(z,-w)=0\}\cap\{x=y=0\}
    \end{aligned}
\end{equation*}
Both of $C^\pm$ meet the singular point of $Y_0$ and the anti-holomorphic involution $I$ exchanges
$C^\pm$. Consider the path $\gamma^+$ defined by $\phi_0(C^+\backslash \{0\})\cap P_a$. Then by the
construction in (\ref{LiftKnot}), we obtain the corresponding Lagrangian $M_{\gamma^+}$ in
$T^*\ThreeSphere$ and we denote $\phi_0^{-1}(M_{\gamma^+})$ by $M_0$. For small $\mu>0$, there
exists an irreducible component $C_\mu$ of $Z_\mu$ such that there exists a connected component
$\gamma_{\mu,a}$ of $\phi_\mu(C_\mu)\cap P_a$ which specializes to $\gamma^+$ as $\mu\rightarrow 0$.
Therefore we obtain a family of Lagrangians $M_{\mu}$ which specializes to $M_0$ as $\mu\rightarrow0$.
On the other hand, the boundary component of $C_{\mu, a}$ lying on $M_{\mu}$ is not shrunk
as $\mu\rightarrow 0$. It implies that $C_{0,a}$ is a singular holomorphic disk in $Y_0$.

\subsection{Conifold transition and Lagrangians in the orbifold resolved conifold}
\label{sec:conifold transition}
The orbifold resolved conifold is $\X=\ReducedConi$. The conifold transition is a transition
from $(\Lens{-1},\Knot)$ to the orbifold resolved conifold with Lagrangian $(\X,\cL_\Knot)$ where
we can do Gromov-Witten theory.
This process consists of the following steps:
\begin{itemize}
    \item [$\bullet$] Lift $(\Lens{-1},\Knot)$ to $(\ThreeSphere,\Tilde{\Knot})$
                    and lift $(Y_\mu/\Zp, N_\Knot^*)$ to $(Y_\mu, M)$,
                    where $M$
                    is the total space of
                    the conormal bundle of $\tilde{\Knot}$ in $\ThreeSphere$.
    \item [$\bullet$] Translate the Lagrangian $M$ fiberwisely to obtain $M_\mu$, which does not intersect with
                    the zero section of $Y_\mu\cong T^*\ThreeSphere$.
    \item [$\bullet$] Shrink $\ThreeSphere_\mu$ by setting $\mu\rightarrow 0^+$.
    \item [$\bullet$] Blow up the singularity of $Y_0$ at the origin, we get the
                    resolved conifold $\tilde{\X}$.
    \item [$\bullet$] Equip $\tilde{\X}$ a sympletic form $\omega_{\tilde{\X},\ep}$ and a Lagrangian $\tilde{L}_\ep$ obtained from $M_\mu$.
    \item [$\bullet$] Take the quotient of $\tilde{\X}$ by $\Zp$ and we get the orbifold
                    resolved conifold $\X$ together with a Lagrangian $L_\Knot$.
\end{itemize}

\begin{figure}[h]
 \centering
 \begin{tikzcd}
 (Y_\mu, M_{\mu}) \arrow[rr, "\mu\rightarrow0^+"] \arrow[dd] & &  (Y_0, M_0) & &  (\tilde{\mathcal{X}},\tilde{L}_{\ep}) \arrow[ll,"p" '] \arrow[dd] \\
                                                & &    & &                                           \\
 (Y_\mu/\mathbb{Z}_p, N_\Knot^*) \arrow[rrrr, dashed]        & &    & & (\mathcal{X},L_\Knot)
 \end{tikzcd}
 \caption{conifold transition}
\end{figure}

In the following, we restrict ourselves to the torus knots $\torus$ in
lens space. As shown in Section \ref{KnotsInLensSpace}, we can lift $\torus$
to $\SphereTorus\subset\ThreeSphere$. The knot after lifting has $p$ components.
Every connected component $\tilde{\Knot}$ is the intersection of $C_\mu\cap\ThreeSphere_\mu$
obtained from the polynomial $f(x,y)=x^r-y^s$ (i.e. $\tilde{\Knot}\cong\SphereTorus$). Then the conormal bundle $M$ of $\SphereTorus$ is
a Lagrangian of $Y_\mu$. As in equation (\ref{LiftKnot}), we can lift $M$ fiber-wisely by the
cotangent vector $h(t)$ to obtain a new Lagrangian $M_{\mu}$.
\par
We shrink $\ThreeSphere_\mu$ and blow up $\bC^4$ along the subspace $\{(x,y,z,w)~|~y=z=0\}$.
The resolution of $Y_0$ is $p:\tilde{\X}\rightarrow Y_0$, where
\[
    \tilde{\X} = \ResolvedConi,
\]
and $\tilde{\X}$ is a subspace of $\bC^4\times\bP^1$ defined by:
\[
    \{((x,y,z,w),[\lambda,\rho])~|~x\lambda = w\rho, y\lambda = z\rho\}.
\]
For $\epsilon\geq 0$, we consider the symplectic form ($\omega_{\bC^4}+\epsilon^2\omega_{\bP^1}$)
on $\bC^4\times \bP^1$. $\tilde{\X}$ can be equipped with the symplectic form $\omega_{\tilde{\X},\epsilon}$
by restriction
\[
    \omega_{\tilde{\X},\epsilon} := (\omega_{\bC^4}+\epsilon^2\omega_{\bP^1}) |_{\tilde{\X}}
\]
Let $B(\epsilon)=\{(y,z)\in\bC^2~|~|y|^2+|z|^2\leq \epsilon^2\}$ be the ball of radius $\epsilon$.
There is a radial map $\rho_\epsilon: \bC^2\backslash\{0\} \rightarrow \bC^2\backslash B(\epsilon)$,
\[
    \rho_\epsilon(y,z) = \frac{\sqrt{|y|^2+|z|^2+\epsilon^2}}{\sqrt{|y|^2+|z|^2}}(y,z)
\]
Consider the map $\varrho_\epsilon = id_{\bC^2}\times \rho_\ep: \bC^2\times
(\bC^2\backslash\{0\}) \rightarrow \bC^2\times(\bC^2\backslash B(\epsilon))$. Then $\varrho_\epsilon
(Y_0)\subset Y_0$ and it maps $Y_0\backslash\{0\}$ to $Y_0(\epsilon):= Y_0\backslash (Y_0\cap
(\bC^2\times B(\epsilon)))$. We also have that the map
\[
    \psi_\ep = \varrho_\epsilon|_{Y_0\backslash\{0\}}\circ p|_{\tilde{\X}\backslash\bP^1}: \tilde{\X}\backslash
    \bP^1\rightarrow Y_0(\epsilon)
\]
is a symplectomorphism.
\par
Recall the path $\gamma^+$ defined in Section \ref{KnotsInLensSpace}.
Now consider the path $\gamma_\epsilon^+$ obtained by applying the radial map
\[
    \gamma_\epsilon^+ = \phi_0\circ\varrho_\epsilon\circ\phi_0^{-1}\circ\gamma^+
    : S^1\rightarrow T^*\ThreeSphere
\]
and the Lagrangian
$M_{\gamma_\epsilon^+}\subset T^*\ThreeSphere$
defined by (\ref{LiftKnot}).
Then we can get a Lagrangian $\tilde{L}_\epsilon$ on the resolved conifold ($\tilde{\X}, \omega_{\tilde{\X},\ep}$) by
\[
    \tilde{L}_\epsilon = \psi^{-1}_\epsilon(\phi_0^{-1}(M_{\gamma_\epsilon^+}))
\]
Since $\tilde{\Knot}$ is the torus knot $\SphereTorus$, we use $\tilde{L}_{r,s}$
to denote the Lagrangian $\tilde{L}_\ep$ on $\tilde{\X}$. The strict transform $C\subset\tilde{\X}$ of $C^+$ is contained
in the affine patch $\{\rho\neq 0\}\subset \tilde{\X}$ with coordinates
\[
    x,y,\xi = \frac{\lambda}{\rho}
\]
and $C$ is given by
\[
    f(x,y) = 0, \quad\xi = 0
\]
The intersection of $\tilde{L}_{r,s}$ and $C$ is the boundary of a singular holomorphic disk $D_\ep$. The disk $D_\ep$ has the
parametrization
\[
    (x,y,\xi) = (t^s, t^r, 0), |t|\leq b_1,
\]
where $b_1$ is the unique positive real number satisfying
\[
    b_1^{2r} + b_1^{2s} = 4a.
\]
Here $a$ is the real number for the sphere bundle $P_a$ and it is used in the definition of
$M_{\gamma_a}$ in (\ref{LiftKnot}).
\par
Now consider the torus action $\tilde{T}_{r,s}\cong\C^*$
which acts on $\tilde{\X}$ in the following way:
\begin{equation}\label{ActionOnX}
    u\cdot((x,y,z,w),[\lambda:\rho]) = ((u^sx, u^ry, u^{-s}z, u^{-r}w), [u^{-(s+r)}\lambda:\rho])
\end{equation}
By \cite[Section 6.3]{DSV}, we have the following properties:
\begin{itemize}
    \item [$\bullet$] The torus action $(\tilde{T}_{r,s})_{\mathbb{R}}$ preserves the pair $(\tilde{\X},\tilde{L}_{r,s})$.
    \item [$\bullet$] The torus action $(\tilde{T}_{r,s})_{\mathbb{R}}$ preserves $C$.
    \item [$\bullet$] The torus action $(\tilde{T}_{r,s})_{\mathbb{R}}$
                    preserves the singular holomorphic disk $D_\ep$.
    \item [$\bullet$] $D_\ep$ is the unique torus invariant holomorphic disc on $\tilde{\cX}$ with boundary in $\tilde{L}_{r,s}$.
\end{itemize}
Furthermore, we consider the $\Zp$-action on $\tilde{\cX}$ derived from
$Y_\mu$: let $\zeta=e^{2\pi i/p}$
\begin{equation}\label{ZpOnX}
    \zeta\cdot((x,y,z,w),[\lambda:\rho]) = ((\zeta x,\zeta ^{-1}y, \zeta^{-1}z,\zeta w), [\lambda:\rho]).
\end{equation}
Now consider the following condition
\begin{equation}\label{eqn:cond}
    \left\{
    \begin{aligned}
        &r+s=pk
        \\
        &r, s, p \text{ are pairwisely coprime}.
    \end{aligned}
    \right.
\end{equation}
Under condition \eqref{eqn:cond}, take $u=\zeta= e^{2\pi i/p}$ in the torus action (\ref{ActionOnX}), the torus action reads
\[
    e^{2\pi i/p}\cdot((x,y,z,w),[\lambda:\rho]) = ((\zeta^s x,\zeta ^{-s}y, \zeta^{-s}z,\zeta^s w), [\lambda:\rho]).
\]
By comparing with the $\Zp$-action (\ref{ZpOnX}), we have that the $\Zp$-action preserves $\tilde{L}_{r,s}$ due to the fact that the torus action $(\tilde{T}_{r,s})_{\mathbb{R}}$ preserves $\tilde{L}_{r,s}$.
It means that under condition \eqref{eqn:cond}, we can take the quotient of the
pair $(\tilde{\X}, \tilde{L}_{r,s})$ by $\Zp$ and we denote the resulting pair by $(\X, L_{r,s})$. It is easy to see that the torus action (\ref{ActionOnX}) commutes with the $\Zp$-action (\ref{ZpOnX}). So the action (\ref{ActionOnX}) descents to a $(\tilde{T}_{r,s})_{\mathbb{R}}$-action on $(\X, L_{r,s})$ and the $\tilde{T}_{r,s}$-action on $\tilde{\X}$ descends to $\X$. Moreover, we define a $\Zp$-action on $\tilde{T}_{r,s}$ by multiplication by $p-$th roots of unit and let $\T= \tilde{T}_{r,s}/\Zp\cong \bC^*$. Then the $\tilde{T}_{r,s}$-action on $\X$ induces a $\T$-action on $\X$ and the $(\tilde{T}_{r,s})_{\mathbb{R}}$-action on $(\X, L_{r,s})$ induces a $(\T)_{\mathbb{R}}$-action on $(\X, L_{r,s})$.
More explicitly, the torus action $\T$ on $\X$ is given by
\begin{equation*}
    u\cdot [(x, y, w, z), [\lambda: \rho]] = [(u^{s/p} x,u ^{r/p}y, u^{-s/p}z,u^{-r/p} w), [u^{-(s+r)/p}\lambda:\rho]]
\end{equation*}
for  $u\in\T$, $[(x, y, w, z), [\lambda: \rho]]\in\X$. The disk $[D_\ep/\Zp]$ in $\X$ is parametrized as $[\{z\in\bC||z|\leq b_1\}/\Zp]$ and there is an embedding
\begin{equation*}
    \begin{aligned}
        [D_\ep/\Zp] &\hookrightarrow (\X,\Lrs)
        \\
        t &\mapsto [t^{s/p}, t^{r/p}, 0]
    \end{aligned}
\end{equation*}
where $t=z^p$ and the bracket $[x,y,\frac{\lambda}{\rho}]$ means the $\Zp$-equivalent class of the affine piece $\{\rho\neq 0\}\subset\tilde{\X}$.

In summary, we obtain a toric Calabi-Yau 3-orbifold $\X=\ReducedConi$ with a $(\T)_{\mathbb{R}}$-invariant Lagrangian $\Lrs$ bounding a singular $(\T)_\bR$-invariant holomorphic disk $[D_\ep/\Zp]$.
These features of $(\X,\Lrs)$ are indispensable for our definition and computation of open Gromov-Witten invariants. \emph{From now on, we always assume that condition \eqref{eqn:cond} holds}.
\section{Open-closed Gromov-Witten theory: A-model topological string}

\subsection{Geometry and equivariant cohomology of the orbifold resolved conifold}\label{sec:equivariantCR}
Let $N\cong \bZ^3$ be a free abelian group, and $\{b_i\}_{1\leq i\leq p+3}$ be vectors in $N$:
\[
    b_1 = (p,0,1), b_2=(0,1,1), b_3=(0,0,1), b_4=(p,-1,1)
\]
\[
    b_{i+4} = (i,0,1), \quad 1\leq i\leq p-1.
\]
Let $\Sigma$ be the fan with two 3-cones $\si_0, \si_1$,
\[
    \sigma_0 = \text{span} \{b_1,b_2,b_3\},\quad \sigma_1 = \text{span} \{b_1,b_3,b_4\}
\]

\begin{figure*}[h]
    \begin{center}
\tikzset{every picture/.style={line width=0.75pt}} 

\begin{tikzpicture}[x=0.75pt,y=0.75pt,yscale=-1,xscale=1]

\draw  [fill={rgb, 255:red, 0; green, 0; blue, 0 }  ,fill opacity=1 ] (95,120.5) .. controls (95,120.22) and (95.22,120) .. (95.5,120) .. controls (95.78,120) and (96,120.22) .. (96,120.5) .. controls (96,120.78) and (95.78,121) .. (95.5,121) .. controls (95.22,121) and (95,120.78) .. (95,120.5) -- cycle ;
\draw  [fill={rgb, 255:red, 0; green, 0; blue, 0 }  ,fill opacity=1 ] (95,75.5) .. controls (95,75.22) and (95.22,75) .. (95.5,75) .. controls (95.78,75) and (96,75.22) .. (96,75.5) .. controls (96,75.78) and (95.78,76) .. (95.5,76) .. controls (95.22,76) and (95,75.78) .. (95,75.5) -- cycle ;
\draw  [fill={rgb, 255:red, 0; green, 0; blue, 0 }  ,fill opacity=1 ] (230,120.5) .. controls (230,120.22) and (230.22,120) .. (230.5,120) .. controls (230.78,120.01) and (231,120.23) .. (231,120.51) .. controls (230.99,120.79) and (230.77,121.01) .. (230.49,121) .. controls (230.21,121) and (229.99,120.78) .. (230,120.5) -- cycle ;
\draw  [fill={rgb, 255:red, 0; green, 0; blue, 0 }  ,fill opacity=1 ] (275,165.5) .. controls (275,165.22) and (275.22,165) .. (275.5,165) .. controls (275.78,165) and (276,165.22) .. (276,165.5) .. controls (276,165.78) and (275.78,166) .. (275.5,166) .. controls (275.22,166) and (275,165.78) .. (275,165.5) -- cycle ;
\draw  [fill={rgb, 255:red, 0; green, 0; blue, 0 }  ,fill opacity=1 ] (140,120.5) .. controls (140,120.22) and (140.22,120) .. (140.5,120) .. controls (140.78,120) and (141,120.22) .. (141,120.5) .. controls (141,120.78) and (140.78,121) .. (140.5,121) .. controls (140.22,121) and (140,120.78) .. (140,120.5) -- cycle ;
\draw  [fill={rgb, 255:red, 0; green, 0; blue, 0 }  ,fill opacity=1 ] (185,120.5) .. controls (185,120.22) and (185.22,120) .. (185.5,120) .. controls (185.78,120) and (186,120.22) .. (186,120.5) .. controls (186,120.78) and (185.78,121) .. (185.5,121) .. controls (185.22,121) and (185,120.78) .. (185,120.5) -- cycle ;
\draw  [fill={rgb, 255:red, 0; green, 0; blue, 0 }  ,fill opacity=1 ] (275,120.5) .. controls (275,120.22) and (275.22,120) .. (275.5,120) .. controls (275.78,120.01) and (276,120.23) .. (276,120.51) .. controls (275.99,120.79) and (275.77,121.01) .. (275.49,121) .. controls (275.21,121) and (274.99,120.78) .. (275,120.5) -- cycle ;
\draw    (95.5,120.5) -- (275.48,120.52) ;
\draw    (95.5,75.5) -- (275.48,120.52) ;
\draw    (95.5,75.5) -- (95.5,120.5) ;
\draw    (95.5,120.5) -- (275.5,165.5) ;
\draw    (275.49,121) -- (275.51,165.98) ;

\draw (279.8,113) node [anchor=north west][inner sep=0.75pt]   [align=left] {$b_1$};
\draw (77.5,68.57) node [anchor=north west][inner sep=0.75pt]   [align=left] {$b_2$};
\draw (77,113.57) node [anchor=north west][inner sep=0.75pt]   [align=left] {$b_3$};
\draw (279.7,157.57) node [anchor=north west][inner sep=0.75pt]   [align=left] {$b_4$};
\draw (131.7,103.37) node [anchor=north west][inner sep=0.75pt]   [align=left] {$b_5$};
\draw (212.5,105.37) node [anchor=north west][inner sep=0.75pt]   [align=left] {$b_{p+3}$};
\draw (115.5,94.07) node [anchor=north west][inner sep=0.75pt]   [align=left] {$\sigma_0$};
\draw (215,129.57) node [anchor=north west][inner sep=0.75pt]   [align=left] {$\sigma_1$};
\draw (176.93,113.37) node [anchor=north west][inner sep=0.75pt]   [align=left] {$\dots$};
\end{tikzpicture}
\end{center}
\caption{the fan of orbifold resolved conifold $\ReducedConi$}
\label{fig:fan}
\end{figure*}
\noindent
The toric orbifold $X_\Sigma$ obtained from $\Sigma$ is the orbifold resolved conifold $\X=\ReducedConi$.

Let $\widetilde{\bT}\cong (\bC^*)^3$ be the embedded torus of $\cX$. Let $\bT\cong(\bC^*)^2\subset \widetilde{\bT}$ be the Calabi-Yau sub-torus which acts trivially on the canonical bundle of $\cX$. We also have a one dimensional sub-torus $\T\subset\bT$. Let $\RT:= H_{\bT}^*(\mathrm{pt})= H^*(B\bT)$, and let $\ST$ be the fractional field of $\RT$:
$$
\RT =\bC[\su_1,\su_2],\quad \ST=\bC(\su_1,\su_2).
$$
Similarly, let $H^*(B\T)=\bC[\sv]$.

Let $\cX_0=\cX_1=[\bC^3/\bZ_p]$ be two copies of affine toric Calabi-Yau 3-orbifolds, where the $\bZ_p$ action on $\bC^3$ is given by
$$
\zeta\cdot (x_1,x_2,x_3):=(x_1,\zeta x_2,\zeta^{-1}x_3)
$$
and $\zeta=e^{\frac{2\pi i}{p}}$ is the generator of $\bZ_p$. Consider the two $\T$ fixed points $\fp_0=[0:1],\fp_1=[1:0]$ on $\cX$. We have two embedding $\iota_\si:\cX_\si\to \cX,\si=0,1$, such that $\iota_\si$ maps the origin of $\cX_\si$ to $\fp_\si$. The actions of $\widetilde{\bT},\bT$, and $\T$ on $\cX$ induces corresponding actions on $\cX_0,\cX_1$.

Let $I\cX_0$ be the inertia stack of $\cX_0$. Then we have
$$
I\cX_0=\bigsqcup_{h\in\bZ_p} \cX_h, \quad \textrm{with}\quad \cX_h=[(\bC^3)^h/\bZ_p].
$$
As a vector space, we have the following decomposition of the Chen-Ruan orbifold cohomology of $\cX_0$:
$$
H^*_\CR(\cX_0;\bC) = \bigoplus_{h\in \bZ_p} H^*(\cX_h;\bC)[2\age(h)] =\bigoplus_{h\in \bZ_p}\bC \one_h,
$$
where $\one_h$ is unit of the ring $H^*(\cX_h;\bC)\cong \bC$ with $\deg(\one_h)=2\age(h)$. Here we have
$$
\age(h)=\left\{\begin{array}{ll}1, &\textrm{if $h=\zeta^k$ with $1\leq k\leq p-1$},\\
0, &\textrm{if $h=1$}.\end{array} \right.
$$
Let $\sw_{1,0},\sw_{2,0},\sw_{3,0}$ be the weights of the $\bT$ action along the 3 coordinate axes of $\cX_0$.
For $h=\xi^k$, $0\leq k\leq p-1$, we define
$$c_1(\xi^k)=\frac{k}{p}, \quad c_2(\xi^k)=0,\quad c_3(\xi^k)=1-\frac{k}{p}-\delta_{0,k} .$$
Then it is easy to see that the $\bT$-equivariant Poincar\'{e} pairing is given by
$$
\langle \one_h,\one_{h'}\rangle_{\cX_0} = \frac{\delta_{hh',1}}{\displaystyle{p \prod_{i=1}^3 \sw_{i,0}^{\delta_{c_i(h),0}} } },
$$
and the $\bT$-equivariant
orbifold cup product is given by
$$
\one_h \star \one_{h'} =\Big(\prod_{i=1}^3 \sw_{i,0}^{c_i(h)+c_i(h')-c_i(hh')}\Big) \one_{hh'}.
$$
Let
$$
\bar{\one}_h:=\frac{\one_h}{\prod_{i=1}^3 \sw_{i,0}^{c_i(h)}}.
$$
$\bST$ is the minimal extension of $\ST$ which contains $\{\sw_{i,0}^{\delta_{c_i(h),0}}|i=1,2,3\}$.
We view $\bar{\one}_h$ as an element in $H^*_{\CR,\bT}(\cX_0;\bC)\otimes_{\RT}\bST$. For any $\gamma\in \bZ_{p}^*$, define
$$
\bar{\phi}_\gamma:=\frac{1}{p}\sum_{h\in \bZ_{p}}\chi_\gamma(h^{-1})\bar{\one}_h.
$$
Then under the basis $\{\bar{\phi}_\gamma\}_{\gamma\in \bZ_{p}^*}$ the $\bT$-equivariant Poincar\'{e} pairing and the $\bT$-equivariant
orbifold cup product is given by
$$
\langle \bar{\phi}_\gamma,\bar{\phi}_{\gamma'}\rangle_{\cX_0} = \frac{\delta_{\gamma \gamma'}}{p^2 \prod_{i=1}^3 \sw_{i,0} },\quad
\bar{\phi}_\gamma\star \bar{\phi}_{\gamma'} = \delta_{\gamma\gamma'}\bar{\phi}_{\gamma}.
$$
In terms of Frobenius structures, $\{\bar{\phi}_\gamma\}_{\gamma\in \bZ_{p}^*}$ is a canonical basis of the semisimple Frobenius algebra
$$
\big(H^*_{\CR,\bT}(\cX_0)\otimes_\RT \bST, \star, \langle \  ,\  \rangle_{\cX_0}\big)
$$
over the field $\bST$.

Applying the same construction to $\cX_1$, we obtain a semisimple Frobenius algebra $\big(H^*_{\CR,\bT}(\cX_1)\otimes_\RT \bST, \star, \langle \  ,\  \rangle_{\cX_1}\big)$.

Consider the two $\T$ fixed points $\fp_0=[0:1],\fp_1=[1:0]$ on $\cX$. We have two embedding $\iota_\si:\cX_\si\to \cX,\si=0,1$, such that $\iota_\si$ maps the origin of $\cX_\si$ to $\fp_\si$. Then we have an isomorphism of Frobenius algebras
\begin{equation}\label{eqn:direct-sum}
\bigoplus_{\si=0,1} \iota_\si^*: H^*_{\CR,\bT}(\cX;\bC)\otimes_\RT\bST
\stackrel{\cong}{\longrightarrow}  \bigoplus_{\si=0,1} H^*_{\CR,\bT}(\cX_\si;\bC)\otimes_{\RT} \bST,
\end{equation}
where the left hand side is also equipped with the $\bT$-equivariant Poincar\'{e} pairing and the $\bT$-equivariant
orbifold cup product. Therefore there exists a unique $\phi_{\si,\gamma}\in H^*_{\CR,\bT}(\cX)\otimes_{\RT} \bST$
such that $\phi_{\si,\gamma}|_{\cX_\sigma}= \bar{\phi}_\gamma$
and $\phi_{\si,\gamma}|_{\fp_{\si'}}=0$ for $\si'\neq\si$. Let $I_\cX:=\{(\si,\gamma): \si=0,1,\gamma\in \bZ_{p}^*\}$. Then
$$
\{\phi_{\si,\gamma}: (\si,\gamma)\in I_\cX\}
$$
is a canonical basis of the semisimple $\bST$-algebra  $H^*_{\CR,\bT}(\cX;\bC)\otimes_{\RT}\bST$:
$$
\phi_{\si,\gamma}\star \phi_{\si',\gamma'}
=\delta_{\si,\si'}\delta_{\gamma,\gamma'} \phi_{\si,\gamma}.
$$
The $\bT$-equivariant Poincar\'{e} pairing of $\cX$ is given by
$$
( \phi_{\si,\gamma},\phi_{\si',\gamma'})_{\cX,\bT}
=\frac{\delta_{\si,\si'}\delta_{\gamma,\gamma'}}{\Delta^{\si,\gamma}},\quad
\Delta^{\si,\gamma}=p^2 \prod_{i=1}^3\sw_{i,\si}.
$$
For convenience, we use $\bsi$ to denote the pair $(\si,\gamma)$. Let
$$
\hat{\phi}_{\bsi}=\sqrt{\Delta^{\bsi}}\phi_{\bsi},\quad \bsi\in I_\cX.
$$
Then $\{\hat{\phi}_{\bsi}\}_{\bsi\in I_\cX}$ is a normalized canonical basis of $H^*_{\CR,\bT}(\cX;\bC)\otimes_{\RT}\bST$, namely
$$
\hat{\phi}_{\bsi}\star_{\cX}\hat{\phi}_{\bsi'} = \delta_{\bsi,\bsi'} \sqrt{\Delta^{\bsi}}\hat{\phi}_{\bsi},\quad
(\hat{\phi}_{\bsi}, \hat{\phi}_{\bsi'})_{\cX,\bT}= \delta_{\bsi,\bsi'}.
$$

\subsection{Generating functions of equivariant Gromov-Witten invariants}
Let $X$ be the coarse moduli space of $\cX$. Given nonnegative integers $g$, $n$ and an effective curve class
$d\in H_2(X;\bZ)$, let $\Mbar_{g,n}(\cX,d)$ be the
moduli stack of genus $g$, $n$-pointed, degree $d$ twisted
stable maps to $\cX$. Let $\ev_i:\Mbar_{g,n}(\cX,d)\to \cI\cX$ be the evaluation map
at the $i$-th marked point. The $\bT$-action on $\cX$ induces
$\bT$-actions on the moduli space $\Mbar_{g,n}(\cX,d)$ and on the inertia stack $\cI\cX$, and
the evaluation map $\ev_i$ is $\bT$-equivariant. Similarly, we can define the moduli space $\Mbar_{g,n}(X,d)$.

For $i=1,\cdots,n$, let $\bL_i$ be the $i$-th tautological line bundle over $\Mbar_{g,n}(X,d)$ formed
by the cotangent line at the $i$-th marked point. Define the $i$-th descendent class $\psi_i$ as
$$
\psi_i=c_1(\bL_i)\in H^2(\Mbar_{g,n}(X,d);\bQ).
$$
Consider the map $p:\Mbar_{g,n}(\cX,d)\to \Mbar_{g,n}(X,d)$ induced by the forgetful map $\cX\to X$. For $i=1,\cdots,n$, let
$$
\hat{\psi}_i := p^*\psi_i  \in H^2(\Mbar_{g,n}(\cX,d);\bQ).
$$
Given $\gamma_1,\ldots, \gamma_n\in H_{\bT}^*(\cX,\bC)$ and $a_1,\ldots,a_n\in \bZ_{\geq 0}$, we define genus $g$, degree $d$, $\bT$-equivariant descendant Gromov-Witten invariants of $\cX$:
$$
\langle \tau_{a_1}(\gamma_1)\ldots\tau_{a_n}(\gamma_n)\rangle_{g,n,d}^{\cX,\bT}:=
\int_{[\Mbar_{g,n}(\cX,d)^{\bT}]^\vir} \frac{\prod_{j=1}^n \hat{\psi}_j^{a_j}\ev_j^*(\gamma_j)\mid_{\Mbar_{g,n}(\cX,d)^{\bT}}}{e_{\bT}(N^\vir)}
$$
where $\Mbar_{g,n}(\cX,d)^{\bT}$ is the $\bT$-fixed locus, and $e_{\bT}(N^\vir)$ is the $\bT$-equivariant Euler class of the virtual normal bundle. We also define genus $g$, degree $d$ primary Gromov-Witten invariants of $\cX$ as:
$$
\langle \gamma_1,\ldots, \gamma_n\rangle_{g,n,d}^{\cX,\bT}:=
\langle \tau_0(\gamma_1) \cdots \tau_0(\gamma_n)\rangle_{g,n,d}^{\cX,\bT}.
$$
By restricting to $\T$, we obtain $\T$-equivariant descendant Gromov-Witten invariants $\langle \gamma_1,\ldots, \gamma_n\rangle_{g,n,d}^{\cX,\T}$ and $\T$-equivariant primary Gromov-Witten invariants $\langle \tau_{a_1}(\gamma_1)\ldots\tau_{a_n}(\gamma_n)\rangle_{g,n,d}^{\cX,\T}$.

Let $\NE(\cX)\subset H_2(\cX;\bR)=H_2(X;\bR)$ be the Mori cone
generated by effective curve classes in $X$.
Let $E(\cX)$ denote the semigroup $\NE(\cX)\cap H_2(X;\bZ)$. In our case, $E(\cX)$ is identified with the set of nonnegative integers. Define the Novikov ring
\[
\nov:=\widehat{\bC[E(\cX)]}=\{\sum_{d\in E(\cX)} c_d \fQ^d: c_d\in \bC\}.
\]
Given $a_1,\ldots, a_n\in \bZ_{\geq 0}$,
$\gamma_1,\ldots,\gamma_n\in H^*_{\CR,\bT}(\cX)\otimes_\RT \bST$, define
the following generating function:
$$
\ll  \gamma_1\hat{\psi}^{a_1},  \cdots, \gamma_n\hat{\psi}^{a_n} \gg ^{\cX,\bT}_{g,n}
:=\sum_{m=0}^\infty \sum_{d\in E(\cX)}\frac{\fQ^d}{m!}\langle
\gamma_1\hat{\psi}^{a_1}, \cdots, \gamma_n\hat{\psi}^{a_n}, t^m \rangle^{\cX,\bT}_{g,n+m,d}
$$
where $t\in H^*_{\CR,\bT}(\cX)\otimes_{R_\bT}\bST$.  Let $t=\sum_{\bsi}\hat{t}^\bsi\hat{\phi}_{\bsi}$. For $i=1,\ldots,n$, introduce formal variables
$$
\bu_i =\bu_i(z)= \sum_{a\geq 0}(u_i)_a z^a
$$
where $(u_i)_a \in H^*_{\CR,\bT}(\cX;\bC)\otimes_\RT\bST$.
Define
\begin{align*}
\ll \bu_1,\ldots, \bu_n  \gg_{g,n}^{\cX,\bT} &=
\ll \bu_1(\hat{\psi}),\ldots, \bu_n(\hat{\psi})  \gg_{g,n}^{\cX,\bT}\\
&=\sum_{a_1,\ldots,a_n\geq 0}
\ll (u_1)_{a_1}\hat{\psi}^{a_1}, \cdots, (u_n)_{a_n}\hat{\psi}^{a_n} \gg_{g,n}^{\cX,\bT}.
\end{align*}

\subsection{Equivariant quantum cohomology and Frobenius structures}\label{sec:QH}
The equivariant quantum cohomology of $\cX$ is defined by its genus zero primary Gromov-Witten invariants. More concretely, for any  $a,b,c\in H_{\CR,\bT}^*(\cX;\bST)$, define the quantum product $\star_t$ by
\begin{equation}
(a\star_t b,c)_{\cX,\bT}:=\ll a,b,c\gg_{0,3}^{\cX,\bT}.
\end{equation}
Let
$$
\novT:= \bST\otimes_{\bC}\nov =\bST [\![ E(\cX)]\!].
$$
Then $H:=H^*_{\CR,\bT}(\cX;\novT)$ is a free $\novT$-module of rank $2p$. Given a point $t\in H$, we write it as $t=\sum_{\bsi\in I_{\cX}}\hat{t}^{\bsi} \hat{\phi}_{\bsi}$, where $\{\hat{t}^{\bsi}\}$ are viewed as the coordinates of $t$. We also write $t=\btau'+\btau''\in H^{*}_{\CR,\bT}(\X)\otimes_{R_\bT} \bST$, where $\btau'\in H^2_\bT(\X)\otimes_{R_\bT} \bST$, and $\btau''$ is a linear combination of elements in $H^{\neq 2}_{\CR,\bT}(\X)\otimes_{R_\bT} \bST$
and degree 2 twisted sectors. Choose arbitrary coordinate systems $\tau'$ and $\tau''$ for $\btau'$ and $\btau''$ respectively. By divisor equation, we have
\begin{equation}
(a\star_t b,c)_{\cX,\bT}=\ll a,b,c\gg_{0,3}^{\cX,\bT}\in \bST[\![\tfQ, \tau'' ]\!],
\end{equation}
where $\tfQ^d:=\fQ^d\exp(\langle d,\btau'\rangle)$.
Define the formal scheme
$$
\hat{H} :=\mathrm{Spec}(\novT[\![ \hat{t}^{\bsi}:\bsi\in I_{\cX} ]\!]).
$$
Denote the structure sheaf and the tangent sheaf of $\hat{H}$ by $\cO_{\hat{H}}$ and $\cT_{\hat{H}}$ respectively.
Then $\cT_{\hat{H}}$ is a sheaf of free $\cO_{\hat{H}}$-modules of rank $2p$.
For any open set $U$ in $\hat{H}$, we have
$$
\cT_{\hat{H}}(U)  \cong \bigoplus_{\bsi\in I_{\cX}}\cO_{\hat{H}}(U) \frac{\partial}{\partial \hat{t}^{\bsi}}.
$$
The quantum product together with the $\bT$-equivariant Poincar\'{e} pairing define the structure of a formal
Frobenius manifold on $\hat{H}$:
$$
\frac{\partial}{\partial \hat{t}^{\bsi}} \star_t \frac{\partial}{\partial \hat{t}^{\bsi'}}
=\sum_{\bsi''\in I_{\cX}} \ll \hat{\phi}_{\bsi},\hat{\phi}_{\bsi'},\hat{\phi}_{\bsi''}\gg_{0,3}^{\cX,\bT}
\frac{\partial}{\partial \hat{t}^{\bsi''}}
\in \Gamma(\hat{H}, \cT_{\hat{H}}).
$$
$$
( \frac{\partial}{\partial \hat{t}^{\bsi}},\frac{\partial}{\partial \hat{t}^{\bsi'}})_{\cX,\bT} =\delta_{\bsi,\bsi'}.
$$

The set of global sections $\Gamma(\hat{H},\cT\hat{H})$ is a free $\cO_{\hat{H}}(\hat{H})$-module of rank $2p$:
$$
\Gamma(\hat{H},\cT\hat{H})=
\bigoplus_{\bsi\in I_{\cX}}\cO_{\hat{H}}(\hat{H})\frac{\partial}{\partial \hat{t}^{\alpha}}.
$$
Under the quantum product $\star_t$, the triple $(\Gamma(\hat{H},\cT\hat{H}),\star_t,(, )_{\cX,\bT})$ is a Frobenius algebra over the ring $\cO_{\hat{H}}(\hat{H})=\novT[\![ \hat{t}^{\bsi}:\bsi\in I_{\cX} ]\!]$. The triple $(\Gamma(\hat{H},\cT\hat{H}),\star_t,(, )_{\cX,\bT})$ is called the \emph{big quantum cohomology} of $\cX$ and is denoted by $QH^*_{\bT}(\cX)$.

The semi-simplicity of the classical cohomology $H^*_{\bT}(\cX;\bC)\otimes_\RT \bST$ implies the semi-simplicity of the quantum cohomology $QH^*_{\bT}(\cX)$. In fact, there exists a canonical basis $\{\phi_{\bsi}(t)\}_{\bsi\in I_{\cX}}$ of $QH^*_{\bT}(\cX)$ characterized by the property that
\begin{equation}
\phi_{\bsi}(t)\to  \phi_{\bsi},\quad\mathrm{when }\quad t,\tfQ\to 0,\quad \bsi\in I_{\cX}.
\end{equation}
We define $\{\phi^{\bsi}(t)\}_{\bsi\in I_{\cX}}$ to be the dual basis to $\{\phi_{\bsi}(t)\}_{\bsi\in I_{\cX}}$ with respect to the metric $(,)_{\cX,\bT}$. See \cite{LP} for more general discussions on the canonical basis.

\subsection{Canonical coordinates and the transition matrix}
The canonical coordinates $\{ u^{\bsi}=u^{\bsi}(t)|\bsi\in I_{\cX}\}$ on the formal Frobenius
manifold $\hat{H}$ are characterized by
\begin{equation}\label{eqn:partial-u}
\frac{\partial}{\partial u^{\bsi}} = \phi_{\bsi}(t).
\end{equation}
up to additive constants in $\novT$. We choose canonical coordinates
such that they vanish when $\tfQ=0,\hat{t}^\bsi=0,\bsi\in I_{\cX}$. Then $u^{\bsi}-\sqrt{\Delta^{\bsi}}\hat{t}^{\bsi}$ vanishes when $\tfQ=0,\hat{t}^\bsi=0,\bsi\in I_{\cX}$.

Define $\Delta^{\bsi}(t)$ by the following equation:
$$
(\phi_{\bsi}(t), \phi_{\bsi'}(t))_{\cX,\bT} =\frac{\delta_{\bsi,\bsi'}}{\Delta^{\bsi}(t)}.
$$
Then $\Delta^{\bsi}(t) \to \Delta^{\bsi}$  in the large radius limit $\tfQ,\hat{t}^\bsi\to 0$. The normalized canonical basis of $(\hat{H},\star_t)$ is given by
$$
\{ \hat{\phi}_{\bsi}(t):= \sqrt{\Delta^{\bsi}(t)}\phi_{\bsi}(t)| \bsi\in I_{\cX}\}.
$$
They satisfy
$$
\hat{\phi}_{\bsi}(t)\star_t \hat{\phi}_{\bsi'}(t) =\delta_{\bsi, \bsi'}\sqrt{\Delta^{\bsi}(t)}\hat{\phi}_{\bsi}(t),\quad
(\hat{\phi}_{\bsi}(t), \hat{\phi}_{\bsi'}(t))_{\cX,\bT}=\delta_{\bsi,\bsi'}.
$$
We call $\{\hat{\phi}_{\bsi}(t)|\bsi\in I_{\cX}\}$ the {\em quantum} normalized canonical basis
to distinguish it from the  {\em classical} normalized canonical basis $\{ \hat{\phi}_{\bsi}|\bsi\in I_{\cX}\}$.
The quantum canonical basis tends to the classical canonical
basis in the large radius limit: $\hat{\phi}_{\bsi}(t)\to \hat{\phi}_{\bsi}$ as $\tfQ,\hat{t}^\bsi\to 0$.

Let $\Psi=(\Psi_{\bsi'}^{\  \bsi})$ be the transition matrix between the classical and quantum
normalized canonical bases:
\begin{equation}\label{eqn:Psi-phi}
\hat{\phi}_{\bsi'}=\sum_{\bsi\in I_{\cX}} \Psi_{\bsi'}^{\ \bsi} \hat{\phi}_\bsi(t).
\end{equation}
Then $\Psi$ is a $2p\times 2p$ matrix, and $\Psi\to \one$ (the identity matrix)
in the large radius limit $\tfQ,\hat{t}^\bsi\to 0$. Both
the classical and quantum normalized canonical bases are orthonormal with respect
to the $\bT$-equivariant Poincar\'{e} pairing $(\ ,\ )_{\cX,\bT}$, so $\Psi^T\Psi= \Psi \Psi^T= \one$,
where $\Psi^T$ is the transpose of $\Psi$, or equivalently
$$
\sum_{\bsi''\in I_{\cX}} \Psi_{\bsi''}^{\ \bsi} \Psi_{\bsi''}^{\ \bsi'} =\delta_{\bsi,\bsi'}
$$
Equation \eqref{eqn:Psi-phi} can be rewritten as
$$
\frac{\partial}{\partial \that^{\bsi'}} =\sum_{\bsi\in I_{\cX}} \Psi_{\bsi'}^{\ \bsi}
\sqrt{\Delta^{\bsi}(t)} \frac{\partial}{\partial u^{\bsi}}
$$
which is equivalent to
\begin{equation}
\label{eqn:Psi-matrix}
\frac{du^{\bsi}}{\sqrt{\Delta^{\bsi}(t)}} =
\sum_{\bsi'\in I_{\cX}}
d\that^{\bsi'} \Psi_{\bsi'}^{\ \bsi}.
\end{equation}

\subsection{Equivariant quantum differential equation and fundamental solution}
We consider the Dubrovin connection $\nabla^z$, which is a family
of connections parametrized by $z\in \bC\cup \{\infty\}$, on the tangent bundle
$T_{\hat{H}}$ of the formal Frobenius manifold $\hat{H}$:
$$
\nabla^z_{\bsi}=\frac{\partial}{\partial \hat{t}^{\bsi}} -\frac{1}{z} \hat{\phi}_{\bsi}\star_t.
$$
The commutativity (resp. associativity)
of $\star_t$ implies that $\nabla^z$ is a torsion
free (resp. flat) connection on $T_{\hat{H}}$ for all $z$. The equation
\begin{equation}\label{eqn:qde}
\nabla^z \mu=0
\end{equation}
for a section $\mu\in \Gamma(\hat{H},\cT_{\hat{H}})$ is called the {\em $\bT$-equivariant
big quantum differential equation} ($\bT$-equivariant big QDE). Let
$$
\cT_{\hat{H}}^{f,z}\subset \cT_{\hat{H}}
$$
be the subsheaf of flat sections with respect to the connection $\nabla^z$.
For each $z$, $\cT_{\hat{H}}^{f,z}$ is a sheaf of
$\novT$-modules of rank $2p$.

A section $L\in \End(T_{\hat{H}})=\Gamma(\hat{H},\cT_{\hat{H}}^*\otimes\cT_{\hat{H}})$
defines an $\cO_{\hat{H}}(\hat{H})$-linear map
$$
L: \Gamma(\hat{H},\cT_{\hat{H}})= \bigoplus_{\bsi\in I_{\cX}} \cO_{\hat{H}}(\hat{H})
\frac{\partial}{\partial \hat{t}^{\bsi}}
\to \Gamma(\hat{H},\cT_{\hat{H}})
$$
from the free $\cO_{\hat{H}}(\hat{H})$-module $\Gamma(\hat{H},\cT_{\hat{H}})$ to itself.
Let $L(z)\in \End(T_{\hat H})$ be a family of endomorphisms of the tangent bundle $T_{\hat{H}}$
parametrized by $z$. The section $L(z)$ is called a {\em fundamental solution} to the $\bT$-equivariant QDE if
the $\cO_{\hat{H}}(\hat{H})$-linear map
$$
L(z): \Gamma(\hat{H},\cT_{\hat{H}}) \to \Gamma(\hat{H},\cT_{\hat{H}})
$$
restricts to a $\novT$-linear isomorphism
$$
L(z): \Gamma(\hat{H},\cT_H^{f,\infty})=\bigoplus_{\bsi\in I_{\cX}} \novT \frac{\partial}{\partial \hat{t}^{\bsi}}
\to \Gamma(\hat{H},\cT_H^{f,z}).
$$
between rank $2p$ free $\novT$-modules.

\subsection{The $\cS$-operator}\label{sec:A-S}
We define the $\cS$-operator as follows.
For any cohomology classes $a,b\in H_{\CR,\bT}^*(\cX;\bST)$,
$$
(a,\cS(b))_{\cX,\bT}:=(a,b)_{\cX,\bT}
+\ll a,\frac{b}{z-\hat{\psi}}\gg^{\cX,\bT}_{0,2}
$$
where
$$
\frac{b}{z-\hat{\psi}}=\sum_{i=0}^\infty b\hat{\psi}^i z^{-i-1}.
$$
The $\cS$-operator can be viewed as an element in $\End(T_{\hat{H}})$ and is a fundamental solution to the $\bT$-equivariant
big QDE \eqref{eqn:qde}.  The proof for $\cS$ being a fundamental solution can be found in  \cite[Section 10.2]{CK}
for the smooth case and in \cite{Ir09} for the orbifold case which is a direct generalization of the smooth case.

We consider several different (flat) bases for $H_{\CR,\bT}^*(\cX;\bST)$:
\begin{enumerate}
\item The classical canonical basis $\{ \phi_{\bsi}:\bsi\in I_\cX \}$ defined in Section \ref{sec:equivariantCR}.
\item The basis dual to the classical canonical basis with respect to the $\bT$-equivariant Poincare pairing:
$\{ \phi^{\bsi} =\Delta^{\bsi} \phi_{\bsi}: \bsi \in I_\cX \} $.
\item The classical normalized canonical basis
$\{ \hat{\phi}_{\bsi}=\sqrt{\Delta^{\bsi}}\phi_{\bsi} :\bsi\in I_\cX\}$ which is self-dual: $\{ \hat{\phi}^{\bsi}=\hat{\phi}_{\bsi}: \bsi \in I_\cX \}$.
\end{enumerate}

For $\bsi, \bsi'\in I_\cX$, define
$$
S^{\bsi'}_{\spa \bsi}(z) := (\phi^{\bsi'}, \cS(\phi_{\bsi})).
$$
Then $(S^{\bsi'}_{\spa \bsi}(z))$ is the matrix  of the $\cS$-operator with respect to the canonical basis
$\{\phi_{\bsi}:\bsi\in I_\cX \}$:
\begin{equation}\label{eqn:S}
\cS(\phi_{\bsi}) =\sum_{\bsi'\in I_\cX}
\phi_{\bsi'} S^{\bsi'}_{\spa \bsi}(z).
\end{equation}

For $\bsi,\bsi'\in I_\cX$, define
$$
S_{\bsi'}^{\spa \widehat{\bsi} }(z) := (\phi_{\bsi'}, \cS(\hat{\phi}^{\bsi})).
$$
Then $(S_{\bsi'}^{\spa  \widehat{\bsi}})$ is the matrix of the $\cS$-operator
with respect to the basis $\{\hat{\phi}^{\bsi}:\bsi\in I_\cX\}$ and
$\{\phi^{\bsi}: \bsi\in I_{\cX}\}$:
\begin{equation}\label{eqn:barS}
\cS(\hat{\phi}^{\bsi})=\sum_{\bsi'\in I_{\cX}} \phi^{\bsi'}
 S_{\bsi'}^{\spa \widehat{\bsi}}(z).
\end{equation}

Introduce
\begin{align*}
S_z(a,b)&=(a,\cS(b))_{\cX,\bT},\\
V_{z_1,z_2}(a,b)&=\frac{(a,b)_{\cX,\bT}}{z_1+z_2}+\ll \frac{a}{z_1-\hat{\psi}_1},
                  \frac{b}{z_2-\hat{\psi}_2}\gg^{\cX,\bT}_{0,2}.
\end{align*}
The following identity is known (see e.g. \cite{GT}, \cite{Gi97}):
\begin{equation}
\label{eqn:two-in-one}
V_{z_1,z_2}(a,b)=\frac{1}{z_1+z_2}\sum_i S_{z_1}(T_i,a)S_{z_2}(T^i,b),
\end{equation}
where $T_i$ is any basis of $H^*_{\CR,\bT}(\cX;\bST)$ and $T^i$ is its dual basis.
In particular,
$$
V_{z_1,z_2}(a,b)=\frac{1}{z_1+z_2}\sum_{\bsi\in I_\cX} S_{z_1}(\hat{\phi}_{\bsi},a)S_{z_2}(\hat{\phi}_{\bsi},b).
$$

\subsection{The A-model $R$-matrix}
Let $U$ denote the diagonal matrix whose diagonal entries are the canonical coordinates.
The results in \cite{Gi97, Gi96b} and \cite{Zo} imply the following statement.
\begin{theorem}\label{R-matrix}
There exists a unique matrix power series $R(z)= \one + R_1z+R_2 z^2+\cdots$
satisfying the following properties.
\begin{enumerate}
\item The entries of $R_d$ lie in $\bSTQ$.
\item $\tS=\Psi R(z) e^{U/z}$  is a fundamental solution to the $\bT$-equivariant
big QDE \eqref{eqn:qde}.
\item $R$ satisfies the unitary condition $R^T(-z)R(z)=\one$.
\item
\begin{equation}\label{eqn:R-at-zero}
\lim_{\tfQ,\tau''\to 0} R_{\rho,\delta}^{\spa\si,\gamma}(z)
= \frac{\delta_{\rho,\si}}{p}\sum_{h\in \bZ_p}\chi_\delta(h) \chi_\gamma(h^{-1})
\prod_{i=1}^3 \exp\Big( \sum_{m=1}^\infty \frac{(-1)^m}{m(m+1)}B_{m+1}(c_i(h))
(\frac{z}{\sw_i(\si)})^m \Big),
\end{equation}
where $B_m(x)$ is the $m$-th Bernoulli polynomial, defined by the following identity:
$$
\frac{te^{tx}}{e^t-1}=\sum_{m\geq 0}\frac{B_m(x)t^m}{m!}.
$$
\end{enumerate}
\end{theorem}

Each matrix in (2) of Theorem \ref{R-matrix} represents an operator with respect to the classical canonical basis $\{ \hat{\phi}_{\bsi}: \bsi\in I_\cX\}$.
So $R^T$ is the adjoint of $R$ with respect to the $\bT$-equivariant
Poincar\'{e} pairing $(\ , \ )_{\cX,\bT}$.
The matrix $(\tS_{\bsi'}^{\spa \widehat{\bsi}})(z)$ is of the form
$$
\tS_{\bsi'}^{\spa\widehat{\bsi}}(z)
= \sum_{\bsi''\in I_\cX} \Psi_{\bsi'}^{\spa \bsi''}
R_{\bsi''}^{\spa \bsi}(z) e^{u^{\bsi}/z}
=(\Psi R(z))_{\bsi'}^{\spa \bsi} e^{u^{\bsi}/z}
$$
where $R(z)= (R_{\bsi''}^{\spa\bsi}(z)) = \one + \sum_{k=1}^\infty R_k z^k$.

We call the unique $R(z)$ in Theorem \ref{R-matrix} the {\em A-model $R$-matrix}.
The A-model $R$-matrix plays a central role in the quantization formula of the descendent potential of $\bT$-equivariant Gromov-Witten
theory of $\cX$. We will state this formula in terms of graph sum in the the next subsection.

\subsection{Graph sum formula for descendent Gromov-Witten potentials}\label{sec:A}
\begin{itemize}
\item We define
$$
S^{\widehat{\underline{\bsi}}}_{\spa \widehat{\underline{\bsi'}} }(z)
:= (\hat{\phi}_{\bsi}(t), \cS(\hat{\phi}_{\bsi'}(t))).
$$
Then $(S^{ \widehat{\underline{\bsi}}  }_{\spa \widehat{\underline{\bsi'}} }(z))$ is the matrix of the $\cS$-operator with
respect to the normalized canonical basis  $\{ \hat{\phi}_{\bsi}(t): \bsi\in I_\cX\} $:
\begin{equation}
\cS(\hat{\phi}_{\bsi'}(t))=\sum_{\bsi\in I_\cX} \hat{\phi}_{\bsi}(t)
S^{\widehat{\underline{\bsi}} }_{\spa \widehat{\underline{\bsi'}} }(z).
\end{equation}
\item We define
$$
S^{\widehat{\underline{\bsi}}}_{\spa \bsi'}(z)
:= (\hat{\phi}_{\bsi}(t), \cS(\phi_{\bsi'})).
$$
Then $(S^{ \widehat{\underline{\bsi}}  }_{\spa \bsi'}(z))$ is the matrix of the $\cS$-operator with
respect to the  basis $\{\phi_{\bsi}:\bsi\in I_\cX\}$ and
$\{ \hat{\phi}_{\bsi}(t): \bsi\in I_\cX\} $:
\begin{equation}
\cS(\phi_{\bsi'})=\sum_{\bsi\in I_\cX} \hat{\phi}_{\bsi}(t)
S^{\widehat{\underline{\bsi}} }_{\spa \bsi'}(z).
\end{equation}
\end{itemize}

Given a connected graph $\Ga$, we introduce the following notation.
\begin{enumerate}
\item $V(\Ga)$ is the set of vertices in $\Ga$.
\item $E(\Ga)$ is the set of edges in $\Ga$.
\item $H(\Ga)$ is the set of half edges in $\Ga$.
\item $L^o(\Ga)$ is the set of ordinary leaves in $\Ga$. The ordinary
leaves are ordered: $L^o(\Ga)=\{l_1,\ldots,l_n\}$ where
$n$ is the number of ordinary leaves.
\item $L^1(\Ga)$ is the set of dilaton leaves in $\Ga$. The dilaton leaves are unordered.
\end{enumerate}

With the above notation, we introduce the following labels:
\begin{enumerate}
\item (genus) $g: V(\Ga)\to \bZ_{\geq 0}$.
\item (marking) $\bsi: V(\Ga) \to I_\cX$. This induces
$\bsi :L(\Ga)=L^o(\Ga)\cup L^1(\Ga)\to I_\cX$, as follows:
if $l\in L(\Ga)$ is a leaf attached to a vertex $v\in V(\Ga)$,
define $\bsi(l)=\bsi(v)$.
\item (height) $k: H(\Ga)\to \bZ_{\geq 0}$.
\end{enumerate}

Given an edge $e$, let $h_1(e),h_2(e)$ be the two half edges associated to $e$. The order of the two half edges does not affect the graph sum formula in this paper. Given a vertex $v\in V(\Ga)$, let $H(v)$ denote the set of half edges
emanating from $v$. The valency of the vertex $v$ is defined to be the cardinality of the set $H(v)$: $\val(v)=|H(v)|$.
A labeled graph $\vGa=(\Ga,g,\bsi,k)$ is {\em stable} if
$$
2g(v)-2 + \val(v) >0
$$
for all $v\in V(\Ga)$.

Let $\bGa(\cX)$ denote the set of all stable labeled graphs
$\vGa=(\Gamma,g,\bsi,k)$. The genus of a stable labeled graph
$\vGa$ is defined to be
$$
g(\vGa):= \sum_{v\in V(\Ga)}g(v)  + |E(\Ga)|- |V(\Ga)|  +1
=\sum_{v\in V(\Ga)} (g(v)-1) + (\sum_{e\in E(\Gamma)} 1) +1.
$$
Define
$$
\bGa_{g,n}(\cX)=\{ \vGa=(\Gamma,g,\bsi,k)\in \bGa(\cX): g(\vGa)=g, |L^o(\Ga)|=n\}.
$$

We assign weights to leaves, edges, and vertices of a labeled graph $\vGa\in \bGa(\cX)$ as follows.
\begin{enumerate}
\item {\em Ordinary leaves.}
To each ordinary leaf $l_j \in L^o(\Ga)$ with  $\bsi(l_j)= \bsi\in I_\Si$
and  $k(l)= k\in \bZ_{\geq 0}$, we assign the following descendant  weight:
\begin{equation}\label{eqn:u-leaf}
(\cL^{\bu})^{\bsi}_k(l_j) = [z^k] (\sum_{\bsi',\bsi''\in I_\cX}
\left(\frac{\bu_j^{\bsi'}(z)}{\sqrt{\Delta^{\bsi'}(t)} }
S^{\widehat{\underline{\bsi''}} }_{\spa
  \widehat{\underline{\bsi'}}}(z)\right)_+ R(-z)_{\bsi''}^{\spa \bsi} ),
\end{equation}
where $(\cdot)_+$ means taking the nonnegative powers of $z$.

\item {\em Dilaton leaves.} To each dilaton leaf $l \in L^1(\Ga)$ with $\bsi(l)=\bsi
\in I_\cX$
and $2\leq k(l)=k \in \bZ_{\geq 0}$, we assign
$$
(\cL^1)^{\bsi}_k := [z^{k-1}](-\sum_{\bsi'\in I_\cX}
\frac{1}{\sqrt{\Delta^{\bsi'}(t)}}
R_{\bsi'}^{\spa \bsi}(-z)).
$$

\item {\em Edges.} To an edge connecting a vertex marked by $\bsi\in I_\cX$ and a vertex
marked by $\bsi'\in I_\cX$, and with heights $k$ and $l$ at the corresponding half-edges, we assign
$$
\cE^{\bsi,\bsi'}_{k,l} := [z^k w^l]
\Bigl(\frac{1}{z+w} (\delta_{\bsi\bsi'}-\sum_{\bsi''\in I_\cX}
R_{\bsi''}^{\spa \bsi}(-z) R_{\bsi''}^{\spa \bsi'}(-w)\Bigr).
$$
\item {\em Vertices.} To a vertex $v$ with genus $g(v)=g\in \bZ_{\geq 0}$ and with
marking $\bsi(v)=\bsi$, with $n$ ordinary
leaves and half-edges attached to it with heights $k_1, ..., k_n \in \bZ_{\geq 0}$ and $m$ more
dilaton leaves with heights $k_{n+1}, \ldots, k_{n+m}\in \bZ_{\geq 0}$, we assign
$$
 \Big(\sqrt{\Delta^{\bsi}(t)}\Big)^{2g(v)-2+\val(v)}\langle  \tau_{k_1}\cdots\tau_{k_{n+m}}\rangle_g,
$$
where $\langle  \tau_{k_1}\cdots\tau_{k_{n+m}}\rangle_{g}=\int_{\Mbar_{g,n+m}}\psi_1^{k_1} \cdots \psi_{n+m}^{k_{n+m}}$.
\end{enumerate}

We define the weight of a labeled graph $\vGa\in \bGa(\cX)$ to be
\begin{eqnarray*}
w_A^{\bu}(\vGa) &=& \prod_{v\in V(\Ga)} \Bigl(\sqrt{\Delta^{\bsi(v)}(t)}\Bigr)^{2g(v)-2+\val(v)} \langle \prod_{h\in H(v)} \tau_{k(h)}\rangle_{g(v)}
\prod_{e\in E(\Ga)} \cE^{\bsi(v_1(e)),\bsi(v_2(e))}_{k(h_1(e)),k(h_2(e))}\\
&& \cdot \prod_{l\in L^1(\Ga)}(\cL^1)^{\bsi(l)}_{k(l)}\prod_{j=1}^n(\cL^{\bu})^{\bsi(l_j)}_{k(l_j)}(l_j).
\end{eqnarray*}

With the above definition of the weight of a labeled graph, we have
the following theorem which expresses the $\bT$-equivariant descendent
Gromov-Witten potential of $\cX$ in terms of graph sum.

\begin{theorem}[{\cite{Zo}}]
\label{thm:Zong}
 Suppose that $2g-2+n>0$. Then
$$
\ll \bu_1,\ldots, \bu_n\gg_{g,n}^{\cX,\bT}=\sum_{\vGa\in \bGa_{g,n}(\cX)}\frac{w_A^{\bu}(\vGa)}{|\Aut(\vGa)|}.
$$
\end{theorem}

\subsection{Equivariant $J$-function}\label{section-J-function}
The equivalent $J$-function is a cohomology-valued function defined by the $\cS$-operator.
\begin{defn}
    [$\bT$-equivariant big $J$-function] \rm
    The $\bT$-equivariant big $J$-function $J^\Bigbig_\bT(z)$ is characterized by
    $$
        (J^\Bigbig_\bT(z),a)_{\cX,\bT} = (1,\cS(a))_{\cX,\bT}
    $$
    for any $a\in H_{\CR,\bT}^*(\cX;\bST)$. Equivalently, let $\{T_i\}_{i=1}^{2p}$ be a homogeneous basis of $H_{\CR,\bT}^*(\cX;\bST)$
    and $\{T^i\}_{i=1}^{2p}$ be its dual basis, then
    \[
        \Jbig = 1 + \sum_{i=1}^{2p} \ll 1,\frac{T_i}{z-\hat{\psi}}\gg_{0,2}^{\X,\bT}T^i
    \]
\end{defn}

The $\bT$-equivariant small $J$-function $J_\bT(\btau,z)$ is the restriction of the $\bT$-equivariant big $J$-function to the small
phase space. More precisely, given $\btau\in H_{\CR,\bT}^2(\cX;\bST)$, let
\[
    J_\bT(\btau,z) := \Jbig|_{t=\btau, \fQ=1}
\]
\[
    J_\bT(\btau,z) = 1 + \sum_{l=0}^\infty \sum_{d\in E(\X)} \frac{1}{l!}\sum_{i=1}^{2p}
    \langle 1,\btau^l, \frac{T_i}{z-\hat{\psi}} \rangle_{0,l+2,d}^{\X,\bT}T^i
\]
Let $G_\sigma = \Zp$ be the stabilizer of the stacky point $\fp_\sigma$. We index the elements in $\Zp$ by the set of fractions
\[
    \cI = \{h\in\bQ | h = \frac{i}{p}, 0\leq i <p\}
\]
via $h\in\cI\mapsto [ph]\in\Zp$. The weights of $\bT$-action on $T_{\fp_0}\X$ are given by
\[
    \sw_{1,0} = \frac{1}{p}\su_1, \quad \sw_{2,0} = \su_2, \quad \sw_{3,0} = -\frac{1}{p}\su_1 - \su_2
\]
We will pullback the small $J$-function $J_\bT$ to the chart $\X_0$, which contains the $\T$-invariant holomorphic disk $[D_\epsilon/\Zp]$.
Hence we will specifically consider the sub-torus $\T= \ker(k\su_1+r\su_2)$. Define the morphism $\iota_{r,s}^*: H^*_{\CR,\bT}(\X)
\rightarrow H^*_{\CR,\T}(\X)$ by:
\[
    \su_1\mapsto r\sv, \quad \su_2\mapsto -k\sv.
\]
The weights of $\T$-action on $T_{\fp_0}\X$ are $w_1\sv$, $w_2\sv$, $w_3\sv$, where
\[
    w_1 = \frac{r}{p},\quad w_2 = -\frac{r+s}{p} = -k , \quad w_3 = \frac{s}{p}.
\]
Given $h\in G_\si$, define $\one_{\sigma,h}^* = p\iota_{r,s}^*\iota_{\sigma*}\one_{h^{-1}}$. Take $\sigma = 0$, we get
\[
    \iota_0^*J_\bT(\btau,z)|_{\su_1 = r\sv, \su_2 = -k\sv} = \sum_{h\in\cI}J_{0,h}(\btau,z)\one_{h}
\]
where
\[
    J_{0,h}(\btau,z) = \delta_{h,0} + \sum_{l=0}^\infty \sum_{d\in E(\X)}
    \frac{1}{l!}\DescdentGwol{1, (\iota_{r,s}^*\btau)^l,\frac{\one_{0,h}^*}{z-\hat{\psi}}}{l+2}
\]

\subsection{Equivariant $I$-function and genus zero mirror symmetry}
In this section, we will define the equivariant $I$-function of $\X$ following \cite{FLT22}.
Recall the fan $\Sigma$ in Section \ref{sec:equivariantCR} (see Figure \ref{fig:fan}), by the construction of $\Sigma$,
there is a surjective group homomorphism
\[
    \phi: \tilde{N} := \bigoplus_{i=1}^{p+3} \bZ\tilde{b_i} \rightarrow N, \quad \tilde{b_i}\mapsto b_i
\]
Let $\bL =\ker(\phi)\cong \bZ^p$. Then we have a short exact sequence of abelian groups
\[
    0\rightarrow \mathbb{L}\xrightarrow{\psi} \mathbb{Z}^{p+3} \xrightarrow{\phi} \mathbb{Z}^3 \rightarrow 0.
\]
Let $\{e_1,\dots,e_p\}$ be a $\bZ$-basis of $\bL\cong\bZ^p$. Under the basis of $\bL, \tilde{N}$ and $N$, we get
\begin{equation*}
    \phi =\left[
    \begin{array}{cccccccc}
        p & 0 & 0 & p & 1 & 2 & \dots  & p-1\\
        0 & 1 & 0 &-1 & 0 & 0 & \dots  & 0\\
        1 & 1 & 1 & 1 & 1 & 1 & \dots  & 1
    \end{array}
    \right]
\end{equation*}
\begin{equation*}
    \psi =
    \left[
    \begin{array}{ccccc}
        -1 & -1 & -2 & \dots & -p+1\\
         1 & 0 & 0 & \dots & 0\\
        -1 & -p+1 & -p+2 & \dots & -1 \\
         1 & 0 & 0 & \dots & 0 \\
         0 & p & 0 & \dots & 0 \\
         0 & 0 & p & \dots & 0 \\
        \vdots & \vdots & \vdots & \ddots & \vdots \\
         0 & 0 & 0 & \dots & 0 \\
         0 & 0 & 0 & \dots & p
    \end{array}
    \right]
\end{equation*}
Let $\{e_1^\vee,\dots,e_p^\vee\}$ be the dual $\bZ$-basis of $\bL^\vee$, and define $D_i\in\bL^\vee$ as row vectors of $\psi$:
\[
    D_1 = (-1,-1,\dots,-p+1), ~ D_3 = (-1,-p+1,\dots, -1),
\]
\[
    D_2 = D_4 = (1, 0,\dots, 0),
\]
\[
    D_5 = (0,p,0,\dots,0),~\dots, ~D_{p+3}=(0,\dots,0,p).
\]
Let $\bL_\bC:= \bL\otimes_\bZ \bC$. Then there is an isomorphism of vector spaces over $\bC$:
\[
    H^2(\X;\bC)\cong \bL_\C^\vee/\oplus_{i=5}^{p+3}\bC D_i.
\]
We denote the image of $\{D_i\}_{i=1}^4$ in $H^2(\X;\bC)$ by $\bar{D}_i$.
Let $\mathbb{K}_\eff$ be the lattice in the extended Mori-cone of $\X$. In our case,
\[
    \mathbb{K}_\eff = \mathbb{Z}_{\geq 0}e_1 \oplus \frac{1}{p}\mathbb{Z}_{\geq 0}e_2 \oplus\dots \oplus \frac{1}{p}\mathbb{Z}_{\geq 0}e_{p} \subset \bL_\bQ:=\bL\otimes_\bZ \bQ.
\]
Let $\beta = (a_1,\dots, a_n)\in\mathbb{K}_\eff$, $a_i\in\mathbb{Z}$, represent the element
\[
    \beta = a_1e_1 + \frac{a_2}{p}e_2 +\dots + \frac{a_p}{p}e_p \in \mathbb{K}_\eff,
\]
and let
\[
    w(\beta) := \frac{1}{p}\sum_{m=1}^{p-1} ma_{m+1}.
\]
Then
\[
    \langle D_1, \beta\rangle = -a_1 - w(\beta),~ \langle D_3, \beta\rangle = -a_1 - a_2 -\dots -a_p + w(\beta)
\]
\[
    \langle D_2,\beta\rangle = \langle D_4,\beta\rangle = a_1,~ \langle D_5,\beta\rangle = a_2,~\dots, ~\langle D_{p+3},\beta\rangle =a_p.
\]
Let
\[
    \text{Box}(\X) := \{(0,0,0)\}\cup\{b_i\in N|5\leq i\leq p+3\}.
\]
Then there is a bijection between $\text{Box}(\X)$ and $G_\si$, $\si=0,1$.
As a graded vector space over $\bC$,
\[
    H^*_{\CR,\bT}(\X) = \bigoplus_{v\in \text{Box}(\X)}H^*_{\bT}(\X_v)[2\age(v)].
\]
The elements $\bar{D}_i\in H^2(\X)$ can be lifted as elements $\bar{D}^\bT_i$ in $H^2_{\CR,\bT}(\X)$.
Let
$\bar{H}^\bT:= \bar{D}_4^\bT\in H^2_{\bT}(\X)$ be the unique $\bT$-equivariant lifting of $\bar{D}_4\in H^2(\X)$ such that $\bar{H}^\bT|_{\fp_0}=0$.
For $i=1,2,3$, $\bar{D}_i^\bT = \sw_{i,0} + (-1)^i\bar{H}^\bT$. Let $\one_v$ be the unit in $H^*_{\bT}(\X_v)$.
Then
\[
    H^{2}_{\CR,\bT}(\X) = \bC\su_1 \oplus\bC\su_2 \oplus\bC \bar{H}^\bT\oplus \bigoplus_{a=2}^p \bC\one_{b_{a+3}}.
\]
Any $\btau\in H_{\CR,\bT}^{2}(\X)$ can be written as
\[
    \btau = \tau_0 + \btau_2 = \tau_0 + \tau_1\bar{H}^\bT + \sum_{a=2}^p\tau_a\one_{b_{a+3}},
\]
where $\tau_0\in H^*_\bT(\pt) = \bC\su_1 \oplus\bC\su_2$ and $\tau_1,\dots,\tau_p\in\bC$.
Let $(q_0',q_1,\dots,q_p)$ be formal variables, the $I$-function of $\X$ is
\begin{equation*}
    \begin{aligned}
        I_\bT(q_0',q,z) = &e^{\frac{\log q_0' + \bar{H}^\bT\log q_1}{z}}
        \sum_{\beta\in\mathbb{K}_\eff} \frac{q_1^{a_1}\dots q_p^{a_p}}{z^{1-\delta_{\langle w(\beta)\rangle,0}}}
        \frac{\prod_{m=-a_1+\lceil -w(\beta)\rceil}^{-1}(\frac{\bar{D}^\bT_1}{z}-(a_1+\wbeta+m))}{\prod_{m=0}^{a_1-1}(\frac{\bar{D}^\bT_2}{z}+a_1 -m)}
        \\
        &\cdot\frac{
            \prod_{m=-a_1-a_2-\dots-a_p +\lceil w(\beta)\rceil}^{-1}(\frac{\bar{D}^\bT_3}{z}-(a_1+\dots+a_p+m)+\wbeta)}
        {\prod_{m=0}^{a_1-1}(\frac{\bar{D}_4^\bT}{z}+a_1 -m)a_2!\dots a_p!}\one_{v(\beta)}
    \end{aligned}
\end{equation*}
where $v(\beta) = \langle w(\beta)\rangle b_1 + \langle -w(\beta)\rangle b_3$.
There is a $\bT$-equivariant mirror theorem from \cite{CCIK15,CCIT15}:
\begin{theorem}\label{thm:mirror theorem}
\[
    e^{\frac{\tau_0(q_0',q)}{z}}J_\bT(\btau_2(q),z) = I_\bT(q_0',q,z),
\]
where the equivariant closed mirror map $(q_0',q)\mapsto \btau(q_0',q)$ is determined by the first-order term in the asymptotic expansion
of the $I$-function
\[
    I_\bT(q_0',q,z) = 1+ \frac{\btau(q_0',q)}{z} + o(z^{-1}).
\]
More explicitly, the equivariant closed mirror map is given by
\[
    \btau = \log(q_0') + \log(q_1)\bar{H}^\bT + \sum_{a=2}^p \tau_a(q)\one_{b_{a+3}},
\]
where $\tau_0(q_0', q) = \log (q_0')$, $\tau_1(q) = \log(q_1)$, and
for $2\leq a\leq p$
\begin{equation*}
    \begin{aligned}
        \tau_a(q) =& \sum_{a_1,\dots,a_p,l\geq 0\atop a_2+\dots + (p-1)a_p = pl+a-1}
    \frac{
        q_1^{a_1}\dots q_p^{a_p}}{
            (a_1!)^2
            a_2!\dots a_p!
        }\prod_{m=-a_1-l}^{-1}(-a_1-l-\frac{a-1}{p}-m)
        \\
    \cdot&\prod_{m=-a_1-\dots-a_p+l+1}^{-1}(-a_1-\dots-a_p+l+\frac{a-1}{p}-m).
    \end{aligned}
\end{equation*}
\end{theorem}
Now let
\[
    \iota_0^*\bar{D}^\bT_1= \frac{r}{p}\sv, ~\iota_0^*\bar{D}^\bT_2= -k\sv,
    ~\iota_0^*\bar{D}^\bT_3= \frac{s}{p}\sv, ~\iota_0^*\bar{D}^\bT_4= 0,
\]
\[
    \iota_0^*I_\bT(q_0',q,z) = \sum_{h\in\cI} I_{0,h}(q_0',q,z)\one_h.
\]
By mirror theorem, we have
\begin{equation*}
    \begin{aligned}
        J_{0,h}(\btau_2(q),z) &= e^{-\frac{\log q_0'}{z}}I_{0,h}(q_0',q,z)
        \\
        &= \sum_{\beta\in\mathbb{K}_\eff\atop \langle w(\beta)\rangle = h} \frac{q_1^{a_1}\dots q_p^{a_p}}{z^{1-\delta_{h,0}}}
        \frac{\prod_{m=-a_1+\lceil -w(\beta)\rceil}^{-1}(\frac{r\sv}{pz}-(a_1+\wbeta+m))}{\prod_{m=0}^{a_1-1}(\frac{-k\sv}{z}+a_1 -m)}
        \\
        \cdot&\frac{
            \prod_{m=-a_1-a_2-\dots-a_p +\lceil w(\beta)\rceil}^{-1}(\frac{s\sv}{pz}-(a_1+\dots+a_p+m)+\wbeta)}
        {a_1!a_2!\dots a_p!}.
    \end{aligned}
\end{equation*}

\subsection{Open-closed Gromov-Witten invariants}\label{sec:open-closed-GW}
Consider the pair of the orbifold resolved conifold $\X$ with Lagrangian $\Lrs$ constructed in Section \ref{sec:conifold transition}.
Given any partition $\vec{\mu}$ with $l(\vec{\mu})=n$ and $g,d\in\bZ_{\geq 0}$, we define
\[
    \overline{\cM}_{g,l,d,\vec{\mu}} := \overline{\cM}_{g,l;n}(\X, \Lrs| d; \mu_1, \dots, \mu_n)
\]
to be the moduli space of stable maps $u:(\Sigma, x_1,\dots,x_l,\partial\Sigma)\rightarrow(\X,\Lrs)$.
Here $\Sigma$ is a prestable genus $g$ bordered Riemann surface with interior stacky marked points $x_i = B\bZ_{r_i},i=1,\cdots,l$,
and $\partial\Sigma$ has $n$ connected components $R_j\cong S^1,j=1,\cdots,n$. We require that $u_*[\Sigma] = d+ (\sum_{j=1}^{n}\mu_j)b\in H_2(\cX,\Lrs)$ and $u_*([R_j])=\mu_jb\in H_1(\Lrs;\bZ)$, where $b$ is the generator of $H_1(\Lrs;\bZ)\cong\bZ$.
\par
As discussed in Section \ref{sec:conifold transition}, the $(\T)_{\bR}$-action preserves the pair $(\X,\Lrs)$, so it induces an action
on $\Mgldmu$. The fixed locus $\Mgldmu^{(\T)_\bR}$ of this action consists of the maps
\[
    f: D_1 \cup \dots\cup D_n\cup C \rightarrow (\X, \Lrs),
\]
Here $f|_{D_i}$ is a $\T$-invariant holomorphic morphism from the orbidisk $[\{z\in\bC | |z|\leq b_1\}/\bZ_{r_i}]$ and
$C$ is a (possibly empty) closed nodal curve and $f|_C$
is also a $\T$-invariant morphism. The point $z=0$ on each $D_i$ is the only possible nodal point on $D_i$.
\par
The $(\T)_{\bR}\cong S^1$-fixed locus $\Mgldmu^{(\T)_{\bR}}$ is compact. Given $\gamma_1,\dots,\gamma_l \in \HcrX{*}$, we define
\[
      \GWgl{\gamma_1,\dots,\gamma_l} = \int_{[\Mgldmu^{(\T)_{\bR}}]^\vir} \frac{
        \prod_{j=1}^{l} \ev_j^*\gamma_j
      }{
        e_{(\T)_{\bR}}(N^\vir)
      }
\]
where $N^\vir$ is the virtual normal bundle of $\Mgldmu^{(\T)_{\bR}}\subset \Mgldmu$.
\par
Let $\btau\in\HcrX{*}$, we define the A-model open Gromov-Witten potential associated to $(\X,\Lrs)$ as
\begin{equation}
    \label{eqn:open-potential}
    F^{\X,\Lrs}_{g,n}(\btau, \fQ; X_1,\dots,X_n)= \sum_{l=0}^\infty\sum_{d\in E(\X)}\sum_{\mu_1,\dots,\mu_n>0}
    \frac{\GWgl{\btau^l}}{l!}\fQ^dX_1^{\mu_1}\dots X_n^{\mu_n}
\end{equation}
Define the disk factor $D'(\mu)$ for $\mu\in\bZ_{>0}$
\begin{equation}\label{eqn:diskfactor}
        D'(\mu) = \frac{p}{\mu}\left(\frac{\sv}{\mu}\right)^{1-\delta_{\AlphaMu,0}}\frac{\prod_{j=1}^{\mu k-1}(-\mu k+j)}
        {\lfloor\frac{\mu r}{p}\rfloor!\lfloor\frac{\mu s}{p}\rfloor!}.
\end{equation}
By the same proof as \cite[Proposition 3.3]{FLT22}, we have the following formula
\begin{theorem}\label{thm:localization}
    Let $D'(\mu)$ be as in \eqref{eqn:diskfactor}, and let $\gamma_i\in H^*_{\CR,\T}(\X)$, $\vec{\mu} = (\mu_1, \dots, \mu_n)$. Then
    \[
        \langle \gamma_1,\dots,\gamma_l\rangle_{g,l,d,\vec{\mu}}^{\X,\Lrs,\T} = \prod_{j=1}^{n}D'(\mu_j)
        \cdot \int_{[\overline{\mathcal{M}}_{g,n+l}(\X,d)]^{\vir}}\frac{\prod_{i=1}^l \ev^*_i\gamma_i\prod_{j=1}^{n}\ev^*_{l+j}\rho_{0, \langle-\frac{\mu_j r}{p}\rangle}}
        {\prod_{j=1}^{n}\frac{\sv}{\mu_j}(\frac{\sv}{\mu_j} - \hat{\psi}_{l+j})}
    \]
    where $\rho_{0,h} = \iota_{r,s}^*\iota_{0_*}\one_h\in H^*_{\CR,\T}(\X)$.
\end{theorem}

\noindent
We introduce some notation.
\begin{itemize}
    \item [(1)] Let $\cI$ be the set of fraction to index $G_0=G_{\si_0}$ in Section \ref{section-J-function}. Given $h\in\cI$, define
    \begin{equation*}
        \begin{aligned}
            \Phi_0^h(X) &:= \frac{1}{p}\sum_{\mu> 0\atop \AlphaMu=h} D'(\mu)\left(\frac{\mu}{\sv}\right)^{2}X^\mu
            \\
            &= \sum_{\mu> 0\atop \AlphaMu=h} \frac{1}{\mu}\left(\frac{\sv}{\mu}\right)^{-1-\delta_{h,0}}\frac{\prod_{j=1}^{\mu k-1}(-\mu k+j)}
            {\lfloor\frac{\mu r}{p}\rfloor!\lfloor\frac{\mu s}{p}\rfloor!}X^\mu
        \end{aligned}
    \end{equation*}
    For $a\in\bZ$ and $h\in\cI$, we define
    \[
        \Phi_a^h(X) := \frac{1}{p}\sum_{\mu> 0\atop \AlphaMu=h} D'(\mu)\left(\frac{\mu}{\sv}\right)^{a+2}X^\mu.
    \]
    i.e.
    \[
        \Phi_{a+1}^h(X) = (\frac{1}{\sv}X\frac{d}{dX})\Phi_a^h(X).
    \]
    \item[(2)] For $a\in\bZ$ and $\alpha\in G_0^*$, we define
    \[
        \widetilde{\xi}^\alpha_a(X) := p\sum_{[h]\in G_0} \chi_\alpha(-[h])(\prod_{i=1}^{3}(w_i\sv)^{1-c_i([h])})\Phi_a^{\langle\frac{h}{p}\rangle}(X)
    \]
    where $c_1([h]) = \langle\frac{h}{p}\rangle$, $c_2(h) = 0$, $c_3(h)=1-\langle\frac{h}{p}\rangle - \delta_{0,\langle \frac{h}{p}\rangle}$. Furthermore, we define
    \[
        \widetilde{\xi}^\alpha(z,X) := \sum_{a\in\bZ_{\geq -2}}z^a\widetilde{\xi}^\alpha_a(X).
    \]
\end{itemize}

With the above notation, by \cite[Proposition 3.13]{FLZ16} we have:
\begin{prop}\label{prop:a-potential}\rm
    \begin{enumerate}
        \item [(1)](disk invariants)
        \begin{equation*}
            F^{\X,\Lrs}_{0,1}(\btau, \fQ; X) = [z^{-2}] \sum_{\alpha\in G_0^*}S_z(1,\phi_{0,\alpha})\widetilde{\xi}^\alpha(z,X)\big|_{t=\btau\atop\sw_i=w_i\sv}.
        \end{equation*}
        \item[(2)](annulus invariants)
        \begin{equation*}
            \begin{aligned}
                &\quad\quad F^{\X,\Lrs}_{0,2}(\btau, \fQ; X_1, X_2) - F^{\X,\Lrs}_{0,2}(0; X_1, X_2)
                \\
                &=[z_1^{-1}z_2^{-1}]\sum_{\alpha_1,\alpha_2\in G_0^*}\ll\frac{\phi_{0,\alpha_1}}{z_1-\hat{\psi_1}},\frac{\phi_{0,\alpha_2}}{z_2-\hat{\psi_2}}\gg_{0,2}^{\X,\T}\big|_{t=\btau}\widetilde{\xi}^{\alpha_1}(z_1,X_1)\widetilde{\xi}^{\alpha_2}(z_2,X_2)
            \end{aligned}
        \end{equation*}
        where
        \begin{equation*}
            \begin{aligned}
            (X_1\frac{\partial}{\partial X_1}+X_2\frac{\partial}{\partial X_2})F^{\X,\Lrs}_{0,2}(0; X_1, X_2)
            \\  = \frac{1}{p^2w_1w_2w_3}\bigg(\sum_{\gamma\in G_0^*}(
                \widetilde{\xi}_0^\gamma(X_1)\widetilde{\xi}_0^\gamma(X_2)
            )\bigg)\big|_{\sv=1}
            \end{aligned}
        \end{equation*}
        \item[(3)] For $2g-2+n>0$
        \begin{equation*}
            \begin{aligned}
                &F^{\X,\Lrs}_{g,n}(\btau, \fQ; X_1,\dots, X_n)
                \\
                = \ &[z_1^{-1}\dots z_n^{-1}]\sum_{\alpha_1,\dots,\alpha_n\in G_0^*}\left( \ll\frac{\phi_{0,\alpha_1}}{z_1-\hat{\psi}_1},
                \frac{\phi_{0,\alpha_2}}{z_2-\hat{\psi}_2},\dots, \frac{\phi_{0,\alpha_n}}{z_n-\hat{\psi}_n} \gg^{\cX,\T}_{g,n}\bigg|_{t=\btau}\right)\prod_{j=1}^{n}\widetilde{\xi}^{\alpha_j}(z_j,X_j).
            \end{aligned}
        \end{equation*}
    \end{enumerate}

\end{prop}

To obtain the graph sum formula for $F^{\X,\Lrs}_{g,n}$, we introduce the notation
\begin{equation*}
    \widetilde{\xi}^{\bsi}(z,X) = \left\{
    \begin{aligned}
        \widetilde{\xi}^\alpha(z,X), \quad &\text{if } \bsi = (0,\alpha),
        \\
        0, \spa\spa\spa\spa\spa &\text{if } \bsi = (1,\alpha).
    \end{aligned}
    \right.
\end{equation*}
\begin{itemize}
    \item [$\bullet$] Given a labeled graph $\vec{\Gamma}\in\bGa_{g,n}(\X)$, to each ordinary leaf $l_j\in L^o(\Gamma)$ with $\bsi(l_j)=\bsi\in I_\X$ and $k(l_j)\in\bZ_{\geq 0}$
    we assign the following weight (open leaf)
    \[
        (\widetilde{\cL}^O)^\bsi_k(l_j) = [z^k] \left(\sum_{\bsi',\bsi''\in I_\cX}
        \left(\widetilde{\xi}^{\bsi'}(z,X_j)
        S^{\widehat{\underline{\bsi''}} }_{\spa
          \bsi'}(z)\left|_{t=\btau\atop \sw_i=w_i\sv}\right)_+ R(-z)_{\bsi''}^{\spa \bsi} \right|_{t=\btau\atop \sw_i=w_i\sv}\right).
    \]
    \item[$\bullet$] Given a labeled graph $\bGa_{g,n}(\X)$, we define a weight
    \begin{eqnarray*}
\widetilde{w}_A^{X}(\vGa) &=& \prod_{v\in V(\Ga)} \Bigl(\sqrt{\Delta^{\bsi(v)}(t)}\Bigr)^{2g(v)-2+\val(v)} \langle \prod_{h\in H(v)} \tau_{k(h)}\rangle_{g(v)}
\\
&&\cdot \left(\prod_{e\in E(\Ga)} \cE^{\bsi(v_1(e)),\bsi(v_2(e))}_{k(h_1(e)),k(h_2(e))}\prod_{l\in L^1(\Ga)}(\cL^1)^{\bsi(l)}_{k(l)}\right)\left|_{t=\btau\atop\sw_i=w_i\sv}\right.\prod_{j=1}^n(\widetilde{\cL}^{O})^{\bsi(l_j)}_{k(l_j)}(l_j).
\end{eqnarray*}
\end{itemize}

Following Theorem \ref{thm:Zong} and Proposition \ref{prop:a-potential}, we have the following theorem
\begin{theorem}\label{thm:a graph sum}
    \[
        F^{\X,\Lrs}_{g,n}(\btau, \fQ; X_1,\dots, X_n) = \sum_{\vGa\in \bGa_{g,n}(\cX)}\frac{\widetilde{w}_A^{X}(\vGa)}{|\Aut(\vGa)|}.
    \]
\end{theorem}
\noindent
As discussed in \cite[Remark 3.6]{FLZ16}, the function
\[
    \GWPgn{\empty}{g}{n} := F^{\X,\Lrs}_{g,n}(\btau, 1; X_1,\dots, X_n)
\]
is well-defined. Theorem \ref{thm:a graph sum} implies
\begin{coro} \label{eqn:graph-open}
    Let $w_A^{X}(\vGa) = \widetilde{w}_A^{X}(\vGa)|_{\fQ=1}$,
    \[
        \GWPgn{\empty}{g}{n} =   \sum_{\vGa\in \bGa_{g,n}(\cX)}\frac{w_A^{X}(\vGa)}{|\Aut(\vGa)|}.
    \]
\end{coro}

\noindent
Let $\btau_2\in\HcrX{2}$ and apply Theorem \ref{thm:localization}, the disk potential is
\begin{equation*}
    \begin{aligned}
    & \GWPol{2} = \sum_{l\geq 0}\sum_{d\geq 0}\sum_{\mu>0}
    \frac{\GWol{\btau_2^l}}{l!}X^{\mu}
\\
    =\ &\frac{1}{p}\sum_{l\geq 0}\sum_{d\geq 0}\sum_{\mu>0}
    \frac{1}{l!}\DescdentGwol{\btau_2^l,\frac{\one_{0,\AlphaMu}^*}{\frac{\sv}{\mu}(\frac{\sv}{\mu}-\hat{\psi})}}{l+1}D'(\mu)X^{\mu}
\\
    =\ &\frac{1}{p}\sum_{l\geq 0}\sum_{d\geq 0}\sum_{\mu>0}
    \frac{1}{l!}\DescdentGwol{1, \btau_2^l,\frac{\one_{0,\AlphaMu}^*}{\frac{\sv}{\mu}-\hat{\psi}}}{l+2}D'(\mu)X^{\mu}
    \end{aligned}
\end{equation*}

Hence,
\[
    \GWPol{2} = \frac{1}{p}\sum_{h\in\cI}\sum_{\mu>0 \atop\AlphaMu=h}
    J_{0,h}(\btau_2,\frac{\sv}{\mu})D'(\mu)X^{\mu}.
\]
Applying the mirror theorem \ref{thm:mirror theorem},
\begin{equation}\label{eqn:disk-I-function}
    \begin{aligned}
        \GWPol{2&(q)} = \frac{1}{p}\sum_{h\in\cI}\sum_{\mu>0\atop\AlphaMu=h}D'(\mu)X^{\mu}
        \\
        &\cdot\sum_{\beta\in\mathbb{K}_\eff\atop \langle w(\beta)\rangle = h} \frac{q_1^{a_1}\dots q_p^{a_p}}{\left(\frac{\sv}{\mu}\right)^{1-\delta_{h,0}}}
        \frac{\prod_{m=-a_1+\lceil -w(\beta)\rceil}^{-1}(\frac{\mu r}{p}-(a_1+\wbeta+m))}{\prod_{m=0}^{a_1-1}(-\mu k+a_1 -m)}
        \\
        &\cdot\frac{
            \prod_{m=-a_1-a_2-\dots-a_p +\lceil w(\beta)\rceil}^{-1}(\frac{\mu s}{p}-(a_1+\dots+a_p+m)+\wbeta)}
        {a_1!a_2!\dots a_p!}
    \end{aligned}
\end{equation}

\subsection{Degree zero Open-closed Gromov-Witten invariants as relative Gromov-Witten invariants}
\label{sec:relative1}
The open-closed Gromov-Witten invariants can be interpreted in terms of relative Gromov-Witten invariants.
In this subsection, we express the degree 0 open-closed Gromov-Witten invariants
$\langle\gamma_1,\dots,\gamma_l\rangle_{g,l,0,\vec{\mu}}^{\X,\Lrs,\T}$ as certain relative Gromov-Witten invariants. In this case,
the curve class satisfies
$u_*[\Sigma] = (\sum_{j=1}^{n}\mu_i)b\in H_2(\cX,\Lrs)$.
\par
In order to compute the
$\T$-invariant holomorphic disk lying on the affine piece $\X_0\subset\X$, we consider the weighted projective
space
\[
    \fX = \bP(r,s,p) = (\bC^3\backslash\{0\})/\bC^*
\]
where the $\bC^*$-action on $\bC^3$ is
\[
    u\cdot(x_1,x_2,x_3) = (u^r x_1, u^s x_2, u^p x_3).
\]
We define $D_i=\{x_i = 0\}\subset \bP(r,s,p)$, $i=1,2,3$ as the $\bC^*$-equivariant divisors. In the Picard group
$\Pic(\fX)$ we have
\[
    \frac{[D_1]}{r} = \frac{[D_2]}{s} = \frac{[D_3]}{p}.
\]
The total space of $\cO_{\fX}(-D_1-D_2)$ is a partial compatification of the affine piece $\X_0= [\C^3/\Zp]$.
Let $d'\in\mathbb{N}^*$ and let $\vec{\mu}$ be a length $n$ partition of $d'$. Let $\vec{\gamma}=(\gamma_1,\cdots,\gamma_l)$ be a $l$-dimensional vector with $\gamma_j$'s nontrivial elements in $G_0\cong\bZ_p$.
Consider the moduli space $\Mbar_{g,l}(\fX,D_3,d',\vec{\mu})$ of stable maps relative the the divisor $D_3$ with ramification profile given by $\vec{\mu}$. Let
$$
\Mbar_{g,\vec{\gamma}}(\fX,D_3,d',\vec{\mu}):=\Mbar_{g,l}(\fX,D_3,d',\vec{\mu})\cap
\big(\bigcap_{i=1}^l\ev_i^*\one_{\gamma_i} \big),
$$
where $\ev_i$ is the evaluation map corresponding to the $i$-th non-relative marked point. The virtual dimension of $\Mbar_{g,\vec{\gamma}}(\fX,D_3,d',\vec{\mu})$ is equal to $d'k+n+g-1$, where $k\in\mathbb{N}$  is given by the identity $r+s=pk$.
Let
\[
    \pi: \cU_{g,\gamma,\vec{\mu}}\rightarrow\Mbar_{g,\vec{\gamma}}(\fX,D_3,d',\vec{\mu})
\]
be the universal curve and let
\[
    \cT\rightarrow\Mbar_{g,\vec{\gamma}}(\fX,D_3,d',\vec{\mu})
\]
be the universal target. There is an evaluation map
\[
    F:\cU_{g,\vec{\gamma},\vec{\mu}}\rightarrow \cT
\]
and a contraction map
\[
    \tilde{\pi}: \cT\rightarrow \fX.
\]
Let $\tilde{F}:= \tilde{\pi}\circ F: \cU_{g,\vec{\mu}}\rightarrow\fX$. Then we define the obstruction bundle $V_{g,\vec{\gamma},\vec{\mu}}$ as
\[
    V_{g,\vec{\gamma},\vec{\mu}} := R^1\pi_*\tilde{F}^*\cO(-D_1-D_2).
\]
The Riemann-Roch theorem tells that the rank of the obstruction bundle $V_{g,\vec{\gamma},\vec{\mu}}$ is $d'k+g-1$.
Let
\[
    N_{g,\vec{\gamma},\vec{\mu}} = \int_{[\Mbar_{g,\vec{\gamma}}(\fX,D_3,d',\vec{\mu})]^\vir} e(V_{g,\vec{\gamma},\vec{\mu}})\prod_{j=1}^{n}\ev^*_j[\text{pt}]
\]
where $[\pt]$ is the point class on $D_3$ and $\ev_j$ is the evaluation map corresponding to the $j$-th relative marked point. We have $N_{g,\vec{\gamma},\vec{\mu}}$ is a topological invariant.
\begin{remark}\rm\label{disk in compact X}
Let $D=[\{z\in\bC | |z|\leq 1\}/\Zp]$ be the orbidisk and let $t=z^p$. Consider the map
\[
    g: D\rightarrow\fX
\]
\[
    t\mapsto [t^{r/p},t^{s/p},1].
\]
The map $g$ can extend to 
\[
    \tilde{g}: \bP(p,1)\rightarrow \fX
\]
\[
    [x,y] \mapsto [x^{r/p}, x^{s/p}, y^p]
\]
where $\tilde{g}([1,0]) = [1,1,0] \in D_3$. Therefore, intuitively the relative Gromov-Witten invariant $N_{g,\vec{\gamma},\vec{\mu}}$ is equal to
the open-closed Gromov-Witten invariant $\langle\gamma_1,\dots,\gamma_l\rangle_{g,l,0,\vec{\mu}}^{\X,\Lrs,\T}$ in principle.
\end{remark}
The relative Gromov-Witten invariant $N_{g,\vec{\gamma},\vec{\mu}}$ can be computed by the virtual localization formula. Consider the embedded 2-torus $\bT\subset\fX$. The 2-torus $\bT$ extends to a $\bT$-action on $\fX$, and
it can be lifted to the line bundle $\cO(-D_1-D_2)$. Let $p_1=[1,0,0]$, $p_2=[0,1,0]$, $p_3=[0,0,1]$, then the weights of the $\bT$-action at $p_1, p_2, p_3$ are given by
\begin{equation*}
    \begin{array}{ccc}
            & T_\fX                                     & \cO(-D_1-D_2)\\
        p_1 & -\frac{\su_1}{r}, -\frac{k}{r}\su_1-\su_2 & \frac{k}{r}\su_1+\su_2\\
        p_2 & \frac{\su_1}{s}+\frac{p\su_2}{s}, \frac{k\su_1}{s}+\frac{r\su_2}{s} &  -\frac{k\su_1}{s}-\frac{r\su_2}{s}  \\
        p_3 & \frac{\su_1}{p}, -\frac{\su_1}{p}-\su_2 & \su_2
    \end{array}
\end{equation*}
Let $\su_1=r\sv$, $\su_2=-k\sv$, where $\sv$ is the equivariant parameter of the subtorus $\T\cong \bC^*\subset\bT$. Then the weight of $\T$-action at $p_1, p_2, p_3$
are given by
\begin{equation*}
    \begin{array}{ccc}
            & T_\fX   & \cO(-D_1-D_2)\\
        p_1 & -\sv, 0 & 0\\
        p_2 & -\sv, 0 & 0\\
        p_3 & \frac{r\sv}{p}, \frac{s\sv}{p} & -k\sv
    \end{array}
\end{equation*}
The subtorus $\T$ preserves the projective line in Remark \ref*{disk in compact X}. Moreover, since the $\T$-action along $D_3$ is trivial
and it is also trivial on $\cO(-D_1-D_2)|_{D_3}$, we just need to consider the non-bubble component in $\Mbar_{g,\vec{\gamma}}(\fX,D_3,d',\vec{\mu})$. 
The tangent obstruction sequence in this case is
\begin{alignat*}{4}
    0 & \to \text{Aut}(C,x,y) & \to H^0(C, f^*T\mathfrak{X}(-\log D_3)) & \to \mathcal{T}^1 &\\
     & \to \text{Def}(C,x,y) & \to  H^1(C, f^*T\mathfrak{X}(-\log D_3)) & \to \mathcal{T}^2 &\to 0 ,
\end{alignat*}
where $C$ is the domain curve, $x=(x_1,\cdots,x_l)$ are the non-relative marked points, and $y=(y_1,\cdots,y_n)$ are the relative marked points. The relative Gromov-Witten invariant $N_{g,\vec{\gamma},\vec{\mu}}$ is given by
\[
    N_{g,\vec{\gamma},\vec{\mu}} = \frac{1}{\prod_{j=1}^{n}\mu_j} \frac{e_\bT(V_{g,\vec{\gamma},\vec{\mu}})e_\bT(\mathcal{T}^{2,m})}{e_\bT(\mathcal{T}^{1,m})}
    \Big|_{\su_1=r\sv,\su_2=-k\sv}
\]

\begin{figure*}[h]

\begin{center}
\tikzset{every picture/.style={line width=0.75pt}} 

\begin{tikzpicture}[x=0.75pt,y=0.75pt,yscale=-1,xscale=1]

\draw    (237.36,193.36) -- (409,193.36) ;
\draw    (432.36,63.36) -- (409,193.36) ;
\draw    (237.36,193.36) -- (481.09,30.96) ;
\draw [shift={(482.75,29.85)}, rotate = 146.33] [color={rgb, 255:red, 0; green, 0; blue, 0 }  ][line width=0.75]    (10.93,-3.29) .. controls (6.95,-1.4) and (3.31,-0.3) .. (0,0) .. controls (3.31,0.3) and (6.95,1.4) .. (10.93,3.29)   ;
\draw    (432.36,63.36) -- (194.76,221.77) ;
\draw [shift={(193.1,222.87)}, rotate = 326.31] [color={rgb, 255:red, 0; green, 0; blue, 0 }  ][line width=0.75]    (10.93,-3.29) .. controls (6.95,-1.4) and (3.31,-0.3) .. (0,0) .. controls (3.31,0.3) and (6.95,1.4) .. (10.93,3.29)   ;
\draw    (409,193.36) -- (424.9,226.43) ;
\draw [shift={(425.76,228.23)}, rotate = 244.33] [color={rgb, 255:red, 0; green, 0; blue, 0 }  ][line width=0.75]    (10.93,-3.29) .. controls (6.95,-1.4) and (3.31,-0.3) .. (0,0) .. controls (3.31,0.3) and (6.95,1.4) .. (10.93,3.29)   ;
\draw    (237.36,193.36) -- (276.43,193.56) ;
\draw [shift={(278.43,193.57)}, rotate = 180.29] [color={rgb, 255:red, 0; green, 0; blue, 0 }  ][line width=0.75]    (10.93,-3.29) .. controls (6.95,-1.4) and (3.31,-0.3) .. (0,0) .. controls (3.31,0.3) and (6.95,1.4) .. (10.93,3.29)   ;
\draw    (409,193.36) -- (414.96,158.98) ;
\draw [shift={(415.3,157.01)}, rotate = 99.83] [color={rgb, 255:red, 0; green, 0; blue, 0 }  ][line width=0.75]    (10.93,-3.29) .. controls (6.95,-1.4) and (3.31,-0.3) .. (0,0) .. controls (3.31,0.3) and (6.95,1.4) .. (10.93,3.29)   ;
\draw    (409,193.36) -- (369.3,193.43) ;
\draw [shift={(367.3,193.44)}, rotate = 359.89] [color={rgb, 255:red, 0; green, 0; blue, 0 }  ][line width=0.75]    (10.93,-3.29) .. controls (6.95,-1.4) and (3.31,-0.3) .. (0,0) .. controls (3.31,0.3) and (6.95,1.4) .. (10.93,3.29)   ;
\draw    (432.36,63.36) -- (425.66,100.11) ;
\draw [shift={(425.3,102.08)}, rotate = 280.33] [color={rgb, 255:red, 0; green, 0; blue, 0 }  ][line width=0.75]    (10.93,-3.29) .. controls (6.95,-1.4) and (3.31,-0.3) .. (0,0) .. controls (3.31,0.3) and (6.95,1.4) .. (10.93,3.29)   ;
\draw    (432.36,63.36) -- (385.96,94.47) ;
\draw [shift={(384.3,95.58)}, rotate = 326.16] [color={rgb, 255:red, 0; green, 0; blue, 0 }  ][line width=0.75]    (10.93,-3.29) .. controls (6.95,-1.4) and (3.31,-0.3) .. (0,0) .. controls (3.31,0.3) and (6.95,1.4) .. (10.93,3.29)   ;
\draw    (237.36,193.36) -- (286.63,160.55) ;
\draw [shift={(288.3,159.44)}, rotate = 146.34] [color={rgb, 255:red, 0; green, 0; blue, 0 }  ][line width=0.75]    (10.93,-3.29) .. controls (6.95,-1.4) and (3.31,-0.3) .. (0,0) .. controls (3.31,0.3) and (6.95,1.4) .. (10.93,3.29)   ;
\draw [color={rgb, 255:red, 255; green, 0; blue, 0 }  ,draw opacity=1 ]   (331.21,131.21) .. controls (380.3,123.51) and (406.3,167.01) .. (409,193.36) ;
\draw [color={rgb, 255:red, 255; green, 0; blue, 0 }  ,draw opacity=1 ]   (331.21,131.21) .. controls (338.51,156.01) and (356.8,188.3) .. (409,192.15) ;
\draw [color={rgb, 255:red, 74; green, 144; blue, 226 }  ,draw opacity=1 ]   (354.71,170.11) .. controls (352.21,160.54) and (374.71,139.54) .. (386.64,148.5) ;
\draw [color={rgb, 255:red, 74; green, 144; blue, 226 }  ,draw opacity=1 ]   (354.71,170.11) .. controls (377.71,180.61) and (391.21,159.11) .. (386.64,148.5) ;
\draw  [fill={rgb, 255:red, 0; green, 0; blue, 0 }  ,fill opacity=1 ] (234,193.36) .. controls (234,191.5) and (235.5,190) .. (237.36,190) .. controls (239.21,190) and (240.71,191.5) .. (240.71,193.36) .. controls (240.71,195.21) and (239.21,196.71) .. (237.36,196.71) .. controls (235.5,196.71) and (234,195.21) .. (234,193.36) -- cycle ;
\draw  [fill={rgb, 255:red, 0; green, 0; blue, 0 }  ,fill opacity=1 ] (405.64,193.36) .. controls (405.64,191.5) and (407.15,190) .. (409,190) .. controls (410.85,190) and (412.36,191.5) .. (412.36,193.36) .. controls (412.36,195.21) and (410.85,196.71) .. (409,196.71) .. controls (407.15,196.71) and (405.64,195.21) .. (405.64,193.36) -- cycle ;
\draw  [fill={rgb, 255:red, 0; green, 0; blue, 0 }  ,fill opacity=1 ] (429,63.36) .. controls (429,61.5) and (430.5,60) .. (432.36,60) .. controls (434.21,60) and (435.71,61.5) .. (435.71,63.36) .. controls (435.71,65.21) and (434.21,66.71) .. (432.36,66.71) .. controls (430.5,66.71) and (429,65.21) .. (429,63.36) -- cycle ;

\draw (485.73,28.5) node [anchor=north west][inner sep=0.75pt]   [align=left] {$-\frac{k\su_1}{s}-\frac{r\su_2}{s}$};
\draw (337,62) node [anchor=north west][inner sep=0.75pt]   [align=left] {$\frac{k\su_1}{s}+\frac{r\su_2}{s}$};
\draw (437,80) node [anchor=north west][inner sep=0.75pt]   [align=left] {$\frac{\su_1}{s}+\frac{p\su_2}{s}$};
\draw (219.64,143.43) node [anchor=north west][inner sep=0.75pt]   [align=left] {$-\frac{k\su_1}{r}-\su_2$};
\draw (121,202) node [anchor=north west][inner sep=0.75pt]   [align=left] {$\frac{k\su_1}{r}+\su_2$};
\draw (252.64,203.93) node [anchor=north west][inner sep=0.75pt]   [align=left] {$-\frac{\su_1}{r}$};
\draw (361.64,203.93) node [anchor=north west][inner sep=0.75pt]   [align=left] {$\frac{\su_1}{p}$};
\draw (430.64,218.93) node [anchor=north west][inner sep=0.75pt]   [align=left] {$\su_2$};
\draw (422.14,162.93) node [anchor=north west][inner sep=0.75pt]   [align=left] {$-\frac{\su_1}{p}-\su_2$};
\draw (173.64,171.93) node [anchor=north west][inner sep=0.75pt]   [align=left] {$p_1=[\pt/\bZ_r]$};
\draw (421.64,187.93) node [anchor=north west][inner sep=0.75pt]   [align=left] {$p_3=[\pt/\Zp]$};
\draw (376.64,35.93) node [anchor=north west][inner sep=0.75pt]   [align=left] {$p_2=[\pt/\bZ_s]$};
\draw (270,113) node [anchor=north west][inner sep=0.75pt]   [align=left] {$D_3$};
\draw (297.67,220.67) node [anchor=north west][inner sep=0.75pt]   [align=left] {$D_2$};
\draw (452,131.67) node [anchor=north west][inner sep=0.75pt]   [align=left] {$D_1$};

\end{tikzpicture}
\end{center}
\caption{The toric diagram of $\mathcal{O}(-D_1-D_2)\rightarrow \bP(r,s,p)$, with orbifold points $p_1, p_2,p_3$
and weights of torus action labeled. The red projective line is given by the map $t\mapsto [t^{r/p}, t^{s/p}, 1]\in
\bP(r,s,p)$. The blue circle corresponds to the intersection of $\Lrs$ with the fiber at $\fp_0= p_3$
before the compatification.}
\end{figure*}
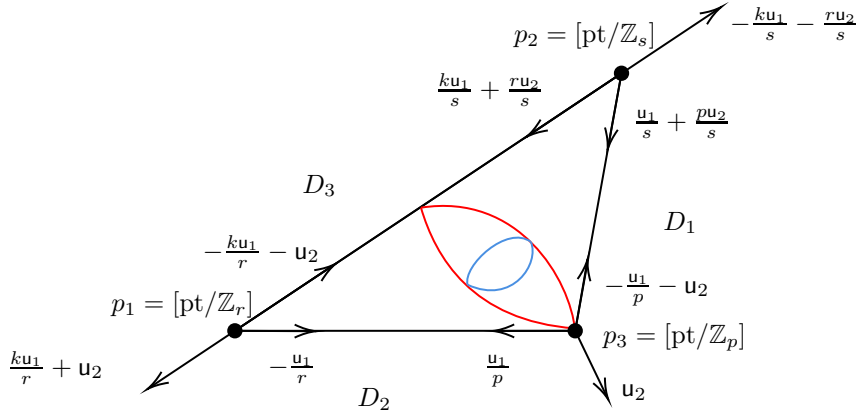

By localization computation, one obtains
\[
    N_{g,\vec{\gamma},\vec{\mu}} = \prod_{j=1}^{n}D'(\mu_j)
        \cdot \int_{[\overline{\mathcal{M}}_{g,n+l}(\X,0)]^{\vir}}\frac{\prod_{i=1}^l \ev^*_i\gamma_i\prod_{j=1}^{n}\ev^*_{l+j}\rho_{0, \langle-\frac{\mu_j r}{p}\rangle}}
        {\prod_{j=1}^{n}\frac{\sv}{\mu_j}(\frac{\sv}{\mu_j} - \hat{\psi}_{l+j})},
\]
where $D'(\mu_j)$ is defined in \eqref{eqn:diskfactor}. By comparing with Theorem \ref{thm:localization}, we obtain the following theorem.
\begin{theorem}\label{thm:open-relative1}
 We have   
    \[
        \langle \gamma_1,\dots,\gamma_l\rangle_{g,l,0,\vec{\mu}}^{\X,\Lrs,\T} = N_{g,\vec{\gamma},\vec{\mu}}.
    \]
\end{theorem}

\begin{remark}
In the localization computation above, we can also use the two dimensional torus $\bT$ to do localization.
In this situation, we can put the point class insertion $[\pt]$ in $N_{g,\vec{\gamma},\vec{\mu}}$ as $r[p_1]$ or $s[p_2]$ and we can still obtain Theorem \ref{thm:open-relative1}.
\end{remark}

\subsection{General Open-closed Gromov-Witten invariants as formal relative Gromov-Witten invariants}
\label{sec:relative2}
In the general case when $d>0$, we cannot express our open-closed Gromov-Witten invariants as relative Gromov-Witten invariants of an ordinary toric CY 3-orbifold. In order to solve this problem, we consider \emph{formal toric CY 3-folds} as in \cite[Section 3, Section 4]{LLLZ09}. Then we can interpret the open-closed Gromov-Witten invariants $\langle \gamma_1,\dots,\gamma_\ell\rangle^{\cX,L_{r,s},\T}_{g,\ell,d,\vec{\mu}}$ as relative Gromov-Witten invariants of a certain formal toric CY 3-fold. For completeness, we briefly review the construction of formal toric CY 3-folds in \cite{LLLZ09}.

\subsubsection{Formal toric CY graphs}
Let $\Gamma$ be an oriented planar graphs. We first introduce some notations.
\begin{enumerate}
\item $E^o(\Gamma)$ is the set of oriented edges of $\Gamma$.
\item $V(\Gamma)$ is the set of vertices of $\Gamma$.
\item Define the \emph{orientation reversing map} $-:E^o(\Gamma)\to E^o(\Gamma), e\mapsto -e$ which reverses the orientation of an edge.
\item $E(\Gamma):=E^o(\Gamma)/\{\pm 1\}$ is the set of unoriented edges of $\Gamma$.
\item $\mathfrak{v}_0:E^o(\Gamma)\to V(\Gamma)$ is the initial vertex map and $\mathfrak{v}_0:E^o(\Gamma)\to V(\Gamma)$ is the terminal vertex map
\end{enumerate}
We require that the orientation reversing map is fixed point free and that both $\mathfrak{v}_0$ and $\mathfrak{v}_1$ are surjective and $\mathfrak{v}_0(e)=\mathfrak{v}_1(-e)$ for $\forall e\in E^o(\Gamma)$. We also require that the valence of any vertex of $\Gamma$ is 3 or 1 and let $V_1(\Gamma)$ and $V_3(\Gamma)$ be the set of univalent vertices and trivalent vertices respectively.

We define
$$
E^\mathfrak{f}=\{e\in E^o(\Gamma)|\mathfrak{v}_1(e)\in V_1(\Gamma)\}.
$$
Let $u_1,u_2$ be the standard basis of $\bZ^{\oplus 2}$ such that the ordered basis $(u_1,u_2)$ determines an orientation on $\bR^2$.

\begin{definition}
A formal toric CY graph is a planar graph $\Gamma$ described above together with a position map
$$
\mathfrak{p}:E^o(\Gamma)\to \bZ^{\oplus 2}\setminus \{0\}
$$
and a framing map
$$
\mathfrak{f}:E^\mathfrak{f}(\Gamma)\to \bZ^{\oplus 2}\setminus \{0\}
$$
such that
\begin{enumerate}
\item $\mathfrak{p}(-e)=-\mathfrak{p}(e)$ for $\forall e\in E^o(\Gamma)$.\\
\item  For a trivalent vertex $v\in V_3(\Gamma)$ with $\mathfrak{v}_0^{-1}(v)=\{e_1,e_2,e_3\}$, we require that $\mathfrak{p}(e_1)+\mathfrak{p}(e_2)+\mathfrak{p}(e_3)=0$.\\
\item For a trivalent vertex $v\in V_3(\Gamma)$ with $\mathfrak{v}_0^{-1}(v)=\{e_1,e_2,e_3\}$, any two vectors in $\{\mathfrak{p}(e_1),\mathfrak{p}(e_2),\mathfrak{p}(e_3)\}$ form an integral basis of $\bZ^{\oplus 2}$.\\
\item For $e\in E^\mathfrak{f}(\Gamma), \mathfrak{p}(e)\wedge \mathfrak{f}(e)=u_1\wedge u_2$.

\end{enumerate}
\end{definition}

\subsubsection{Formal toric CY 3-folds}
The motivation of introducing formal toric CY 3-folds is based on the following observation. Let $Y$ be a toric CY 3-fold and let $\bT\cong (\bC^*)^2$ be the CY torus acting on $Y$. Let $Y^1$ be the closure of the 1-dimensional $\bT-$orbits in $Y$. By virtual localization formula, the $\bT-$equivariant Gromov-Witten theory of $Y$ is determined by $Y^1$ together with the normal bundle of each irreducible component of $Y^1$ in $Y$. In other words, we only need the information of the formal neighborhood of $Y^1$ in $Y$ to study the $\bT-$equivariant Gromov-Witten theory of $Y$. So one may define $\hat{Y}$ to be the formal completion of $Y$ along $Y^1$ and try to define and study the Gromov-Witten theory of $\hat{Y}$. The advantage of considering $\hat{Y}$ is that if we are given a formal toric CY graph $\Gamma$, we can always construct an associated $\hat{Y}$ even if there does not exist a usual toric CY 3-fold corresponding to this formal toric CY graph. This construction is called the \emph{formal toric CY 3-fold} associated to $\Gamma$. In the following, we briefly describe the construction of formal toric CY 3-folds.

We fix a formal toric CY graph $\Gamma$. For any edge $e\in E(\Gamma)$, we associate a relative toric CY 3-fold $(Y_e,D_e)$, where $Y_e$ is the total space of the direct sum of two line bundles  over $\bP^1$ and $D_e$ is a divisor of $Y_e$ which can be empty or the fiber(s) of $Y_e\to \bP^1$ over 1 or 2 points on $\bP^1$, depending on the number of univalent vertices of $e$ (see \cite[Section 3]{LLLZ09}). Here $(Y_e,D_e)$ being a relative CY 3-fold means
$$\Lambda^3\Omega_{Y_e}(\log D_e)\cong\cO_{Y_e}.$$
Moreover, the formal toric CY graph also determines a $\bT\cong (\bC^*)^2$ action on $Y_e$, under which $D_e$ is an invariant divisor. Let $\Sigma(e)$ be the formal completion of $Y_e$ along its zero section $\bP^1$. The divisor $D_e$ descends to a divisor $\hat{D}_e$ of $\Sigma(e)$ and there is an induced $\bT$ action on $\Sigma(e)$. The formal relative toric CY 3-fold $(\hat{Y},\hat{D})$ is obtained by gluing all the $(\Sigma(e),\hat{D}_e)$'s along the trivalent vertices of $\Gamma$ (Each trivalent vertex $v$ corresponds to a formal scheme $\Spec(\bC[[x_1,x_2,x_3]])$, which can be naturally embedded into $\Sigma(e)$ for $e$ connected to $v$). As a set, $\hat{Y}$ is a union of $\bP^1$'s. Each connected component of the divisor $\hat{D}$ corresponds to a univalent vertex of $\Gamma$. The $\bT$ actions on each $\Sigma(e)$ are also glued together to form a $\bT$ action on $\hat{Y}$ such that the divisor $\hat{D}$ is $\bT$-invariant.

The result in \cite[Section 4]{LLLZ09} shows that one can define $\bT-$equivariant relative Gromov-Witten invariants for the formal relative toric CY 3-fold $(\hat{Y},\hat{D})$. When the formal toric CY graph comes from the toric graph of a usual toric CY 3-fold, these formal Gromov-Witten invariants coincide with the usual Gromov-Witten invariants of the toric CY 3-fold.

\subsubsection{Our case}
In this subsection, we apply the above construction to our case. First of all, we should notice that the above construction can be generalized to the case when there are orbifold structures on $\hat{Y}$. This means that if we have an edge $e\in E^\mathfrak{f}(\Gamma)$, $Y_e$ can be the direct sum of two orbifold line bundles over the weighted projective line $\bP(p,r)$ and the divisor $D_e$ is the fiber over the orbifold point $[\pt/\bZ_r]$ on $\bP(p,r)$.

\begin{figure}[!ht]
    \centering

    \tikzset{every picture/.style={line width=0.75pt}} 

    \begin{tikzpicture}[x=0.90pt,y=0.90pt,yscale=-1,xscale=1]

    \draw    (250,160) -- (236.41,173.59) ;
    \draw [shift={(235,175)}, rotate = 315] [color={rgb, 255:red, 0; green, 0; blue, 0 }  ][line width=0.75]    (10.93,-3.29) .. controls (6.95,-1.4) and (3.31,-0.3) .. (0,0) .. controls (3.31,0.3) and (6.95,1.4) .. (10.93,3.29)   ;
    \draw    (250,160) -- (250,122) ;
    \draw [shift={(250,120)}, rotate = 90] [color={rgb, 255:red, 0; green, 0; blue, 0 }  ][line width=0.75]    (10.93,-3.29) .. controls (6.95,-1.4) and (3.31,-0.3) .. (0,0) .. controls (3.31,0.3) and (6.95,1.4) .. (10.93,3.29)   ;
    \draw    (250,160) -- (350,160) ;
    \draw    (350,160) -- (350,200) ;
    \draw    (350,160) -- (390,120) ;
    \draw    (210,200) -- (223.59,186.41) ;
    \draw [shift={(225,185)}, rotate = 135] [color={rgb, 255:red, 0; green, 0; blue, 0 }  ][line width=0.75]    (10.93,-3.29) .. controls (6.95,-1.4) and (3.31,-0.3) .. (0,0) .. controls (3.31,0.3) and (6.95,1.4) .. (10.93,3.29)   ;
    \draw    (210,200) -- (250,160) ;
    \draw  [line width=3] [line join = round][line cap = round] (210.5,199.5) .. controls (210.5,199.5) and (210.5,199.5) .. (210.5,199.5) ;
    \draw  [line width=3] [line join = round][line cap = round] (250.5,159.5) .. controls (250.5,159.5) and (250.5,159.5) .. (250.5,159.5) ;
    \draw  [line width=3] [line join = round][line cap = round] (350.5,160) .. controls (350.5,160) and (350.5,160) .. (350.5,160) ;
    \draw  [dash pattern={on 4.5pt off 4.5pt}]  (210,200) -- (216.69,157.98) ;
    \draw [shift={(217,156)}, rotate = 99.04] [color={rgb, 255:red, 0; green, 0; blue, 0 }  ][line width=0.75]    (10.93,-3.29) .. controls (6.95,-1.4) and (3.31,-0.3) .. (0,0) .. controls (3.31,0.3) and (6.95,1.4) .. (10.93,3.29)   ;
    \draw  [dash pattern={on 4.5pt off 4.5pt}]  (210,200) -- (205.31,229.42) ;
    \draw [shift={(205,231.4)}, rotate = 279.05] [color={rgb, 255:red, 0; green, 0; blue, 0 }  ][line width=0.75]    (10.93,-3.29) .. controls (6.95,-1.4) and (3.31,-0.3) .. (0,0) .. controls (3.31,0.3) and (6.95,1.4) .. (10.93,3.29)   ;
    \draw    (250,160) -- (278,160) ;
    \draw [shift={(280,160)}, rotate = 180] [color={rgb, 255:red, 0; green, 0; blue, 0 }  ][line width=0.75]    (10.93,-3.29) .. controls (6.95,-1.4) and (3.31,-0.3) .. (0,0) .. controls (3.31,0.3) and (6.95,1.4) .. (10.93,3.29)   ;

    \draw (246.21,169.71) node [anchor=north west][inner sep=0.75pt]   [align=left] {$\frac{\su_1}{p}$};
    \draw (266.71,147.21) node [anchor=north west][inner sep=0.75pt]   [align=left] {$\su_2$};
    \draw (255.21,118.21) node [anchor=north west][inner sep=0.75pt]   [align=left] {$-\frac{\su_1}{p}-\su_2$};
    \draw (222.21,190.71) node [anchor=north west][inner sep=0.75pt]   [align=left] {$-\frac{\su_1}{r}$};
    \draw (149.71,151.71) node [anchor=north west][inner sep=0.75pt]   [align=left] {$-\frac{k}{r}\su_1-\su_2$};
    \draw (216.71,214.21) node [anchor=north west][inner sep=0.75pt]   [align=left] {$\frac{k}{r}\su_1+\su_2$};
    \draw (161.71,190.21) node [anchor=north west][inner sep=0.75pt]   [align=left] {$[\pt/\bZ_r]$};
    \draw (272,162) node [anchor=north west][inner sep=0.75pt]   [align=left] {$\bP^1\times \cB\Zp$};
    \draw (292,151) node [anchor=north west][inner sep=0.75pt]   [align=left] {$e$};
    \draw (236.71,138.21) node [anchor=north west][inner sep=0.75pt]   [align=left] {$e_2$};
    \draw (236.71,152.21) node [anchor=north west][inner sep=0.75pt]   [color={rgb, 255:red, 208; green, 2; blue, 27}  ,opacity=1 ] [align=left] {$v_0$};
    \draw (219.71,170.21) node [anchor=north west][inner sep=0.75pt]   [align=left] {$e_1$};
    \draw (353.71,155.21) node [anchor=north west][inner sep=0.75pt]   [color={rgb, 255:red, 208; green, 2; blue, 27}  ,opacity=1 ] [align=left] {$v_1$};
    \draw (370.71,138.21) node [anchor=north west][inner sep=0.75pt]   [align=left] {$e_3$};
    \draw (353.71,190.21) node [anchor=north west][inner sep=0.75pt]   [align=left] {$e_4$};

    \end{tikzpicture}

\caption{Toric diagram of the formal toric Calabi-Yau threefold}\label{fig:relative-ii}
\end{figure}

Now we consider the following formal toric CY graph $\Gamma$ (see Figure \ref{fig:relative-ii}). There are two trivalent vertices $v_0$ and $v_1$ with an
edge $e$ connecting them. There are four edges $e_1,e_2,e_3,e_4\in E^\mathfrak{f}(\Gamma)$ with $\mathfrak{v}_0(e_1)=\mathfrak{v}_0(e_2)=v_0,
\mathfrak{v}_0(e_3)=\mathfrak{v}_0(e_4)=v_1$. The graph $\Gamma$ defines a formal relative toric CY 3-orbifold $(\hat{Y},\hat{D})$.
The divisor $\hat{D}$ is of the form $\hat{D}_1\cup\hat{D}_2\cup\hat{D}_3\cup\hat{D}_4$, where each $\hat{D}_i$ is a connected component of
$\hat D$ corresponding to the univalent vertex $\mathfrak{v}_1(e_i)$. The formal 3-fold $\hat{Y}$ is obtained by gluing $\Sigma(e),\Sigma(e_1),\Sigma(e_2)$
along the trivalent vertex $v_0$ and by gluing $\Sigma(e),\Sigma(e_3),\Sigma(e_4)$ along the trivalent vertex $v_1$. Here $\Sigma(e)$ is the formal
completion of $Y_e$ along the zero section and $Y_e$ is the total space of $[\cO_{\bP^1}(-1)\oplus\cO_{\bP^1}(-1)/\bZ_p]\to\bP^1\times\cB\bZ_p$.
In a similar way, let $Y_{e_1}$ be the total space of $N_{D_2/\fX}\oplus \cO_{\fX}(-D_1-D_2)\mid_{D_2}\to D_2$ and then $\Sigma(e_1)$ is the formal completion
of $Y_{e_1}$ along the zero section $\bP(p,r)$. The $\bT$ action on $\cO_{\fX}(-D_1-D_2)\to\fX$ (see Section \ref{sec:relative1}) induces a $\bT$ action on
$Y_{e_1}$ which coincides with the $\bT$ action given by the formal toric CY graph.

\subsubsection{The formal relative Gromov-Witten invariants}
Let $\Mbar_{g,l}(\hat{Y},\hat{D}_1,d,\vec{\mu})$ be the moduli space of relative stable maps to $(\hat{Y},\hat{D}_1)$, where $\vec{\mu}=(\mu_1,\cdots,\mu_n)$ is a partition of some positive integer $d'$ and $d$ is the degree of the stable map restricted to those components which are mapped to the $\bT-$invariant 1-orbit corresponding to the edge $e$. The partition $\vec{\mu}$ describes the ramification profile over $\hat{D}_1$. Define the formal relative Gromov-Witten invariant as
\begin{equation}
\langle \gamma_1,\dots,\gamma_\ell\mid\vec{\mu},\hat{D}_1\rangle^{\hat{Y},\bT}_{g,\ell,d,\vec{\mu}}
:=\prod_{i=1}^{n}(-\frac{k\su_1}{r}-\su_2)
\int_{[\Mbar_{g,\ell}(\hat{Y},\hat{D}_1,d,\vec{\mu})^\bT]^{\vir}}
\frac{\prod_{j=1}^\ell \ev_j^*\gamma_j}{e_{\bT}(N^\vir)}.
\end{equation}
where $\gamma_1,\cdots,\gamma_\ell$ are $\bT-$equivariant cohomology classes of $\cX$. Notice that since the divisor $\hat{D}$ is decomposed into 4 connected components $\hat{D}_1\cup\hat{D}_2\cup\hat{D}_3\cup\hat{D}_4$, the above relative Gromov-Witten invariant is a special case of the relative Gromov-Witten invariant of $(\hat{Y},\hat{D})$ by letting the ramification profiles over $\hat{D}_2,\hat{D}_3,\hat{D}_4$ be empty. Therefore, the result in \cite[Section 4]{LLLZ09} shows that the above formal relative Gromov-Witten invariant is well-defined. We should also notice that since $\gamma_1,\cdots,\gamma_\ell$ are $\bT-$equivariant cohomology classes of $\cX$, they can be viewed as $\bT-$equivariant cohomology classes of $\hat{Y}$.

By Theorem \ref{thm:localization} and standard localization computation on $\langle \gamma_1,\dots,\gamma_\ell\mid\vec{\mu},\hat{D}_1\rangle^{\hat{Y},\bT}_{g,\ell,d,\vec{\mu}}$, it is easy to obtain the following theorem

\begin{theorem}
The open-closed Gromov-Witten invariant $\langle \gamma_1,\dots,\gamma_\ell\rangle^{\cX,L_{r,s},\T}_{g,\ell,d,\vec{\mu}}$ can be expressed as the formal relative Gromov-Witten invariant in the following way:
$$
\langle \gamma_1,\dots,\gamma_\ell\rangle^{\cX,L_{r,s},\T}_{g,\ell,d,\vec{\mu}}=
\langle \gamma_1,\dots,\gamma_\ell\mid\vec{\mu},\hat{D}_1\rangle^{\hat{Y},\bT}_{g,\ell,d,\vec{\mu}}
\mid_{\su_1=r\sv,\su_2=-k\sv}
$$
\end{theorem}

\section{Mirror curve and B-model topological string}\label{sec:B}
\subsection{Geometry of the mirror curve}\label{v(eta)}
Let $u, v\in \bC$, and define
\[
      e^{-u}=U, \quad e^{-v}=V.
\]
The mirror curve $C_q$ of the orbifold resolved conifold $\X=\ReducedConi$ is the following affine curve in $(\bC^*)^2$
\begin{equation}\label{StandardMirror}
   H(U,V,q):= U^p + V + 1 + q_1U^pV^{-1} + \sum_{m=1}^{p-1}q_{m+1}U^m = 0.
\end{equation}
Let $P$ be the two dimensional polytope whose vertices are given by $b'_1 = (p,0), b'_2=(0,1), b'_3=(0,0), b'_4=(p,-1)$. Recall that the fan of $\cX$ is given by a cone over a triangulation of $P$. The polytope $P$ determines a toric surface $\bS_P$ with a polarization $L_P$ and $H(U,V,q)$ extends to a section $s_q\in H^0(\bS_P,L_P)$. The compactified mirror curve $\overline{C}_q\in \bS_P$ is defined to be the zero locus of $s_q$. For generic $q$, the compactified mirror curve $\overline{C}_q$ is a compact Riemann
surface of genus $p-1$ and it intersects the anti-canonical divisor $\partial \bS_P=\bS_P\setminus(\bC^*)^2$ transversally at 4 points. The inclusion $J: C_q\rightarrow\overline{C}_q$ induces a surjective homomorphism
$$J_*: H_1(C_q;\bZ)\cong \bZ^{2p+1} \rightarrow H_1(\overline{C}_q;\bZ)\cong\bZ^{2p-2}.$$

Because $r, s, p$ are pairwisely coprime, we know that $r, k$ ($pk = r+s$) are coprime and we can find $\gamma, \delta\in \bZ$ such that
\begin{equation*}
    \left(
        \begin{array}{cc}
            r & -k
            \\
            \gamma & \delta
        \end{array}
    \right)\in SL(2;\bZ).
\end{equation*}
Consider the $SL(2;\bZ)$ transform of coordinates:
\[
    x = ru - kv, \quad y = \gamma u + \delta v,
\]
\begin{equation}
    \label{eqn:coord2}X= e^{-x} = U^rV^{-k}, \quad Y= e^{-y} = U^\gamma V^\delta.
\end{equation}

Under the coordinates $X, V$, the mirror curve $C_q$ of ($\X, \Lrs$) is
\[
    X^{\frac{p}{r}}V^{\frac{kp}{r}} + V + 1 + q_1X^{\frac{p}{r}}V^{\frac{kp}{r}-1} + \sum_{m=1}^{p-1}q_{m+1}X^{\frac{m}{r}}V^{\frac{km}{r}} = 0
\]
Let $\eta = X^{\frac{1}{r}}$. Then $\eta$ is a local coordinate of the compactified mirror curve $\overline{C}_q$ around the point $\mathfrak{s}_0=(X, V) = (0,-1)$.
There exists $\delta>0$ and $\ep>0$ such that for $|q|<\ep$, the function $\eta$ is
well-defined and restricts to an isomorphism
\[
    \eta: D_q\rightarrow D_\delta = \{\eta\in\bC: |\eta|<\delta\},
\]
where $D_q\subset\overline{C}_q$ is an open neighborhood of $\mathfrak{s}_0$. Denote the inverse map of $\eta$ as $\rho_q$, and we define
\[
    \rho_q^{\times n} = \rho_q\times\dots\times\rho_q: (D_\delta)^n \rightarrow (D_q)^n\subset(\overline{C}_q)^n.
\]
In fact, around the point $\mathfrak{s}_0=(X, V) = (0,-1)$ with $v|_{X=0}=-\sqrt{-1}\pi$,
we can expand $v= -\log V(X)$ in powers of $\eta = X^{\frac{1}{r}}$ by \cite[Lemma 4.7]{FLT22}:
\begin{equation*}
    \begin{aligned}
        v(\eta) = -\sqrt{-1}\pi - &\sum\limits_{a_1,\dots,a_p,b\geq 0\atop {a_1\leq b \atop (a_1,\dots, a_p, b-a_1)\neq 0}}^{\infty}
        \frac{
            q_1^{a_1}\dots q_p^{a_p}e^{\sqrt{-1}\pi(\frac{kpb}{r} - a_1 + \frac{k}{r}\sum_{m=1}^{p-1}ma_{m+1})}
        }{
            a_1!\dots a_p!(b-a_1)!
        }
        \\
        & \cdot \prod_{i=1}^{\sum_{m=2}^{p}a_m + b - 1}\left(
                \frac{kpb}{r} - a_1 + \frac{k}{r}\sum_{m=1}^{p-1}ma_{m+1}-i
            \right)\eta^{pb+\sum_{m=1}^{p-1}ma_{m+1}}
    \end{aligned}
\end{equation*}
\begin{remark}
    The mirror curve $(C_q,\eta,V)$ can be derived from the mirror curve (\ref{StandardMirror})
    by the change of variables with fractional framing: $U\mapsto \eta V^{\frac{k}{r}}, V\mapsto V.$
\end{remark}

Define
\[
    \lambda := ydx = \log Y\frac{dX}{X},\quad \Phi := \lambda|_{C_q}.
\]
Then $\Phi$ is a multi-valued analytic 1-form on $C_q$. There is a regular covering map $p: \widetilde{C}_q\rightarrow C_q$ by the restriction of $\bC^2\rightarrow(C^*)^2$
given by $(x,y)\mapsto (e^{-x},e^{-y})$. Then $p^*\Phi$ is a holomorphic 1-form on $\widetilde{C}_q$.
Let
\[
    U_\ep := \{q = (q_1,\dots,q_p)\in \bC^*\times \bC^{p-1}||q_a|<\ep\}
\]
and $\widetilde{\bK}$ be a flat bundle over $U_\ep$ whose fiber over $q$ is $H_1(\widetilde{C}_q;\bC)$.
According to \cite[Section 5.9]{FLZ16}, for $a \in \{1,\dots,p\}$ there is a unique flat section $\widetilde{A}_a$ of $\widetilde{\bK}$ such that
\[
    \tau_a(q) = \frac{1}{2\pi\sqrt{-1}}\int_{\widetilde{A}_a}\Phi.
\]
Let $A_a:= p_*\widetilde{A}_a\in H_1(C_q;\bC)$. We also use $A_a$ represent the cycles $J_*(A_a)\in H_1(\overline{C}_q;\bC)$. By permuting $A_1,\dots,A_p$ if necessary, we can choose
$B_1,\dots, B_{p-1}\in H_1(\overline{C}_q;\bC)$ such that $\{A_1,\dots,A_{p-1},B_1,\dots,B_{p-1}\}$ is a symplectic basis of $H_1(\overline{C}_q;\bC)\cong \bC^{2p-2}$ with
respect to the intersection form $\cap$: $H_1(\overline{C}_q;\bC)\times H_1(\overline{C}_q;\bC)\rightarrow \bC$
\[
    A_i\cap A_j= B_i\cap B_j=0, \quad A_i\cap B_j= -B_j\cap A_i=\delta_{ij},\quad i,j\in\{1,\dots,p-1\}.
\]
We define the fundamental differential of the second kind on $\overline{C}_q$ normalized by $A_1,\dots,A_{p-1}$ as a bilinear symmetric meromorphic differential $B(p_1,p_2)$ characterized by
\begin{itemize}
    \item [$\bullet$] $B(p_1,p_2)$ is holomorphic everywhere except for a double pole along the diagonal $p_1=p_2$. Locally, if $z_1,z_2$ are local coordinates around $(p,p)\in (\overline{C}_q)^2$,
    then
    \[
        B(z_1,z_2) = \left(\frac{1}{(z_1-z_2)^2}+f(z_1,z_2)\right)dz_1dz_2,
    \]
    where $f(z_1,z_2)$ is a holomorphic function and $f(z_1,z_2)=f(z_2,z_1)$.
    \item[$\bullet$] $\int_{p_1\in A_i}B(p_1,p_2)=0$, $i=1,\dots,p-1$.
\end{itemize}

\subsection{Eynard-Orantin topological recursion and B-model topological string}
\label{sec:EO-recursion}
There are $2p$ ramification points $P_\bsi$ ($\bsi\in I_\X$)
of the coordinate function $X: C_q\rightarrow\bC$.
Around the ramification
point $P_\bsi$ with $x(P_\bsi)=x_\bsi$, $y(P_\bsi)=y_\bsi$, we can expand $x$ and $y$ as
\begin{equation*}
    \begin{aligned}
        &x = x_\bsi + \zeta_\bsi^2,
        \\
        &y = y_\bsi + \sum_{k= 1}^\infty h_k^\bsi\zeta_\bsi^k,
    \end{aligned}
\end{equation*}
where $\zeta_\bsi$ is a local coordinate around $P_\bsi$.
We expand the fundamental differential $B(p_1,p_2)$ on $(\overline{C}_q)^2$ around $(P_\bsi, P_{\bsi'})$ as
\[
    B(\zeta_\bsi, \zeta_{\bsi'}) = \left(
        \frac{\delta_{\bsi\bsi'}}{(\zeta_\bsi-\zeta_{\bsi'})^2}+\sum_{k,l\geq 0}B^{\bsi,\bsi'}_{k,l}\zeta_\bsi^k\zeta_{\bsi'}^l
    \right)d\zeta_\bsi d\zeta_{\bsi'}.
\]

We introduce the following notations:
\begin{itemize}
    \item [$\bullet$] For any $\bsi\in I_\X$, non-negative integer $d$, we define
        \[
            \theta_\bsi^d(p) := -(2d-1)!!2^{-d}\Res_{p'\rightarrow P_\bsi} B(p,p')\zeta_\bsi^{-2d-1}.
        \]
        Then $\theta_\bsi^d$ is a meromorphic 1-form on $\overline{C}_q$ with a single pole of order $2d+2$ at $P_\bsi$. In local coordinate $\zeta_\bsi$ near $P_\bsi$,
        \[
            \theta_\bsi^d = \left(
                -\frac{(2d+1)!!}{2^d\zeta_\bsi^{2d+2}} + \text{analytic part in } \zeta_\bsi
            \right)
        \]

    \item [$\bullet$] We define the formal power series $\check{R}_{\bsi'}^{\spa\bsi}(z)$ by asymptotic expansion
        \[
            \check{R}_{\bsi'}^{\spa\bsi}(z) = \frac{\sqrt{z}}{2\sqrt{\pi}}\int_{\gamma_\bsi}e^{-\frac{(x-x_\bsi)}{z}}\theta_{\bsi'}^0.
        \]
        Here $\gamma_\bsi$ is the Lefschetz thimble under the map $x$, i.e. $x(\gamma_\bsi)-x_\bsi \in \bR_{\geq 0}$.

    \item [$\bullet$] For any $\bsi, \bsi'\in I_\X$, we define
        \[
            \check{B}_{k,l}^{\bsi,\bsi'} := \frac{(2k-1)!!(2l-1)!!}{2^{k+l+1}}B_{2k,2l}^{\bsi,\bsi'}.
        \]
        By \cite[Appendix B]{Ey14}, we get
        \[
            \check{B}_{k,l}^{\bsi,\bsi'} = [z^kw^l]\left(\frac{1}{z+w}(\delta_{\bsi,\bsi'}-\sum_{\bsi''\in I_\X}
            \check{R}_{\bsi''}^{\spa\bsi}(z)\check{R}_{\bsi''}^{\spa\bsi'}(w)
            )\right)
        \]

    \item[$\bullet$] Given any $\bsi\in I_\X$, define $\hat{\theta}_\bsi^0 = \theta_\bsi^0$. For any positive integer $k$, define
        \[
            \hat{\xi}_\bsi^k := (-1)^k\left(\frac{d}{dx}\right)^{k-1}\frac{\theta_\bsi^0}{dx}, \quad \hat{\theta}_\bsi^k := d\hat{\xi}_\bsi^k,
        \]
        \[
            \theta_\bsi(z):=\sum_{k=0}^{\infty}\theta_\bsi^k z^k, \quad  \hat{\theta}_\bsi(z):=\sum_{k=0}^{\infty}\hat{\theta}_\bsi^k z^k.
        \]
\end{itemize}
We have the following proposition from \cite[Proposition 6.6]{FLZ16}.
\begin{prop}\label{prop:disk-R-matrix}
    \[
        \theta_\bsi(z) = \sum_{\bsi'\in I_\X} \check{R}_{\bsi'}^{\spa\bsi}(z)\hat{\theta}_{\bsi'}(z).
    \]
\end{prop}
Define
\[
    C(p_1,p_2) := (-\frac{\partial}{\partial x(p_1)}- \frac{\partial}{\partial x(p_2)})\left(
        \frac{\omega_{0,2}}{dx(p_1)dx(p_2)}
    \right)(p_1,p_2)dx(p_1)dx(p_2).
\]
Then $C(p_1,p_2)$ is meromorphic on $(\overline{C}_q)^2$ and is holomorphic on $(\overline{C}_q\backslash\{P_\bsi: \bsi\in I_\X\})^2$. By \cite[Lemma 6.9]{FLZ16}, we have
\begin{equation}\label{eqn:C}
    C(p_1,p_2) = \frac{1}{2}\sum_{\bsi\in I_\X}\theta_\bsi^0(p_1)\theta_\bsi^0(p_2).
\end{equation}
For a point $p$ near a ramification point $P_\bsi$, we use $\bar{p}$ to denote the point such that $X(\bar{p})=X(p)$ and $\bar{p}\neq p$ (i.e. $\zeta_\bsi(\bar{p})= -\zeta_\bsi(p)$).
Let $\omega_{g,n}$ be defined recursively by the Eynard-Orantin topological recursion \cite{EO07}:
\[
    \omega_{0,1} = 0, \quad \omega_{0,2}=B(p_1,p_2),
\]
and when $2g-2+n>0$,
\begin{equation*}
    \begin{aligned}
        \omega_{g,n}(p_1,\dots,p_n) = \sum_{\bsi\in I_{\X}}&\Res_{p\rightarrow P_{\bsi}}
        \frac{
            \int_{\xi=p}^{\bar{p}}B(p_n,\xi)
        }{
            2(\Phi(p)-\Phi(\bar{p}))
        }\Big(
            \omega_{g-1,n+1}(p,\bar{p},p_1,\dots,p_{n-1})
            \\
            &+\sum_{g_1+g_2=g}\sum_{I\sqcup J= \{1,\dots,n-1\}} \omega_{g_1,|I|+1}(p,p_I)\omega_{g_2,|J|+1}(\bar{p},p_J)
        \Big),
    \end{aligned}
\end{equation*}
We define the B-model open potentials as follows.
\begin{itemize}
    \item [(1)] (disk invariants) Define the $\textit{B-model disk amplitude}$ by
    \[
        W_{0,1}(\eta,q) = \int_{\eta'=0}^{\eta'=\eta} \rho_q^*(v(\eta')-v(0))(-\frac{d\eta'}{\eta'}).
    \]
    \item[(2)] (annulus invariants) Define the $\textit{B-model annulus potential}$ by
    \[
        W_{0,2}(\eta_1,\eta_2,q) = \int_{0}^{\eta_1}\int_{0}^{\eta_2} ((\rho_q^{\times 2})^*\omega_{0,2}-\frac{d\eta_1'd\eta_2'}{(\eta_1'-\eta_2')^2}).
    \]
    \item[(3)] For $2g-2+n>0$, define
    \[
        W_{g,n}(\eta_1,\dots,\eta_n,q) = \int_{0}^{\eta_1}\dots\int_{0}^{\eta_n} (\rho_q^{\times n})^*\omega_{g,n}.
    \]
\end{itemize}
In Section \ref{v(eta)}, we already get the expression of $v(\eta)$, so
\begin{equation*}
    \begin{aligned}
        W_{0,1}(\eta, q)= &\sum\limits_{a_1,\dots,a_p,b\geq 0\atop {a_1\leq b \atop (a_1,\dots, a_p, b-a_1)\neq 0}}^{\infty}
        \frac{
            q_1^{a_1}\dots q_p^{a_p}e^{\sqrt{-1}\pi(\frac{kpb}{r} - a_1 + \frac{k}{r}\sum_{m=1}^{p-1}ma_{m+1})}
        }{
            a_1!\dots a_p!(b-a_1)!(pb+\sum_{m=1}^{p-1}ma_{m+1})
        }
        \\
        &\cdot \prod_{i=1}^{\sum_{m=2}^{p}a_m + b - 1}\left(
                \frac{kpb}{r} - a_1 + \frac{k}{r}\sum_{m=1}^{p-1}ma_{m+1}-i
            \right)\eta^{pb+\sum_{m=1}^{p-1}ma_{m+1}}
    \end{aligned}
\end{equation*}
Let $\bp=(p_1,\dots,p_n)\in(\overline{C}_q)^n$, $\BoldEta=(\eta_1,\dots, \eta_n)\in(D_\delta)^n$.
Given a label graph $\vec{\Gamma}\in\bGa_{g,n}(\X)$ with $L^o(\Gamma)=\{l_1,\dots, l_n\}$, and
$\bullet = \bp$ or $\BoldEta$,
we assign the weight of $\vec{\Gamma}$ as
\begin{equation*}
    \begin{aligned}
       w_B^\bullet(\vec{\Gamma}) = &(-1)^{g(\vec{\Gamma})-1}\prod_{v\in V(\Gamma)}\left(
        \frac{h_1^{\bsi(v)}}{\sqrt{-2}}
    \right)^{2-2g(v)-\val(v)}
    \left\langle
        \prod_{h\in H(v)}\tau_{k(h)}
    \right\rangle_{g(v)}
    \\
    &\cdot\prod_{e\in E(\Gamma)}\check{B}_{k(h_1(e)),k(h_2(e))}^{\bsi(v_1(e)),\bsi(v_2(e))}
    \prod_{l\in\cL^1(\Gamma)}(\check{\cL}^1)_{k(l)}^{\bsi(l)}\prod_{j=1}^{n}(\check{\cL}^\bullet)_{k(l_j)}^{\bsi(l_j)}
    \end{aligned}
\end{equation*}
where
\begin{itemize}
    \item [$\bullet$](dilaton leaf)
    \[
        (\check{\cL}^1)_k^\bsi = \frac{-1}{\sqrt{-2}}[z^{k-1}]\sum_{{\bsi'}\in I_\X}h_1^{\bsi'}\check{R}_{\bsi'}^{\spa\bsi}(z)
    \]
    \item[$\bullet$](descendant leaf)
    \[
        (\check{\cL}^\bp)_{k}^{\bsi}(l_j)=\frac{1}{\sqrt{-2}}\theta_\bsi^k(p_j)
    \]
    \item[$\bullet$](open leaf)
    \[
        (\check{\cL}^{\BoldEta})_{k}^{\bsi}(l_j)=\frac{1}{\sqrt{-2}}\int_{0}^{\eta_j}\rho_q^*\theta_\bsi^k
    \]
\end{itemize}
We have the following graph sum formula:
\begin{theorem}\rm\label{eqn:B-graph-sum}(\cite{DOSS})
    For $2g-2+n>0$, one has
    \begin{equation*}
        \begin{aligned}
            \omega_{g,n}(\bp) = \sum_{\vec{\Gamma}\in\bGa_{g,n}(\X)}\frac{w_B^\bp(\vec{\Gamma})}{|\Aut(\vec{\Gamma})|},
            \\
            \int_{0}^{\eta_1}\dots\int_{0}^{\eta_n}(\rho_q^{\times n})^*\omega_{g,n} = \sum_{\vec{\Gamma}\in\bGa_{g,n}(\X)}\frac{w_B^\BoldEta(\vec{\Gamma})}{|\Aut(\vec{\Gamma})|}.
        \end{aligned}
    \end{equation*}
\end{theorem}

\section{All genus open-closed mirror symmetry}
\label{sec:ms}

\subsection{Mirror symmetry for disk invariants}
In this section, we identify the disk potential $F^{\cX,\Lrs}_{0,1}(\btau; X)$ with the non-fractional part of the B-model disk potential for $\btau=\btau_2\in\HcrX{2}$
(i.e. the $\tau_0$ component of $\btau$ is set to be 0).
\par
Let $f\in\bC[[\eta_1,\dots, \eta_n]]$. For a fixed positive integer $r\in\mathbb{N}^*$, let $a$ be a primitive $r$-th root of unity. We define the operation
\[
    \fh_{\eta_1,\dots,\eta_n}\cdot f(\eta_1,\dots,\eta_n) = \sum_{k_1,\dots,k_n=0}^{r-1}\frac{
        f(a^{k_1}\eta_1,\dots, a^{k_n}\eta_n)
    }{r^n}.
\]
This operation just preserves the terms with degree of $\eta$ divisible by $r$. In particular, we get
\begin{equation*}
    \begin{aligned}
        \mathfrak{h}_{\eta}\cdot W_{0,1}(\eta,q) = \sum_{\mu>0}\sum\limits_{a_1,\dots,a_p,b\geq 0\atop {a_1\leq b \atop pb+\sum_{m=1}^{p-1}ma_{m+1}=r\mu}}^{\infty}
        &\frac{
            \prod_{i=1}^{\sum_{m=2}^{p}a_m + b - 1}\left(
                k\mu - a_1 -i
            \right)
        }{
            a_1!\dots a_p!(b-a_1)!r\mu
        }
        \\
        &\cdot \eta^{r\mu}q_1^{a_1}\dots q_p^{a_p}(-1)^{k\mu - a_1}.
    \end{aligned}
\end{equation*}

Recall that in Section \ref{sec:relative1}, we get Equation \eqref{eqn:disk-I-function}:
\begin{equation*}
    \begin{aligned}
        F^{\cX,\Lrs}_{0,1}(\btau&(q); X) = \sum_{h\in\cI}\sum_{\mu>0, \atop\AlphaMu=h}
            \frac{1}{\mu}\frac{
                    \prod_{j=1}^{\mu k-1}(-\mu k+j)
                }
                {
                    \lfloor\frac{\mu r}{p}\rfloor!\lfloor\frac{\mu s}{p}\rfloor!
                }X^{\mu}
        \\
        &\cdot\sum_{
                \beta\in\mathbb{K}_\eff\atop\langle w(\beta)\rangle = h
            } q_1^{a_1}\dots q_p^{a_p}
        \frac{
                \prod_{m=-a_1+\lceil -w(\beta)\rceil}^{-1}(\frac{\mu r}{p}-(a_1+\wbeta+m))
            }{
                \prod_{m=0}^{a_1-1}(-\mu k+a_1 -m)
            }
        \\
        &\cdot\frac{
            \prod_{m=-a_1-a_2-\dots-a_p +\lceil w(\beta)\rceil}^{-1}(\frac{\mu s}{p}-(a_1+\dots+a_p+m)+\wbeta)}
        {a_1!a_2!\dots a_p!}
    \end{aligned}
\end{equation*}
\begin{equation*}
    \begin{aligned}
        = \sum_{h\in\cI}&\sum_{\mu>0, \atop\AlphaMu=h}
            \frac{1}{\mu}\frac{
                    \prod_{j=1}^{\mu k-1}(-\mu k+j)
                }
                {
                    \lfloor\frac{\mu r}{p}\rfloor!
                    (\mu k - \lfloor\frac{\mu r}{p}\rfloor - 1 +\delta_{\AlphaMu,0})!
                }X^{\mu}
        \\
        &\cdot\sum\limits_{
                a_1,\dots,a_p, l\geq 0,\atop w(\beta) = l + h
            } q_1^{a_1}\dots q_p^{a_p}
        \frac{
                \prod_{m=1}^{a_1+l}(\lfloor\frac{\mu r}{p}\rfloor-(a_1+l)+m)
            }{
                \prod_{m=0}^{a_1-1}(-\mu k+a_1 -m)
            }
        \\
        &\cdot\frac{
            \prod_{m=1}^{a_1+a_2+\dots+a_p-l+\delta_{h,0}-1}
            (\mu k - \lfloor\frac{\mu r}{p}\rfloor+l-(a_1+\dots+a_p)+m)}
        {a_1!a_2!\dots a_p!}
    \end{aligned}
\end{equation*}
\begin{equation*}
    \begin{aligned}
        = \sum_{h\in\cI}\sum_{\mu>0, \atop\AlphaMu=h}
            \sum\limits_{
                a_1,\dots,a_p, l\geq 0,\atop w(\beta) = l + h
            }
            &\frac{X^{\mu}}{\mu}
            \frac{
                q_1^{a_1}\dots q_p^{a_p}
            }{
                a_1!\dots a_p!
            }
            \\
            \cdot&\frac{
                    \prod_{j=1}^{\mu k-a_1-1}(-\mu k+a_1+j)
                }
                {
                    (\lfloor\frac{\mu r}{p}\rfloor-a_1-l)!
                    (\mu k - \lfloor\frac{\mu r}{p}\rfloor +l - a_1-\dots -a_p)!
                }
    \end{aligned}
\end{equation*}
\begin{equation*}
    \begin{aligned}
        = \sum_{h\in\cI}\sum_{\mu>0, \atop\AlphaMu=h}
            \sum\limits_{
                a_1,\dots,a_p, l\geq 0,\atop w(\beta) = l + h
            }
            &(-1)^{\mu k-a_1-1}\frac{X^{\mu}}{\mu}
            \\
            \cdot&\frac{
                q_1^{a_1}\dots q_p^{a_p}
            }{
                a_1!\dots a_p!
            }
            \frac{
                    \prod_{j=1}^{\sum_{i=2}^{p}a_i + \lfloor\frac{\mu r}{p}\rfloor -l-1}
                    (\mu k-a_1-j)
                }
                {
                    (\lfloor\frac{\mu r}{p}\rfloor-a_1-l)!
                }.
    \end{aligned}
\end{equation*}
Setting $b= \lfloor\frac{\mu r}{p}\rfloor -l$, we have
\begin{equation*}
    \begin{aligned}
        F^{\cX,\Lrs}_{0,1}(\btau(q); X) = \sum_{\mu>0}\sum\limits_{a_1,\dots,a_p,b\geq 0\atop {a_1\leq b \atop pb+\sum_{m=1}^{p-1}ma_{m+1}=r\mu}}^{\infty}
        &(-1)^{k\mu - a_1-1}X^{\mu}
        \frac{
            q_1^{a_1}\dots q_p^{a_p}
        }{
            a_1!\dots a_p!
        }
        \\
        \cdot&\frac{
            \prod_{j=1}^{\sum_{m=2}^{p}a_m + b - 1}\left(
                \mu k- a_1 -j
            \right)
        }{
            (b-a_1)!\mu
        }.
    \end{aligned}
\end{equation*}
Therefore, by the identification $X= \eta^r$, we have
\begin{theorem}[Disk mirror theorem for $\Lrs$]\label{thm:disk-mirror-theorem}
    \[
        F^{\cX,\Lrs}_{0,1}(\btau(q); X) = -r(\mathfrak{h}_{\eta}\cdot W_{0,1})(\eta,q).
    \]
\end{theorem}

\subsection{Mirror symmetry for annulus invariants}

\subsubsection{Identification of open leaves}
We define
\[
U(z)(\btau,X):= \sum_{\bsi'\in I_\cX} \txi^{\bsi'}
  (z,X)S(1,\phi_{\bsi'})\big|_{\mathfrak Q=1\atop \sw_i=w_i\sv}.
\]
By Proposition \ref{prop:a-potential},
\begin{align*}
F^{\cX,\Lrs}_{0,1}(\btau, X) &= [z^{-2}] \sum_{\bsi'\in I_\cX} \txi^{\bsi'}(z,X)S(1,\phi_{\bsi'})\big|_{\mathfrak Q=1\atop \sw_i=w_i\sv},\\
(X\frac{d}{dX})^2F^{\cX,\Lrs}_{0,1}(\btau,X) &=[z^0] \sum_{\bsi'\in I_\cX} \txi^{\bsi'}(z,X)S(1,\phi_{\bsi'})\big|_{\mathfrak Q=1\atop \sw_i=w_i\sv} \sv^2.
\end{align*}
Since $F^{\cX,\Lrs}_{0,1}=-r(\mathfrak h_\eta \cdot W_{0,1})(\eta,q)$ (Theorem \ref{thm:disk-mirror-theorem}), as power series in $X$,
\begin{align}
  [z^0](\sv^2 U(z)(\btau,X))&=(\frac{\eta d}{r d\eta}) (\mathfrak h_\eta\cdot \rho^*_q)(v)\nonumber \\
\left(\sv^2 U(z)(\btau,X)\right)_{\geq -1}&=\sum_{n\geq -1} (\frac{z}{\sv})^n (\frac{\eta d}{rd\eta})^{n+1}(\mathfrak h_\eta \cdot \rho_q^*)(v)\label{eqn:z-disk}\\
[z^{-2}] U(z)(\btau,X)&= -r(\mathfrak{h}_{\eta}\cdot W_{0,1})(\eta,q)\nonumber.
\end{align}
Here $()_{\geq a}$ truncates the $z$-series, keeping terms whose power is greater or equal than $a$. We usually denote $()_{\geq 0}$ by $()_+$.

Let $H_1=\bar H^{\bT}$ and $H_a=\one_{b_{a+3}}$ for $a\geq 2$. Let $\{H_{a_i}\star_t H_{b_i}\}_{1\leq i\leq p-1}$ be a basis of $H^4_{\CR,\bT}(\cX)$.
If we write the normalized canonical basis by flat basis
\[
\hat \phi_{\bsi}(\btau(q))=\sum_{i=1}^{p-1}\hat A^i_\bsi(q)H_{a_i}\star_t H_{b_i}-\sum_{a=1}^{p}\hat B^a_\bsi(q)H_{a}+\hat C_\bsi(q)\one,
\]
then
\[
\sum_{\bsi'\in I_\cX} \txi^{\bsi'} (z,X) S(\hat \phi_\bsi(\btau(q)), \phi_{\bsi'})\big|_{\mathfrak Q=1\atop \sw_i=w_i\sv}=\sum_{i=1}^{p-1}z^2\hat A^i_\bsi(q) \frac{\partial^2 U}{\partial \tau_{a_i}\partial \tau_{b_i}}-\sum_{a=1}^{p}z\hat B^a_\bsi(q) \frac{\partial U}{\partial \tau_a}+ \hat C_\bsi(q) U.
\]
Here we use the fact that $S( H_a,\phi_{\bsi'})=z\frac{\partial S(1, \phi_{\bsi'})}{\partial \tau_a}$ and $S( H_{a_i}\star_t H_{b_i},\phi_{\bsi'})=z^2\frac{\partial^2 S(1, \phi_{\bsi'})}{\partial \tau_{a_i}\partial \tau_{b_i}}$.
\\
The following proposition is proved in \cite[Section 4.4 and Proposition 6.3]{FLZ16}.
\begin{prop}
\label{prop:canonical-in-flat}
We have
$$
-\frac{\theta^0_{\bsi}}{\sqrt{-2}}=\sum_{i=1}^{p-1}\hat A^i_\bsi(q)\vert_{\sv=1} \frac{\partial^2 \Phi}{\partial \tau_{a_i}\partial \tau_{b_i}}+\sum_{a=1}^{p}\hat B^a_\bsi(q)\vert_{\sv=1} d(\frac{\frac{\partial \Phi}{\partial \tau_a}}{d x})
+ \hat C_\bsi(q)\vert_{\sv=1} d(\frac{dy}{dx}).
$$
\end{prop}
\noindent
From Proposition \ref{prop:canonical-in-flat},
\begin{align}
\label{eqn:b-canonical}
\nonumber -\frac{\theta^0_{\bsi}}{\sqrt{-2}}&=\sum_{i=1}^{p-1}\hat A^i_\bsi(q)\vert_{\sv=1} \frac{\partial^2 \Phi}{\partial \tau_{a_i}\partial \tau_{b_i}}+\sum_{a=1}^{p}\hat B^a_\bsi(q)\vert_{\sv=1} d(\frac{\frac{\partial \Phi}{\partial \tau_a}}{d x})
+ \hat C_\bsi(q)\vert_{\sv=1} d(\frac{dy}{dx})\\
&=\frac{1}{r}\left(r\sum_{i=1}^{p-1}\hat A^i_\bsi(q)\vert_{\sv=1} \frac{\partial^2 v}{\partial \tau_{a_i}\partial \tau_{b_i}}(-\frac{d\eta}{\eta})+\sum_{a=1}^{p}\hat B^a_\bsi(q)\vert_{\sv=1} d(\frac{\partial v}{\partial \tau_a})
+ \hat C_\bsi(q)\vert_{\sv=1} d(\frac{dv}{dx})\right).
\end{align}
By Equation \eqref{eqn:z-disk} and \eqref{eqn:b-canonical}, under the open-closed mirror map

\begin{align*}
&-\mathfrak h_\eta \cdot \int_0^\eta \frac{\rho^*_q \theta_\bsi^0}{\sqrt{-2}}
\\
&=\frac{1}{r} \left(r\sum_{i=1}^{p-1}\hat A^i_\bsi(q)\vert_{\sv=1} \mathfrak h_\eta \cdot\int_0^\eta\frac{\partial^2 \rho^*_q v}{\partial \tau_{a_i}\partial \tau_{b_i}}(-\frac{d\eta}{\eta})+\sum_{a=1}^{p}\hat B^a_\bsi(q)\vert_{\sv=1}(\frac{\partial (\mathfrak h_\eta \cdot \rho^*_q v)}{\partial \tau_a}) +\hat C_\bsi(q)\vert_{\sv=1} \frac{d(\mathfrak h_\eta \cdot \rho^*_q v)}{dx}\right)\\
&=\frac{1}{r}\left(r\sum_{i=1}^{p-1}\hat A^i_\bsi(q)\vert_{\sv=1} \mathfrak h_\eta \cdot\int_0^\eta\frac{\partial^2 \rho^*_q v}{\partial \tau_{a_i}\partial \tau_{b_i}}(-\frac{d\eta}{\eta})+\sum_{a=1}^{p}\hat B^a_\bsi(q)\vert_{\sv=1}(\frac{\partial (\mathfrak h_\eta \cdot \rho^*_q v)}{\partial \tau_a}) +\hat C_\bsi(q)\vert_{\sv=1}(-\frac{\eta d}{rd\eta})(\mathfrak h_\eta\cdot \rho_q^* v))\right)\\
&=-\frac{1}{r}\left(\sum_{i=1}^{p-1}[z^{-2}]\hat A^i_\bsi(q)\vert_{\sv=1} \frac{\partial^2 U}{\partial \tau_{a_i}\partial \tau_{b_i}}-\sum_{a=1}^{p}[z^{-1}]\hat B^a_\bsi(q)\vert_{\sv=1} \frac{\partial U}{\partial \tau_a}+\hat C_\bsi(q)\vert_{\sv=1}[z^0] U(z)(\btau,X)\right)_{\sv=1},
\end{align*}
Therefore
\begin{align}
  \label{eqn:xi-in-disk}
  \mathfrak h_\eta \cdot \int_0^\eta \frac{\rho^*_q \hat \theta_\bsi(z)}{\sqrt{-2}}&=\sum_{a\geq 0} z^a (X\frac{d}{dX})^a (\mathfrak h_\eta \cdot \int_0^\eta \frac{\rho^*_q \theta_\bsi^0}{\sqrt{-2}})\\
  &=\frac{1}{r}(\sum_{i=1}^{p-1}z^{2}\hat A^i_\bsi(q) \frac{\partial^2 U_{\geq -2}}{\partial \tau_{a_i}\partial \tau_{b_i}}-\sum_{a=1}^{p}z\hat B^a_\bsi(q) \frac{\partial U_{\geq -1}}{\partial \tau_a}+\hat C_\bsi(q)U_{\geq 0})\nonumber\\
  &=\frac{1}{r}\left(\sum_{\bsi'\in I_\cX} \txi^{\bsi'}(z,X)
  S( \hat\phi_{\bsi}(\btau(q)),\phi_{\bsi'})\big|_{\mathfrak
    Q=1,\sv=1}\right)_+.\nonumber
\end{align}

\subsubsection{Mirror symmetry for annulus invariants}
\begin{theorem}[Annulus mirror theorem]
  \label{thm:main-theorem-annulus}
\[
F^{\cX,\Lrs}_{0,2}(\btau; X_1,X_2)|_{\btau=\btau(q),X_i=\eta_i^r}=-r^2 \mathfrak h_{\eta_1,\eta_2}
\cdot W_{0,2}(\eta_1,\eta_2,q).
\]
\end{theorem}
\begin{proof}
By Equation \eqref{eqn:C}, we have
\begin{align*}
& \mathfrak h_{\eta_1,\eta_2}\cdot\left ((\eta_1\frac{\partial}{r\partial \eta_1} + \eta_2\frac{\partial}{r\partial \eta_2})W_{0,2}(q;\eta_1,\eta_2)\right)\\
=& \mathfrak h_{\eta_1,\eta_2}\cdot \left (\int_0^{\eta_1}\int_0^{\eta_2}(\rho_q^{\times 2})^*C\right )=\frac{1}{2}\sum_{\bsi\in I_\cX} (\mathfrak h_{\eta_1} \cdot \int_0^{\eta_1}\rho_q^* \theta_\bsi^0)(\mathfrak h_{\eta_2}\cdot \int_0^{\eta_2}\rho_q^* \theta_\bsi^0 ).
\end{align*}
By Equation \eqref{eqn:xi-in-disk},
\begin{align*}
& \mathfrak h_{\eta_1,\eta_2}\cdot \left (\int_0^{\eta_1}\int_0^{\eta_2}(\rho_q^{\times 2})^*C\right )=\frac{1}{2}\sum_{\bsi\in I_\cX} (\mathfrak h_{\eta_1} \cdot \int_0^{\eta_1}\rho_q^* \theta_\bsi^0)(\mathfrak h_{\eta_2}\cdot \int_0^{\eta_2}\rho_q^* \theta_\bsi^0 )\\
=&-\frac{1}{r^2}[z_1^{0}z_2^{0}] \sum_{\bsi,\bsi',\bsi''\in I_\cX} \txi^{\bsi''}(z_1,X_1)
   \txi^{\bsi'}(z_2,X_2) S(\hat\phi_{\bsi}(\btau),\phi_{\bsi'})  S(\hat\phi_{\bsi}(\btau),\phi_{\bsi''})\big|_{\btau=\btau(q), \fQ=1,\sv=1}\\
=&-\frac{1}{r^2} [z_1^{0}z_2^{0}](z_1+z_2) \sum_{\bsi',\bsi''\in I_\cX} V( \phi_{\bsi'},\phi_{\bsi''})\big|_{\btau=\btau(q), \fQ=1,\sv=1}\txi^{\bsi'}(z_1,X_1) \txi^{\bsi''}(z_2,X_2)\\
=& -\frac{1}{r^2}(X_1\frac{\partial}{\partial X_1} + X_2\frac{\partial}{\partial X_2})F^{\cX,\Lrs}_{0,2}(\btau; X_1,X_2).
\end{align*}

\end{proof}

\subsection{Mirror symmetry for $2g-2+n>0$}
\begin{theorem}[Mirror theorem for $2g-2+n>0$]
  \label{thm:main-theorem-stable}
$$F^{\cX,\Lrs}_{g,n}(\btau; X_1,\dots,X_n)|_{\btau=\btau(q),X_i=\eta_i^r}=(-1)^{g-1} r^n (\mathfrak h \cdot
W_{g,n})(\eta_1,\dots,\eta_n,q).$$
\end{theorem}
\begin{proof}
We prove the theorem by matching the graph sum formulas in Theorem \ref{eqn:B-graph-sum} and Corollary \ref{eqn:graph-open}. By \cite[Theorem 7.1]{FLZ16} and \cite[Section 4.4 and Lemma 4.3]{FLZ16}, we have
\begin{equation}
\label{eqn:R-identification}
R_{\bsi'}^{\spa \bsi}(z)|_{\btau=\btau(q) , \mathfrak Q=1,\sv=1}=\check R_{\bsi'}^{\spa \bsi} (-z).
\end{equation}
and
\begin{equation}
\frac{h_1^\bsi}{\sqrt{-2}}=\sqrt{\frac{1}{\Delta^\bsi(\btau)\vert_{\sv=1}}}.
\end{equation}
Then by Theorem \ref{eqn:B-graph-sum} and Corollary \ref{eqn:graph-open}, the weights in the graph sum match except for the open leaves. On the other hand, Proposition \ref{prop:disk-R-matrix} and Eq. \eqref{eqn:xi-in-disk} gives the identification of the open leaves. So we have
\[w^X_A(\vGa)|_{\sv=1}=\widetilde{w}^X_A(\vGa)|_{\fQ=1,\sv=1}=(-1)^{g-1}r^n (\mathfrak h\cdot w_B^\eta(\vGa)),
\]
which implies the theorem.
\end{proof}

\nocite{*}
\bibliographystyle{plain}
\bibliography{ref}

\end{document}